\documentclass[12pt,a4paper]{article}

\usepackage{fontspec}
\usepackage{amsmath,amssymb,amsthm,mathtools,bm}
\usepackage[a4paper,margin=2.6cm]{geometry}
\usepackage{xcolor}
\usepackage{graphicx}
\usepackage{enumitem}
\usepackage{needspace}
\usepackage[colorlinks=true,linkcolor=blue!55!black,
 citecolor=blue!55!black,urlcolor=blue!55!black]{hyperref}
\numberwithin{equation}{section}
\newtheoremstyle{pdeplain}{6pt}{6pt}{\itshape}{}{\bfseries}{.}{0.5em}{}
\theoremstyle{pdeplain}
\newtheorem{theorem}{Theorem}[section]
\newtheorem{proposition}[theorem]{Proposition}
\newtheorem{lemma}[theorem]{Lemma}
\newtheorem{corollary}[theorem]{Corollary}
\theoremstyle{remark}
\newtheorem*{remark}{Remark}
\newcommand{\R}{\mathbb R}
\newcommand{\dd}{\,\mathrm d}
\newcommand{\eps}{\varepsilon}
\newcommand{\bzero}{\mathbf0}

\DeclareMathOperator{\supp}{supp}
\allowdisplaybreaks[1]

\hypersetup{pdftitle={Singular Sets and Doubling Properties in Parabolic Homogenization},pdfauthor={Xiujin Chen}}
\title{\vspace{-1em}\bfseries Singular Sets and Doubling Properties\\
in Parabolic Homogenization}
\author{Xiujin Chen\\[0.5em]
{\small School of Mathematical Sciences, Zhejiang University, Hangzhou 310000, China}\\[0.2em]
{\small \href{mailto:xjchen1998@foxmail.com}{\texttt{xjchen1998@foxmail.com}}}}
\date{}
\begin{document}
\maketitle
\vspace{-2em}
\begin{abstract}
We study the singular sets of real-valued solutions to periodic parabolic
equations with rapidly oscillating coefficients. Under a nondegeneracy
condition on the corrector matrix, we obtain a local parabolic
$n$-dimensional Hausdorff measure bound and a codimension-two Minkowski
estimate. The Minkowski estimate holds at every scale, including scales
below the oscillation scale $\eps$, with a constant independent of $\eps$.
The bounds depend on a coarse-scale ratio of space-time $L^2$ mass to
terminal-slice $L^2$ mass. Independently of corrector nondegeneracy, we
prove that this ratio controls doubling at all smaller scales uniformly
in $\eps$. The singular-set argument combines summable errors in
truncated Gaussian approximations, comparison of caloric polynomials
at different time centers, and a stopping cover with microscopic
singular-set estimates.
\end{abstract}
\begingroup
\small
\tableofcontents
\endgroup

\section{Introduction}\label{sec:main-results}

In this paper we study the singular sets of real-valued solutions to
periodic parabolic equations with rapidly oscillating coefficients,
\begin{equation}\label{eq:homogenization-equation}
\partial_tu_\eps-
\operatorname{div}\!\bigl(A_\eps\nabla u_\eps\bigr)=0,
\qquad
A_\eps(x,t):=A\!\left(\frac{x}{\eps},\frac{t}{\eps^2}\right),
\end{equation}

where $\eps>0$. Our main result is a local bound for the parabolic
$n$-dimensional Hausdorff measure of the singular set, uniform in
$\eps$, under a coarse-scale growth condition and a nondegeneracy
assumption on the correctors. The proof also gives a uniform parabolic
Minkowski estimate. We establish the uniform doubling estimates needed
to derive these bounds from the coarse-scale data.

We impose the following conditions on the coefficient matrix $A=A(y,s)$.
\begin{enumerate}[label=\textup{(A\arabic*)},leftmargin=3em]
\item \textbf{Uniform ellipticity.} The matrix $A(y,s)$ is real and symmetric,
and there exist $0<\lambda\le\Lambda<\infty$ such that
\[
\lambda|\xi|^2\le\xi^TA(y,s)\xi\le\Lambda|\xi|^2,
\qquad (y,s)\in\R^{n+1},\quad\xi\in\R^n.
\]
\item \textbf{Parabolic Lipschitz.} There exists $L_A<\infty$ such that
\begin{equation}\label{eq:space-time-lipschitz}
|A(y,s)-A(y',s')|
\le L_A\bigl(|y-y'|+|s-s'|^{1/2}\bigr).
\end{equation}
\item \textbf{Periodicity.} For a fixed full-rank spatial lattice
$\mathcal L\subset\R^n$ and a time period $T>0$,
\[
A(y+\ell,s+kT)=A(y,s),
\qquad\ell\in\mathcal L,\quad k\in\mathbb Z.
\]
\item \textbf{Normalization.} The homogenized matrix $\widehat A$ satisfies
\begin{equation}\label{eq:homogenized-normalization}
\frac{\widehat A+\widehat A^T}{2}=I.
\end{equation}
\end{enumerate}
A fixed spatial linear transformation gives \eqref{eq:homogenized-normalization}
and maps the spatial periodicity lattice to another full-rank lattice
allowed by \textup{(A3)}. Balls and all quantities below are defined in
the normalized coordinates. Constants may depend on $n,\lambda,\Lambda,L_A$,
the geometry of the fixed lattice $\mathcal L$, the time period $T$,
and the fixed cutoff; these dependencies are usually omitted.
Unless otherwise stated, constants are independent of $\eps$.

To study the singular set, we use the periodic correctors.
Let $Y$ be a fundamental cell of $\mathcal L\times T\mathbb Z$.
For $1\le k\le n$, let $\chi_k$ be the $\mathcal L\times T\mathbb Z$-periodic
corrector determined by
\begin{equation}\label{eq:parabolic-correctors}
\partial_s\chi_k-\operatorname{div}_y
\bigl(A(y,s)(e_k+\nabla_y\chi_k)\bigr)=0,
\qquad \int_Y\chi_k(y,s)\dd y\dd s=0.
\end{equation}
Set $\chi=(\chi_1,\ldots,\chi_n)$ and
$\nabla_y\chi=(\partial_{y_i}\chi_j)_{i,j=1}^n$, and define
\begin{equation}\label{eq:corrector-matrix}
\mathcal J_A(y,s):=I+\nabla_y\chi(y,s).
\end{equation}

\Needspace{6\baselineskip}
For the singular-set estimate we impose the additional condition
\begin{enumerate}[label=\textup{(A5)},leftmargin=3em]
\item \textbf{Nondegeneracy.} There exists $\mu>0$ such that
\begin{equation}\label{eq:corrector-nondegeneracy}
|\mathcal J_A(y,s)\xi|\ge\mu|\xi|,
\qquad (y,s)\in\R^{n+1},\quad \xi\in\R^n.
\end{equation}
\end{enumerate}

Constants in estimates using \textup{(A5)} may also depend on $\mu$.

For $\mathbf z=(x,t)$ and $r>0$, write
\[
\begin{aligned}
Q(\mathbf z,r)&:=B(x,r)\times(t-r^2,t+r^2),\\
Q^-(\mathbf z,r)&:=B(x,r)\times(t-r^2,t].
\end{aligned}
\]
For the singular-set estimate, we assume the coarse-scale bound
\begin{equation}\label{eq:rough}
\sup_{\sigma\in[-4,4]}
\frac{\displaystyle\int_{B_6\times[-36,\sigma]}u_\eps^2\dd y\dd s}
{\displaystyle\int_{B_{1/2}}u_\eps^2(y,\sigma)\dd y}
\le\Theta_0<\infty.
\end{equation}

\Needspace{9\baselineskip}
Let $\mathcal H_P^n$ denote the $n$-dimensional parabolic Hausdorff measure
and $S(u_\eps)$ denote the singular set
$\{u_\eps=\nabla_xu_\eps=0\}$.
The spatial critical set $\{\nabla_xu=0\}$ need not have parabolic
codimension two, even for the heat equation. Indeed, the caloric polynomial
$u(x,t)=x_1^2+2t$ has
\[
\{\nabla_xu=0\}=\{x_1=0\}\times\mathbb R,
\]
of parabolic dimension $n+1$. This is why we consider the singular set,
which also imposes $u=0$.

Under \textup{(A5)}, the singular set satisfies the following uniform estimate.

\begin{theorem}\label{thm:main-measure}
Let $u_\eps$ solve \eqref{eq:homogenization-equation} in $Q(\mathbf0,6)$
under \textup{(A1)--(A5)} and satisfy \eqref{eq:rough}. Then
\begin{equation}\label{eq:main-local-measure}
\mathcal H_P^n\!\left(
S(u_\eps)\cap Q(\mathbf0,1/2)
\right)\le C(\Theta_0),
\end{equation}
where $C$ is independent of $\eps$.
\end{theorem}

\begin{remark}
The parabolic dimension bound $n$ is sharp when $n\ge2$: the caloric polynomial
$x_1^2-x_2^2$ has singular set
$\{x_1=x_2=0\}\times\mathbb R$, of parabolic dimension $n$.
\end{remark}

The proof yields the following stronger parabolic Minkowski estimate,
valid at all scales, including $r<\eps$.

\begin{theorem}\label{thm:main-minkowski}
Let $u_\eps$ solve \eqref{eq:homogenization-equation} in $Q(\mathbf0,6)$
under \textup{(A1)--(A5)} and satisfy \eqref{eq:rough}.
Set $E_\eps=S(u_\eps)\cap Q(\mathbf0,1/2)$ and
$Q(E_\eps,r)=\bigcup_{\mathbf z\in E_\eps}Q(\mathbf z,r)$.
Then
\begin{equation}\label{eq:main-minkowski}
\mathcal H_P^{n+2}\!\left(Q(E_\eps,r)\right)
\le C(\Theta_0)r^2,\qquad 0<r\le1,
\end{equation}
where $C$ is independent of $\eps$.
\end{theorem}

For elliptic equations, Han \cite{Han94} studied the
dimension and structure of singular sets. Han--Hardt--Lin \cite{HHL}
obtained Hausdorff measure bounds for singular sets under smooth
coefficient assumptions, while
Hardt--Hoffmann-Ostenhof--Hoffmann-Ostenhof--Nadirashvili \cite{HHHON}
established local finiteness for critical sets. Quantitative estimates
require control of the growth of the solution, usually expressed through
frequency or doubling. Cheeger--Naber--Valtorta \cite{CNV} used
quantitative stratification to obtain Minkowski estimates with an
arbitrarily small loss in the codimension-two exponent.

A decisive advance was made by Naber--Valtorta \cite{NV}, who obtained
codimension-two Minkowski estimates for singular and critical sets under
Lipschitz assumptions on the
leading coefficients. Their proof combines quantitative control of
homogeneous approximations on intervals of pinched frequency with a
covering argument and induction on the frequency bound. In particular,
it avoids the higher coefficient regularity required by earlier
approaches to Hausdorff estimates. Huang--Jiang \cite{HJ} subsequently
obtained Minkowski estimates for singular and critical sets with
H\"older leading coefficients under a
uniform doubling assumption, using almost monotonicity of the doubling
index and quantitative uniqueness of tangent maps.

These results do not directly give uniform estimates in homogenization,
since the regularity bounds of the oscillating coefficients deteriorate
as $\eps\to0$. Lin--Shen \cite{LS} proved uniform Minkowski and Hausdorff
bounds for elliptic critical sets under corrector nondegeneracy. Their
argument develops the covering method of \cite{NV}: successive harmonic
approximations yield a summable estimate for the turning of the
approximate tangent polynomial, and a stopping cover separates scales
where the doubling index drops from the microscopic regime. This is the closest antecedent of our method.

For parabolic equations, space-time estimates and estimates on individual
time slices are distinct problems. Han--Lin \cite{HL94} proved Euclidean
Hausdorff estimates for space-time nodal sets; time-slice estimates in
analytic settings appear in \cite{Lin91,Kuk95,ABG}.
Huang--Jiang \cite{HJ24} obtained time-slice nodal measure estimates under
parabolic Lipschitz assumptions. More directly related to our result,
Hallgren--Koirala--Ma \cite{HKM} established parabolic Minkowski estimates
and rectifiability for nodal and singular sets, with the respective
parabolic dimension bounds $n+1$ and $n$. Their singular-set estimate
supplies the microscopic input in our
proof. For $A(x/\eps,t/\eps^2)$, the parabolic Lipschitz bound becomes
$L_A/\eps$, so their estimate does not directly give a constant
independent of $\eps$.

The broader nodal-set theory provides the background for these
codimension-two questions. Local elliptic nodal estimates were developed
by Hardt--Simon \cite{HS} and Lin \cite{Lin91}. For eigenfunctions,
Donnelly--Fefferman \cite{DF} proved sharp measure bounds for analytic
metrics, while Logunov \cite{LogUpper,LogLower} obtained polynomial upper
bounds and the sharp lower bound for smooth metrics. Our concern here
is the singular subset of the nodal set and the uniformity of its
measure bound under rapidly oscillating coefficients.

Condition \textup{(A5)} allows gradient lower bounds for a caloric
approximation to be transferred to $u_\eps$. In the first-order
approximation, $\nabla u_\eps$ is compared with
$\mathcal J_A(x/\eps,t/\eps^2)\nabla v$, where $v$ is caloric.
By \textup{(A5)}, the latter has magnitude at least $\mu|\nabla v|$.
When the approximation error is sufficiently small relative to this
lower bound, $\nabla u_\eps$ cannot vanish. This is used to exclude
singular points in the low-degree regime and to localize them near
cylindrical models in the covering argument. The same role of corrector
nondegeneracy appears in the elliptic analysis of Lin--Shen \cite{LS}.

Building on the critical-set example of Lin--Shen
\cite[Remark~2.9]{LS}, and using the reflection-symmetric construction
of Briane--Milton--Nesi \cite[Section~5]{BMN}, we obtain a more specific
example in spatial dimension three showing that \textup{(A5)} cannot
in general be omitted from the uniform parabolic Minkowski estimate. This adaptation was found
with the assistance of ChatGPT. The construction arranges that a
rank-two lattice of periodic copies of a critical point lies on the
same level set; subtracting this level produces singular points.
Appendix~\ref{sec:counterexample} gives the coefficient, the geometry,
and the tube-volume calculation. Whether \textup{(A5)} can be removed
from the Hausdorff estimate in Theorem~\ref{thm:main-measure} remains
open; the analogous question for elliptic critical sets was raised in
\cite[p.~3148]{LS}.

The growth control needed for the singular-set estimates follows from
the uniform doubling theorem below and its multicenter consequence.

We use the terminal-slice local quantity
\begin{equation}\label{eq:theta-local}
\Theta(u,\mathbf z,r)
:=\frac{\displaystyle\int_{Q^-(\mathbf z,4r)}
|u(y,s)|^2\dd y\dd s}
{r^2\displaystyle\int_{B(x_0,r)}|u(y,t_0)|^2\dd y},
\qquad \mathbf z=(x_0,t_0).
\end{equation}
The quotient is $+\infty$
when its denominator vanishes.
The quantity in \eqref{eq:theta-local} is scale invariant:
\begin{equation}\label{eq:theta-scale-invariance}
\Theta(cT_{\mathbf z,r}u,\mathbf0,\rho)
=\Theta(u,\mathbf z,r\rho),\qquad c\ne0.
\end{equation}
Here $T_{\mathbf z,r}$ denotes the parabolic rescaling in
\eqref{eq:parabolic-rescaling}.

\begin{theorem}\label{thm:center}
Let $u_\eps$ solve \eqref{eq:homogenization-equation} in $Q^-(\mathbf0,4)$
under \textup{(A1)--(A4)} and suppose
\begin{equation}\label{eq:coarse-center}
\Theta(u_\eps,\mathbf0,1)\le\Theta_0<\infty.
\end{equation}
There exists $C=C(\Theta_0)$, independent of $\eps$, such that
\begin{flalign}
&\textup{(1)}\quad
\Theta(u_\eps,\mathbf0,r)\le C,
\qquad 0<r\le1, &&
\label{eq:center-conclusion}\\
&\textup{(2)}\quad
\int_{B_{2r}}u_\eps^2(x,0)\dd x
\le C\int_{B_r}u_\eps^2(x,0)\dd x,
\qquad 0<r\le\frac12. &&
\label{eq:spatial-conclusion}
\end{flalign}
\end{theorem}

Applying Theorem~\ref{thm:center} along a finite chain of overlapping
spatial balls at each fixed time gives the following consequence of
\eqref{eq:rough}.

\begin{corollary}\label{thm:allcenters}
Let $u_\eps$ solve \eqref{eq:homogenization-equation} in $Q(\mathbf0,6)$
under \textup{(A1)--(A4)} and satisfy \eqref{eq:rough}.
There exists $\Theta_1=\Theta_1(\Theta_0)<\infty$, independent of $\eps$,
such that
\begin{equation}\label{eq:allcenters}
\sup_{\substack{\mathbf z\in Q(\mathbf0,2)\\0<r\le1}}
\Theta(u_\eps,\mathbf z,r)\le\Theta_1.
\end{equation}
\end{corollary}

Doubling inequalities quantify unique continuation and provide growth
bounds for measure estimates of nodal, singular, and critical sets.
For elliptic equations,
Garofalo--Lin \cite{GL86,GL87} developed frequency methods that yield
such inequalities; the nodal estimates in \cite{HS,Lin91} and the
critical-set estimates in \cite{NV} show how growth enters measure
bounds. In the parabolic setting, Poon \cite{Poon} introduced a Gaussian
frequency approach to unique continuation, and
Escauriaza--Fern\'andez--Vessella \cite{EFV} proved spatial and space-time
doubling estimates. Their estimates give the local input for
Theorem~\ref{thm:center}.

In periodic elliptic homogenization, Lin--Shen \cite{LS19} established
uniform doubling and nodal measure bounds. Kenig--Zhu--Zhuge \cite{KZZ}
obtained explicit dependence on the initial doubling bound by combining
estimates from different scale regimes.
For parabolic periodic homogenization, Zhang \cite{Zhang21} proved an
approximate two-sphere one-cylinder inequality under H\"older continuity
of the coefficients, using fundamental-solution expansions and
interpolation. The estimate contains an additive homogenization error
and does not directly yield uniform doubling down to the microscopic
scale. Zhang \cite{Zhang25} subsequently extended approximate
propagation of smallness to equations with suitable lower-order terms.
Corollary~\ref{thm:allcenters} supplies the bound $\Theta_1$
used in the singular-set argument.

The estimates uniform in $\eps$ rely on three ingredients adapted to the
parabolic homogenization problem. First, we propagate bounds for the
truncated Gaussian doubling index and the local growth ratio $\Theta$
simultaneously. This couples the noninteger-threshold argument of
\cite{LS19} with the mass control needed for compactness on expanding
cylinders. Second, we construct full-space caloric approximations whose
normalized homogenization and cutoff errors, of size
$\eps/r+e^{-c/r^2}$, remain summable over the scales above a fixed multiple
of $\eps$. These estimates extend the turning argument of \cite{LS} to
localized Gaussian quantities. Third, we compare the approximate caloric
polynomials at centers with different time coordinates and use the
resulting spatial and temporal approximate invariance to localize
singular points near cylindrical models. The parabolic geometry and
localization in \cite{HKM} provide a further basis for this analysis;
here the comparisons are made uniform through quantitative
homogenization. Together with the microscopic estimates of \cite{HKM},
these ingredients give the Minkowski bound at every scale, including
$r<\eps$. The joint propagation also yields the independent uniform
doubling theorem, which does not require \textup{(A5)}.

The analytic tools come from quantitative homogenization. Classical
accounts are \cite{BLP,JKO}; Avellaneda--Lin \cite{AL87} developed the
compactness method for uniform elliptic regularity, and further elliptic
theory is presented in \cite{Shen18}. For time-dependent periodic
parabolic equations, Geng--Shen established uniform regularity
\cite{GS}, convergence rates \cite{GS17}, and expansions of fundamental
solutions and their gradients \cite{GS20}. We use these estimates to
compare solutions with caloric functions above the microscopic scale
and to control the errors in successive approximations.

The estimates of \cite{GS,GS20}, stated for unit periods, apply here
by writing $\mathcal L=B\mathbb Z^n$ and making the fixed change
$(y,s)=(Bz,T\tau)$. The transformed coefficient is
$T B^{-1}A(Bz,T\tau)B^{-T}$; its ellipticity and regularity bounds
are uniform when the stated coefficient and period data are fixed.

Section~\ref{sec:ls} establishes spatial doubling, compactness, and threshold propagation.
Section~\ref{sec:uniform-doubling} proves Theorem~\ref{thm:center} and
Corollary~\ref{thm:allcenters}. Sections~\ref{sec:drop-nondegeneracy}--\ref{sec:multicenter-geometry}
establish the analytic and geometric estimates for the singular set.
Section~\ref{sec:measure} constructs the stopping-scale cover, iterates it
to any prescribed scale above the microscopic scale, and proves
Theorems~\ref{thm:main-measure} and \ref{thm:main-minkowski}.
Appendix~\ref{sec:gaussian-appendix} records the Gaussian doubling identities,
and Appendix~\ref{sec:counterexample} gives the counterexample without
\textup{(A5)}.

\subsection{Notation and definitions}\label{sec:notation}

We write $\mathbb N_0=\{0,1,2,\ldots\}$. A space--time point is denoted by
$\mathbf z=(x,t)\in\R^n\times\R$. For $r>0$, define the parabolic
cylinder and the backward parabolic cylinder by
\begin{equation}\label{eq:parabolic-cylinders}
\begin{alignedat}{1}
Q(\mathbf z,r)&:=B(x,r)\times(t-r^2,t+r^2),\\
Q^-(\mathbf z,r)&:=B(x,r)\times(t-r^2,t].
\end{alignedat}
\end{equation}

We use the parabolic distance and tubular neighborhoods
\[
 d_P((x,t),(y,s))=\max\{|x-y|,|t-s|^{1/2}\},\qquad
 Q(E,r)=\bigcup_{\mathbf z\in E}Q(\mathbf z,r).
\]
The associated Hausdorff measure is
\begin{equation}\label{eq:parabolic-hausdorff}
\mathcal H_P^k(E):=\lim_{\eta\downarrow0}
\inf\left\{\sum_i(\operatorname{diam}_{d_P}U_i)^k:
 E\subset\bigcup_iU_i,\ \operatorname{diam}_{d_P}U_i<\eta\right\}.
\end{equation}

For $x_0\in\R^n$ and $\tau>0$, define the Gaussian weight
\begin{equation}\label{eq:gaussian-weight}
G_{x_0,\tau}(y):=(4\pi\tau)^{-n/2}
\exp\!\left(-\frac{|y-x_0|^2}{4\tau}\right).
\end{equation}
The corresponding Gaussian measure is
\begin{equation}\label{eq:backward-gaussian-measure}
\dd\nu_{x_0,\tau}(y)=G_{x_0,\tau}(y)\dd y.
\end{equation}
For $\mathbf z=(x_0,t_0)$ and $r>0$, define the full Gaussian mass,
energy and frequency by
\begin{align}
H(u;\mathbf z,r)
&:=\int_{\R^n}|u(y,t_0-r^2)|^2\dd\nu_{x_0,r^2}(y),
\label{eq:full-height}\\
E(u;\mathbf z,r)
&:=2r^2\int_{\R^n}|\nabla_yu(y,t_0-r^2)|^2\dd\nu_{x_0,r^2}(y),
\label{eq:full-energy}\\
N(u;\mathbf z,r)&:=\frac{E(u;\mathbf z,r)}{H(u;\mathbf z,r)}.
\label{eq:full-frequency}
\end{align}
The corresponding doubling index is
\begin{equation}\label{eq:full-doubling-index}
D(u;\mathbf z,r):=\log_4\frac{H(u;\mathbf z,r)}{H(u;\mathbf z,r/2)}.
\end{equation}
These quantities are used when the relevant integrals are finite and
the denominators are positive.

Fix $\vartheta\in C_c^\infty(B_{3/2})$ such that
\begin{equation}\label{eq:spatial-cutoff}
0\le\vartheta\le1,\qquad
\vartheta\equiv1\ \text{in }B_1,\qquad
\operatorname{supp}\vartheta\subset B_{3/2}.
\end{equation}
The unscaled arguments below use physical cutoff radius one.
Thus $\vartheta(\,\cdot-x_0)$ equals one on $B(x_0,1)$, is supported
in $B(x_0,3/2)$, and is independent of the Gaussian scale $r$.
Define the truncated Gaussian mass by
\begin{equation}\label{eq:truncated-height}
H_1(u;\mathbf z,r)
:=\int_{\R^n}|u(y,t_0-r^2)|^2\vartheta^2(y-x_0)
\dd\nu_{x_0,r^2}(y).
\end{equation}
The truncated energy and frequency are
\begin{align}
E_1(u;\mathbf z,r)
&:=2r^2\int_{\R^n}
\left|\nabla_y\!\left[u(y,t_0-r^2)\vartheta(y-x_0)\right]\right|^2
\dd\nu_{x_0,r^2}(y),
\label{eq:truncated-energy}\\
N_1(u;\mathbf z,r)&:=\frac{E_1(u;\mathbf z,r)}{H_1(u;\mathbf z,r)}.
\label{eq:truncated-frequency}
\end{align}
The truncated doubling index is
\begin{equation}\label{eq:truncated-doubling-index}
D_1(u;\mathbf z,r)
:=\log_4\frac{H_1(u;\mathbf z,r)}{H_1(u;\mathbf z,r/2)}.
\end{equation}
Both masses use the same unit cutoff.

For any fixed physical radius $\rho>0$, define
\begin{equation}\label{eq:rescaled-cutoff-height}
\begin{aligned}
H_\rho(u;\mathbf z,r)
&=\int u^2(y,t_0-r^2)\vartheta^2\!\left(\frac{y-x_0}{\rho}\right)
   \dd\nu_{x_0,r^2}(y),\\
D_\rho(u;\mathbf z,r)
&=\log_4\frac{H_\rho(u;\mathbf z,r)}{H_\rho(u;\mathbf z,r/2)}.
\end{aligned}
\end{equation}
Define $E_\rho,N_\rho$ using the same
cutoff in \eqref{eq:truncated-energy}--\eqref{eq:truncated-frequency}.

\Needspace{12\baselineskip}
For $\mathbf z=(x_0,t_0)$ and $r>0$, we use the parabolic rescaling
\begin{equation}\label{eq:parabolic-rescaling}
T_{\mathbf z,r}u(y,s)
:=u(x_0+ry,t_0+r^2s).
\end{equation}

For $w=c^{-1}T_{\mathbf z_0,R}u$, where $c,R>0$,
\begin{equation}\label{eq:truncated-scaling}
\begin{aligned}
H_{\rho/R}(w;\mathbf0,q)&=c^{-2}H_\rho(u;\mathbf z_0,Rq),\\
D_{\rho/R}(w;\mathbf0,q)&=D_\rho(u;\mathbf z_0,Rq).
\end{aligned}
\end{equation}
The corresponding identities hold for $E_\rho,N_\rho$.
We omit the center when it is $\mathbf0$.

\section{Compactness and threshold propagation}\label{sec:ls}

The following spatial growth estimate supplies the polynomial bounds needed for compactness.

\begin{lemma}\label{lem:spatial}
Let $u$ be a weak solution of
\begin{equation}\label{eq:general-parabolic-equation}
\partial_tu-\operatorname{div}\!\bigl(A(x,t)\nabla u\bigr)=0,
\end{equation}
with $A$ satisfying \textup{(A1)}.
Whenever $Q^-(\mathbf0,4r)$ is contained in its domain,
\begin{equation}\label{eq:spatial-single-step}
\int_{B_{2r}}u^2(x,0)\dd x
\le C\Theta(u,\mathbf0,r)\int_{B_r}u^2(x,0)\dd x.
\end{equation}
Here $C$ depends only on $n,\lambda,\Lambda$.
If $\Theta(u,\mathbf0,s)\le\Theta_*$ for $r\le s\le R$, then
\begin{equation}\label{eq:polynomial-spatial-growth}
\int_{B_s}u^2(x,0)\dd x
\le C_0\left(\frac{s}{r}\right)^{\beta_0}\int_{B_r}u^2(x,0)\dd x,
\qquad r\le s\le2R,
\end{equation}
where $C_0=C_0(\Theta_*)$ and $\beta_0=\beta_0(\Theta_*)$.
\end{lemma}
\begin{proof}
The interior estimate gives
\[
\sup_{B_{2r}}|u(\cdot,0)|^2
\le Cr^{-n-2}\int_{Q^-(\mathbf0,4r)}u^2,
\]
which proves \eqref{eq:spatial-single-step}.
For $r\le s\le2R$, choose $J=\max\{0,\lfloor\log_2(s/r)\rfloor\}$.
Set $\beta_0=\max\{n,\log_2(1+C\Theta_*)\}$ and $C_0=2^{\beta_0}$.
Iterating \eqref{eq:spatial-single-step} at the scales
$s/2,\ldots,s/2^J$ and finally at $r$ yields
\begin{align*}
\int_{B_s}u^2(x,0)\dd x&\le(1+C\Theta_*)^{J+1}\int_{B_r}u^2(x,0)\dd x\\
&\le2^{\beta_0}\left(\frac{s}{r}\right)^{\beta_0}\int_{B_r}u^2(x,0)\dd x.\qedhere
\end{align*}
\end{proof}

\Needspace{7\baselineskip}
The next lemma gives compactness of normalized solutions and convergence
of their truncated Gaussian masses. It will be used to pass doubling
bounds to caloric limits.

\begin{lemma}\label{lem:compactness}
Fix $\Theta_*\ge1$. Let $\kappa_j\to\infty$ and $\eps_j\to0$.
Suppose $A_j$ satisfy \textup{(A1)--(A4)} with the same lattice
$\mathcal L$, period $T$ and structural bounds, and $w_j$ solves
\[
\partial_tw_j-\operatorname{div}\!\left[
A_j(y/\eps_j,t/\eps_j^2)\nabla w_j\right]=0
\quad\text{in }Q^-(\mathbf0,4\kappa_j).
\]
Assume $\int_{B_1}w_j^2(y,0)\dd y=1$ and
\begin{equation}\label{eq:normalized-theta}
\Theta(w_j,\mathbf0,2^k)\le\Theta_*,
\qquad k\in\mathbb N_0,\quad 2^k\le\kappa_j.
\end{equation}
After extraction, there is a nonzero caloric polynomial $P$ with
$\int_{B_1}P^2(y,0)\dd y=1$ such that
\[
w_j\longrightarrow P
\qquad\text{locally uniformly on }\R^n\times(-\infty,0].
\]
Set $\mathcal J_j^{\eps_j}(y,t):=\mathcal J_{A_j}(y/\eps_j,t/\eps_j^2)$.
Then
\begin{equation}\label{eq:corrected-gradient-convergence}
\nabla w_j-\mathcal J_j^{\eps_j}\nabla P\longrightarrow0
\qquad\text{locally uniformly on }\R^n\times(-\infty,0].
\end{equation}
For every $0<a<b<\infty$ and every $q_j\to q\in[a,b]$,
\begin{align}
H_{\kappa_j}(w_j;\mathbf0,q_j)&\longrightarrow H(P;\mathbf0,q)>0,
\label{eq:height-convergence}\\
D_{\kappa_j}(w_j;\mathbf0,q_j)&\longrightarrow D(P;\mathbf0,q),
\label{eq:D-convergence}\\
\Theta(w_j,\mathbf0,q_j)&\longrightarrow\Theta(P,\mathbf0,q).
\label{eq:Theta-convergence}
\end{align}
\end{lemma}

\begin{proof}
Put $\rho_j=2^{\lfloor\log_2\kappa_j\rfloor}$, so that $\rho_j\le\kappa_j<2\rho_j$. Set $\beta_0=\beta_0(\Theta_*)\ge n$.
At dyadic $1\le R\le\kappa_j$, \eqref{eq:spatial-single-step} gives
\[
\int_{B_{2R}}w_j^2(y,0)\dd y\le C\Theta_*\int_{B_R}w_j^2(y,0)\dd y.
\]
Starting from $\int_{B_1}w_j^2(y,0)\dd y=1$, we obtain
\begin{equation}\label{eq:terminal-control}
\begin{aligned}
\int_{B_R}w_j^2(y,0)\dd y&\le CR^{\beta_0},\\
\int_{Q^-(\mathbf0,4R)}w_j^2&\le CR^{\beta_0+2},
\qquad 1\le R\le\rho_j.
\end{aligned}
\end{equation}
For the spacetime bound in \eqref{eq:terminal-control}, first use \eqref{eq:normalized-theta} at dyadic radii,
and then increase an arbitrary $R\le\rho_j$ to the next dyadic radius,
which is at most $2R$ and at most $\rho_j$.
Interior estimates yield
\begin{equation}\label{eq:polynomial-cylinder-bound}
\sup_{Q^-(\mathbf0,3R)}|w_j|^2\le CR^{\beta_0-n},
\qquad 1\le R\le\rho_j.
\end{equation}
After extraction, $\widehat A_j\to A_\infty$, where
$(A_\infty+A_\infty^T)/2=I$ by \textup{(A4)}.
Local energy and H\"older estimates and \cite[Theorem~3.6]{GS}
give a locally uniform limit $P$ satisfying
\[
 \partial_tP-\operatorname{div}(A_\infty\nabla P)
 =\partial_tP-\Delta P=0,
 \qquad \int_{B_1}P^2(y,0)\dd y=1.
\]
With $d=(\beta_0-n)/2$, the limit obeys
\begin{equation}\label{eq:limit-growth}
|P(y,t)|\le C(1+|y|+\sqrt{-t})^d,
\qquad t\le0.
\end{equation}
For any fixed $(y,t)$, caloric interior estimates give
\[
|\partial_y^\alpha\partial_t^bP(y,t)|
\le C_{\alpha,b}R^{d-|\alpha|-2b}
\]
for all sufficiently large $R$. Letting $R\to\infty$ shows that
$P$ is a caloric polynomial. In particular its derivatives belong
to every Gaussian $L^2$ space.

Fix $0<a<b<\infty$. For large $j$, $\rho_j\ge\max\{1,b\}$ and
\[
\operatorname{supp}\vartheta(\,\cdot/\kappa_j)
\subset B_{3\kappa_j/2}\subset B_{3\rho_j}.
\]
For $y$ in this support, take $R=\max\{1,b,|y|/3\}\le\rho_j$ in
\eqref{eq:polynomial-cylinder-bound}, using continuity at $|y|=3R$, to obtain
\[
\vartheta^2(y/\kappa_j)|w_j(y,-q^2)|^2
\le C_{\Theta_*,b}(1+|y|)^{\beta_0-n}
\qquad(a\le q\le b).
\]
Consequently
\begin{equation}\label{eq:cutoff-tail}
\lim_{M\to\infty}\limsup_{j\to\infty}\sup_{q\in[a,b]}
\int_{|y|>M}w_j^2(y,-q^2)\vartheta^2(y/\kappa_j)
\dd\nu_{0,q^2}(y)=0.
\end{equation}
Local uniform convergence, the continuity of $q\mapsto G_{0,q^2}$,
and \eqref{eq:cutoff-tail} prove \eqref{eq:height-convergence}.
Applying \eqref{eq:height-convergence} at $q_j$ and $q_j/2$ and taking $\log_4$ proves \eqref{eq:D-convergence}.
Local uniform convergence on the bounded cylinders proves
\eqref{eq:Theta-convergence}.
Finally, we have
\[
\partial_tw_j-\operatorname{div}(A_j^{\eps_j}\nabla w_j)
=\partial_tP-\Delta P=0.
\]
Boundedness and \textup{(A2)} imply the H\"older condition
\cite[(1.9)]{GS20} with exponent $1/2$ and fixed constants.
Fix $R\ge1$. By scaling \cite[Theorem~6.1]{GS20} and using
\textup{(A1)--(A3)}, for all sufficiently large $j$ we obtain
\begin{equation}\label{eq:corrected-gradient-error}
\begin{aligned}
\left\|\nabla w_j-\mathcal J_j^{\eps_j}\nabla P
\right\|_{L^\infty(Q^-(\mathbf0,R))}
&\le C_R\|w_j-P\|_{L^2(Q^-(\mathbf0,2R))}\\
&\quad+C_{P,R}\eps_j\log(2+\eps_j^{-1})\longrightarrow0.
\end{aligned}
\end{equation}
Parabolic Schauder regularity makes
$\nabla w_j$ and $\mathcal J_{A_j}$ continuous up to $t=0$.
Thus \eqref{eq:corrected-gradient-error} also holds on the slice $t=0$.
Every compact subset of $\R^n\times(-\infty,0]$ lies in such a cylinder,
which proves \eqref{eq:corrected-gradient-convergence}.
\end{proof}

Compactness and Gaussian rigidity yield propagation across noninteger doubling thresholds. A failure would force a noninteger constant limiting doubling index, contradicting Appendix~\ref{sec:gaussian-appendix}.

\begin{lemma}\label{lem:gap}
Fix $\Theta_*\ge1$, an integer $l\ge1$, and $0<\delta_0\le1/2$.
There are $\kappa_*=\kappa_*(\Theta_*,l,\delta_0)\ge2$ and
$\eta_*=\eta_*(\Theta_*,l,\delta_0)>0$ with the following property.
Let $u_\eps$ solve \eqref{eq:homogenization-equation} under
\textup{(A1)--(A4)} in $Q^-(\mathbf0,4)$, and let
$0<r\le1/\kappa_*$, $\eps/r\le\eta_*$.
Assume
\begin{equation}\label{eq:single-center-theta}
\Theta(u_\eps,\mathbf0,2^{-k})\le\Theta_*,
\qquad k\in\mathbb N_0,\quad r\le2^{-k}\le1.
\end{equation}
For either fixed sign,
\begin{equation}\label{eq:gap-input}
D_1(u_\eps;\mathbf0,r)\le l\pm\delta_0
\end{equation}
implies
\begin{equation}\label{eq:threshold-propagation}
D_1(u_\eps;\mathbf0,sr)\le l\pm\delta_0,
\qquad 1/4\le s\le1/2.
\end{equation}
\end{lemma}

\begin{proof}
Suppose the implication fails for every choice of thresholds. Choose
$r_j\downarrow0$, $\eps_j/r_j\to0$, and solutions satisfying
\eqref{eq:single-center-theta} and \eqref{eq:gap-input} but violating
\eqref{eq:threshold-propagation}. Let $R_j=2^{-m_j}$ be the smallest dyadic
number no smaller than $r_j$, so that $r_j\le R_j<2r_j$. Set
\[
\kappa_j=R_j^{-1},\quad
c_j^2=\int_{B_1}u_{\eps_j}^2(R_jy,0)\dd y>0,\quad
w_j=c_j^{-1}T_{\mathbf0,R_j}u_{\eps_j}.
\]
For every $k\in\mathbb N_0$ with $2^k\le\kappa_j$, the scale
$2^kR_j=2^{-(m_j-k)}$ lies in $[r_j,1]$. Thus
\eqref{eq:normalized-theta} holds with the same $\Theta_*$ and
homogenization parameter $\eps_j/R_j\to0$.
Lemma~\ref{lem:compactness} gives a nonzero caloric polynomial $P$.
After extraction, $q_j:=r_j/R_j\to q\in[1/2,1]$.
The heights $H_1(u_{\eps_j};\mathbf0,hr_j)$ are positive for all
$h\in[1/8,1]$ and large $j$: otherwise choose $h_j\to h$ with zero height
and apply \eqref{eq:truncated-scaling} and \eqref{eq:height-convergence}
at the rescaled radii $h_jq_j\in[1/16,1]$.
After extraction, fix the sign and let $s_j\to s\in[1/4,1/2]$.
The failed conclusion and \eqref{eq:D-convergence} yield
\[
D(P;\mathbf0,q)\le l\pm\delta_0\le D(P;\mathbf0,sq).
\]
The well-known rigidity of Gaussian doubling (see Appendix~\ref{sec:gaussian-appendix}) forces the noninteger
$l\pm\delta_0$ to be the homogeneous degree of $P$, a contradiction.
\end{proof}

\section{Uniform doubling estimates}\label{sec:uniform-doubling}

\begin{lemma}\label{lem:EFV}
Let $u$ be a weak solution of
\[
\partial_tu-\operatorname{div}\!\bigl(A(x,t)\nabla u\bigr)=0
\qquad\text{in }Q^-(\bzero,4),
\]
with $A$ satisfying \textup{(A1)} and \textup{(A2)} with Lipschitz constant $L_A$. If $\Theta(u,\bzero,1)\le \Theta_0$,
then
\begin{equation}\label{eq:EFV-all}
\sup_{0<r\le1}\Theta(u,\bzero,r)\le C(L_A,\Theta_0).
\end{equation}
\end{lemma}

\begin{proof}
Set $F(r)=\int_{B_r}u^2(x,0)\dd x$. After reversing time,
\cite[Theorem~2(1)]{EFV} gives
\begin{equation}\label{eq:EFV-spatial}
F(2r)\le B F(r),\qquad 0<r\le1/2,
\qquad B=B(L_A,\Theta_0)\ge1.
\end{equation}
By \cite[Theorem~3(1)]{EFV}, there exist $E=E(L_A,\Theta_0)$ and
$q_E=q_E(L_A,\Theta_0)>0$ such that
\begin{equation}\label{eq:EFV-cylinder}
\int_{Q^-(\bzero,q)}u^2\le E q^2F(q),
\qquad 0<q\le q_E.
\end{equation}
Choose a dyadic $q_0=2^{-N}\le\min\{1/4,q_E/4\}$. For $r\le q_0$,
\eqref{eq:EFV-spatial}--\eqref{eq:EFV-cylinder} yield
\[
\Theta(u,\bzero,r)\le16E\frac{F(4r)}{F(r)}\le16EB^2.
\]
For $q_0\le r\le1$, $\Theta(u,\mathbf0,1)\le\Theta_0$ and
\eqref{eq:EFV-spatial} give
\[
\Theta(u,\bzero,r)
\le\frac{\Theta_0F(1)}{q_0^2F(q_0)}\le4^N\Theta_0B^N.
\qedhere
\]
\end{proof}

\begin{lemma}\label{lem:caloric-seed}
For every $\Theta_0\ge1$, there exist $L_0(\Theta_0)<\infty$ and
$a_0(\Theta_0)\in(0,1/2]$ such that any caloric $v$ in $Q^-(\bzero,4)$
with $\Theta(v,\bzero,1)\le \Theta_0$ satisfies
\begin{equation}\label{eq:caloric-D}
D_1(v;r)\le L_0(\Theta_0),
\qquad 0<r\le a_0(\Theta_0).
\end{equation}
Moreover,
\begin{equation}\label{eq:caloric-D-almost}
D_1(v;r)\le D_1(v;R)+C\Theta_0 R^2,
\qquad 0<r\le R\le a_0(\Theta_0),
\end{equation}
where $C>0$ depends only on $n$ and the fixed cutoff.
\end{lemma}

\begin{proof}
Set
\[
J=\int_{Q^-(\bzero,4)}v^2,\qquad
F=\int_{B_1}v^2(y,0)\dd y>0,\qquad J\le \Theta_0F.
\]
The caloric interior estimates give a fixed $C_*\ge1$ such that
\begin{equation}\label{eq:seed-interior}
\sup_{B_2\times[-1,0]}
\bigl(|v|^2+|\nabla v|^2+|\partial_t v|^2\bigr)\le C_*J.
\end{equation}
Enlarging $C_*$ if necessary, for $0<r\le1$ we have
\begin{equation}\label{eq:seed-time-continuity}
\begin{aligned}
\|v(\cdot,-r^2)-v(\cdot,0)\|_{L^2(B_1)}
&\le\int_{-r^2}^0\|\partial_t v(\cdot,t)\|_{L^2(B_1)}\dd t\\
&\le C_*r^2J^{1/2}.
\end{aligned}
\end{equation}
Choose
\begin{equation}\label{eq:seed-parameters}
 a_0(\Theta_0)=\min\left\{\frac12,(2C_*\sqrt{\Theta_0})^{-1/2}\right\}.
\end{equation}
For $0<r\le a_0(\Theta_0)$, \eqref{eq:seed-time-continuity} yields
\[
\begin{aligned}
\|v(\cdot,-r^2)\|_{L^2(B_1)}
&\ge F^{1/2}-C_*r^2J^{1/2}\\
&\ge(1-C_*r^2\sqrt{\Theta_0})F^{1/2}
\ge\frac12F^{1/2}.
\end{aligned}
\]
Since $\vartheta=1$ on $B_1$,
\begin{equation}\label{eq:seed-height-lower}
\begin{aligned}
H_1(v;r)
&\ge(4\pi r^2)^{-n/2}e^{-1/(4r^2)}
\int_{B_1}v^2(y,-r^2)\dd y\\
&\ge cFr^{-n}e^{-1/(4r^2)}>0.
\end{aligned}
\end{equation}

Let $f=\vartheta v$. Its equation is
\begin{equation}\label{eq:seed-cutoff-equation}
\begin{aligned}
g:=(\partial_t-\Delta)f
&=-2\nabla\vartheta\cdot\nabla v-(\Delta\vartheta)v,\\
\supp g(\cdot,t)&\subset B_2\setminus B_1.
\end{aligned}
\end{equation}
In the following spatial integrals, $f$, $g$, and their derivatives
are evaluated at $(y,-r^2)$ unless indicated otherwise.
By \eqref{eq:seed-interior} and \eqref{eq:seed-cutoff-equation},
\begin{equation}\label{eq:seed-error-upper}
\int_{\R^n}(g^2+|fg|)\dd\nu_{0,r^2}
\le CJr^{-n}e^{-1/(4r^2)}.
\end{equation}
Dividing by \eqref{eq:seed-height-lower} gives
\begin{equation}\label{eq:seed-relative-error}
\frac{\displaystyle\int_{\R^n}(g^2+|fg|)\dd\nu_{0,r^2}}
{H_1(v;r)}
\le C\Theta_0,\qquad 0<r\le a_0(\Theta_0).
\end{equation}

Define
\begin{equation}\label{eq:seed-frequency}
\begin{aligned}
N_1(v;r)&=\frac{2r^2\displaystyle\int_{\R^n}|\nabla f|^2
\dd\nu_{0,r^2}}{H_1(v;r)},\\
W_r&=y\cdot\nabla f-2r^2\Delta f-N_1(v;r)f.
\end{aligned}
\end{equation}
A direct computation gives
\begin{equation}\label{eq:seed-frequency-derivative}
\begin{aligned}
\frac{d}{dr}N_1(v;r)
&=\frac{2}{rH_1(v;r)}\int_{\R^n}(W_r-r^2g)^2\dd\nu_{0,r^2}\\
&\quad-\frac{2r^3}{H_1(v;r)}\int_{\R^n}g^2\dd\nu_{0,r^2}.
\end{aligned}
\end{equation}
By \eqref{eq:seed-relative-error}, $\frac{d}{dr}N_1(v;r)\ge-C\Theta_0r^3$.
Hence, for $0<s\le t\le a_0(\Theta_0)$,
\begin{equation}\label{eq:seed-frequency-comparison}
\begin{aligned}
N_1(v;s)&\le N_1(v;t)+C\Theta_0\int_s^t q^3\dd q\\
&\le N_1(v;t)+C\Theta_0t^4.
\end{aligned}
\end{equation}

Differentiating $H_1(v;r)$ gives
\[
\frac{d}{dr}H_1(v;r)
=4r\int_{\R^n}\bigl(|\nabla f|^2-fg\bigr)\dd\nu_{0,r^2}.
\]
By \eqref{eq:seed-relative-error},
\begin{equation}\label{eq:seed-height-derivative}
\left|\frac{d}{dr}\log H_1(v;r)-\frac{2N_1(v;r)}r\right|
\le C\Theta_0 r.
\end{equation}
Integration over $[r/2,r]$ yields
\begin{equation}\label{eq:seed-doubling-average}
\left|D_1(v;r)-\frac1{\log2}
\int_{1/2}^1N_1(v;r\sigma)\frac{\dd\sigma}{\sigma}\right|
\le C\Theta_0 r^2.
\end{equation}
For $0<r\le R\le a_0(\Theta_0)$, apply
\eqref{eq:seed-frequency-comparison} with $s=r\sigma$ and
$t=R\sigma$. Since $R\le1/2$, \eqref{eq:seed-doubling-average}
gives
\[
\begin{aligned}
D_1(v;r)
&\le\frac1{\log2}\int_{1/2}^1
N_1(v;R\sigma)\frac{\dd\sigma}{\sigma}
+C\Theta_0(r^2+R^4)\\
&\le D_1(v;R)+C\Theta_0 R^2.
\end{aligned}
\]
This proves \eqref{eq:caloric-D-almost}.

Finally, fix $R=a_0(\Theta_0)$. Since $0\le\vartheta\le1$ and
$\int_{\R^n}\dd\nu_{0,R^2}=1$, \eqref{eq:seed-interior} and
\eqref{eq:seed-height-lower} imply
\begin{equation}\label{eq:seed-fixed-heights}
\begin{aligned}
H_1(v;R)&\le CJ\le C\Theta_0F,\\
H_1(v;R/2)&\ge cFR^{-n}e^{-1/R^2}.
\end{aligned}
\end{equation}
Therefore
\[
\begin{aligned}
\frac{H_1(v;R)}{H_1(v;R/2)}&\le C\Theta_0 R^n e^{1/R^2},\\
D_1(v;R)&\le C\bigl(1+\log \Theta_0+R^{-2}\bigr)=:C_0(\Theta_0).
\end{aligned}
\]
Combining this bound with \eqref{eq:caloric-D-almost} proves
\eqref{eq:caloric-D}, with
$L_0(\Theta_0)=C_0(\Theta_0)+C\Theta_0a_0(\Theta_0)^2$.
\end{proof}

For $d\in\mathbb N_0$, let $\mathcal P_{\le d}$ denote the vector space of
caloric polynomials of parabolic degree at most $d$, including the
zero polynomial.

\begin{lemma}\label{lem:polynomial-Theta}
For every $d\in\mathbb N_0$,
\begin{equation}\label{eq:poly-Theta-bound}
\Theta_{\mathrm{pol}}(d):=\sup_{P\in\mathcal P_{\le d}\setminus\{0\}}
\Theta(P,\bzero,1/2)<\infty.
\end{equation}
\end{lemma}

\begin{proof}
The space $\mathcal P_{\le d}$ is finite dimensional. For every $t_0\in\R$,
\begin{equation}\label{eq:polynomial-trace}
P(x,t)=\sum_{j=0}^{\lfloor d/2\rfloor}
\frac{(t-t_0)^j}{j!}\Delta^jP(x,t_0).
\end{equation}
Thus the restriction $P\mapsto P(\cdot,t_0)$ is injective.
In particular, $\|P(\cdot,0)\|_{L^2(B_{1/2})}$ is a norm on
$\mathcal P_{\le d}$. By equivalence of norms on this finite-dimensional space,
\[
\|P\|_{L^2(Q^-(\bzero,2))}^2
\le C_d\|P(\cdot,0)\|_{L^2(B_{1/2})}^2.
\]
Consequently, for $P\in\mathcal P_{\le d}\setminus\{0\}$,
\[
\Theta(P,\bzero,1/2)
=4\frac{\|P\|_{L^2(Q^-(\bzero,2))}^2}
{\|P(\cdot,0)\|_{L^2(B_{1/2})}^2}
\le4C_d.
\qedhere
\]
\end{proof}

\begin{lemma}\label{lem:propagation}
Let $u_\eps$ solve \eqref{eq:homogenization-equation} in
$Q^-(\mathbf0,4)$ under \textup{(A1)--(A4)}.
Fix an integer $l\ge1$, $\delta_0\in(0,1/2]$, a dyadic $a\in(0,1/8]$,
and $\Theta_\sharp>\max\{1,\Theta_{\mathrm{pol}}(l)\}$.
For either fixed sign, there exist dyadic
$r_*=r_*(l,\delta_0,\Theta_\sharp,a)\in(0,a]$ and
$\eta_*=\eta_*(l,\delta_0,\Theta_\sharp,a)>0$
such that the following holds. If $r=2^{-k_0}$ for some $k_0\in\mathbb N_0$, $r\le r_*$,
$\eps/r\le\eta_*$, and
\begin{alignat}{2}
&\Theta(u_\eps,\bzero,2^{-k})\le \Theta_\sharp,
&\qquad&k\in\mathbb N_0,\quad r\le2^{-k}\le1,
\label{eq:prop-Theta-hyp}\\
&D_1(u_\eps;2^{-k})\le l\pm\delta_0,
&\qquad&k\in\mathbb N_0,\quad r\le2^{-k}\le a,
\label{eq:prop-D-hyp}
\end{alignat}
then
\begin{equation}\label{eq:prop-conclusion}
\begin{aligned}
&\Theta(u_\eps,\bzero,r/2)\le \Theta_\sharp,\\
&D_1(u_\eps;r/2)\le l\pm\delta_0.
\end{aligned}
\end{equation}
\end{lemma}

\begin{proof}
Choose $r_*,\eta_*$ below the thresholds of Lemma~\ref{lem:gap}
with $\Theta_*=\Theta_\sharp$, $l$, and $\delta_0$.
Applying Lemma~\ref{lem:gap} with $s=1/2$, we obtain
\[
D_1(u_\eps;\mathbf0,r/2)\le l\pm\delta_0.
\]
It remains to prove $\Theta(u_\eps,\mathbf0,r/2)\le\Theta_\sharp$.
Suppose this fails for every smaller choice of $r_*,\eta_*$. Then there
exist counterexamples satisfying \eqref{eq:prop-Theta-hyp}--\eqref{eq:prop-D-hyp}
at dyadic scales $r_j\to0$, with $\eps_j/r_j\to0$, such that
\[
\Theta(u_j,\mathbf0,r_j/2)>\Theta_\sharp.
\]
Normalize $w_j=c_j^{-1}T_{\mathbf0,r_j}u_j$ by
$\int_{B_1}w_j^2(y,0)\dd y=1$ and put $\kappa_j=r_j^{-1}$.
By \eqref{eq:theta-scale-invariance} and \eqref{eq:prop-Theta-hyp},
\[
\Theta(w_j,\mathbf0,2^k)=\Theta(u_j,\mathbf0,2^kr_j)\le\Theta_\sharp,
\qquad k\in\mathbb N_0,\quad 2^k\le\kappa_j.
\]
Thus \eqref{eq:normalized-theta} holds with $\Theta_*=\Theta_\sharp$.
By Lemma~\ref{lem:compactness}, $w_j\to P$, where $P$ is a nonzero
caloric polynomial. For every fixed integer $k\ge0$, eventually
$2^kr_j\le a$, so
\begin{equation}\label{eq:limit-parent-D}
D(P;2^k)=\lim_{j\to\infty}D_1(u_j;2^kr_j)\le l\pm\delta_0.
\end{equation}
By \eqref{eq:degree-at-infinity},
\begin{equation}\label{eq:limit-degree-bound}
\deg_{\rm par}P\le l.
\end{equation}
Equations~\eqref{eq:Theta-convergence} and \eqref{eq:poly-Theta-bound}
then give
\[
\Theta_\sharp\le\Theta(P,\mathbf0,1/2)
\le \Theta_{\mathrm{pol}}(l)<\Theta_\sharp,
\]
a contradiction.
\end{proof}

\begin{proof}[Proof of Theorem~\ref{thm:center}]
Replacing $\Theta_0$ by $\max\{1,\Theta_0\}$, we may assume
$\Theta_0\ge1$.

Let $C_{\rm cal}(\Theta_0)$ be the constant obtained by applying
Lemma~\ref{lem:EFV} with $A=I$, and let $L_0(\Theta_0),a_0(\Theta_0)$ be supplied by
Lemma~\ref{lem:caloric-seed}. Choose
\begin{equation}\label{eq:choose-thresholds}
\begin{aligned}
L&\in\mathbb N+1/2,\qquad L>L_0(\Theta_0)+1,\\
a&=2^{-k_a},\quad k_a\in\mathbb N_0,\qquad a\le\min\{a_0(\Theta_0),1/8\},\\
\Theta_\sharp&>\max\{\Theta_0+1,C_{\rm cal}(\Theta_0)+1,2\Theta_{\mathrm{pol}}(\lfloor L\rfloor)\}.
\end{aligned}
\end{equation}
Apply Lemma~\ref{lem:propagation} with $l=L-1/2$, $\delta_0=1/2$ and the plus sign to obtain $r_*,\eta_*$, reducing
$\eta_*$ to at most $1$. The order of choices is
\[
\Theta_0\ \longmapsto\ (C_{\rm cal},L_0,a_0)
\ \longmapsto\ (L,a,\Theta_\sharp)
\ \longmapsto\ (r_*,\eta_*).
\]
In particular, the choice of $L$ precedes that of $\Theta_\sharp$.

After these choices, there exists $\eps_{\rm seed}>0$,
depending only on $\Theta_0,L,a,\Theta_\sharp,r_*,\eta_*$ and the structural
data, such that $0<\eps<\eps_{\rm seed}$ implies
\begin{alignat}{2}
&\Theta(u_\eps,\bzero,2^{-k})<\Theta_\sharp,
&\qquad&k\in\mathbb N_0,\quad r_*\le2^{-k}\le1,
\label{eq:seed-Theta}\\
&D_1(u_\eps;2^{-k})<L,
&\qquad&k\in\mathbb N_0,\quad r_*\le2^{-k}\le a.
\label{eq:seed-D}
\end{alignat}
Suppose no such $\eps_{\rm seed}$ exists. Then there are
$\eps_j\downarrow0$ and solutions $u_j$ satisfying the assumptions of
Theorem~\ref{thm:center}, with fixed structural constants, for which
at least one of \eqref{eq:seed-Theta} and \eqref{eq:seed-D} fails.
Normalize by
\[
\int_{B_1}u_j^2(x,0)\dd x=1.
\]
Equation~\eqref{eq:coarse-center} gives
$\int_{Q^-(\bzero,4)}u_j^2\le\Theta_0$. Local energy and H\"older
estimates and \cite[Theorem~3.6]{GS} yield, after extraction, a caloric
function $v$ such that
\[
\begin{gathered}
u_j\longrightarrow v
\quad\text{locally uniformly on }B_4\times(-16,0],\\
\int_{B_1}v^2(x,0)\dd x=1,\qquad
\int_{Q^-(\bzero,4)}v^2\le\Theta_0.
\end{gathered}
\]
The last inequality follows by exhaustion and lower semicontinuity.
Thus $\Theta(v,\bzero,1)\le\Theta_0$.

At $r=1$, \eqref{eq:coarse-center} and \eqref{eq:choose-thresholds}
already give $\Theta(u_j,\bzero,1)\le\Theta_0<\Theta_\sharp$.
For each fixed dyadic $r\in[r_*,1/2]$, the integration sets defining
$\Theta(u_j,\bzero,r)$ are compactly contained in
$B_4\times(-16,0]$. Lemma~\ref{lem:EFV}, applied to $v$ with $A=I$,
and local uniform convergence give
\[
\Theta(u_j,\bzero,r)\longrightarrow\Theta(v,\bzero,r)
\le C_{\rm cal}(\Theta_0)<\Theta_\sharp.
\]
For each fixed dyadic $r\in[r_*,a]$, the cutoff has fixed compact
support, so local uniform convergence and \eqref{eq:seed-height-lower} give
\[
H_1(u_j;q)\longrightarrow H_1(v;q)>0
\qquad(q=r,r/2).
\]
Consequently, Lemma~\ref{lem:caloric-seed} and
\eqref{eq:choose-thresholds} imply
\[
D_1(u_j;r)\longrightarrow D_1(v;r)
\le L_0(\Theta_0)<L.
\]
Since both scale ranges contain only finitely many dyadic radii,
\eqref{eq:seed-Theta} and \eqref{eq:seed-D} hold simultaneously for all
sufficiently large $j$, contradicting the choice of $u_j$.

Reduce $\eps_{\rm seed}$ so that
$\eps_{\rm seed}\le\eta_*r_*/2$. For $\eps<\eps_{\rm seed}$,
let $r_{\rm mic}=2^{-k_{\rm mic}}$, $k_{\rm mic}\in\mathbb N_0$, be the least dyadic number not
smaller than $\eps/\eta_*$. Then
\begin{equation}\label{eq:micro-radius}
\eps/\eta_*\le r_{\rm mic}<2\eps/\eta_*\le r_*.
\end{equation}
Starting with \eqref{eq:seed-Theta}--\eqref{eq:seed-D}, apply
Lemma~\ref{lem:propagation} successively at $r_*,r_*/2,\ldots,r_{\rm mic}$.
This yields
\begin{equation}\label{eq:controlled-dyadics}
\Theta(u_\eps,\bzero,2^{-k})\le \Theta_\sharp,
\qquad k\in\mathbb N_0,\quad r_{\rm mic}\le2^{-k}\le1.
\end{equation}

Set $v=T_{\bzero,r_{\rm mic}}u_\eps$. By
\eqref{eq:space-time-lipschitz} and \eqref{eq:micro-radius}, its coefficient
matrix $a_{\rm mic}$ satisfies
\[
\begin{aligned}
|a_{\rm mic}(y,s)-a_{\rm mic}(y',s')|
&\le\frac{2L_A}{\eta_*}
\bigl(|y-y'|+|s-s'|^{1/2}\bigr),\\
\Theta(v,\bzero,1)&\le \Theta_\sharp.
\end{aligned}
\]
Lemma~\ref{lem:EFV} and \eqref{eq:theta-scale-invariance} imply
\begin{equation}\label{eq:micro-all}
\Theta(u_\eps,\bzero,r)\le C(\Theta_\sharp,\eta_*)
\qquad(0<r\le r_{\rm mic}).
\end{equation}
If $r_{\rm mic}\le r\le1$, choose adjacent dyadics
$s/2\le r\le s$ with $s/2\ge r_{\rm mic}$. Then
\eqref{eq:controlled-dyadics} and Lemma~\ref{lem:spatial} give
\begin{equation}\label{eq:intermediate-radii}
\begin{aligned}
\Theta(u_\eps,\bzero,r)
&\le4\Theta(u_\eps,\bzero,s)
\frac{\displaystyle\int_{B_s}u_\eps^2(x,0)\dd x}{\displaystyle\int_{B_{s/2}}u_\eps^2(x,0)\dd x}\\
&\le C\Theta_\sharp^2.
\end{aligned}
\end{equation}
At $r=r_{\rm mic}$ use \eqref{eq:micro-all}.

For $\eps\ge\eps_{\rm seed}$, the coefficient modulus is at most
$L_A/\eps_{\rm seed}$. Lemma~\ref{lem:EFV} applies directly.
Together with \eqref{eq:micro-all} and
\eqref{eq:intermediate-radii}, this proves
\eqref{eq:center-conclusion} for all $\eps>0$. Finally,
\eqref{eq:spatial-conclusion} follows from
Lemma~\ref{lem:spatial} and \eqref{eq:center-conclusion}.
\end{proof}

\begin{proof}[Proof of Corollary~\ref{thm:allcenters}]
Replacing $\Theta_0$ by $\max\{1,\Theta_0\}$, we may assume $\Theta_0\ge1$.
By \eqref{eq:theta-scale-invariance}, Theorem~\ref{thm:center} applies
at any center $\mathbf z$ and radius $R$ to $T_{\mathbf z,R}u_\eps$,
with parameter $\eps/R$ and a phase shift of $A$. We use $R=1$ below.
For $\sigma\in[-4,4]$, set
\[
\begin{aligned}
J(\sigma)&=\int_{B_6\times[-36,\sigma]}u_\eps^2\dd y\dd s,\\
F_x(r)&=\int_{B(x,r)}u_\eps^2(y,\sigma)\dd y.
\end{aligned}
\]
Fix $x\in B_2$ and choose $N=8$, $x_i=ix/N$,
$0\le i\le N$. Then
\begin{equation}\label{eq:chain-inclusions}
\begin{aligned}
|x_{i+1}-x_i|&\le1/4,\\
B(x_i,1/2)&\subset B(x_{i+1},1),\\
Q^-((x_i,\sigma),4)&\subset B_6\times[-36,\sigma].
\end{aligned}
\end{equation}
Define $\Theta_{0,i}=\max\{1,J(\sigma)/F_{x_i}(1)\}$. Initially
$\Theta_{0,0}\le\max\{1,J(\sigma)/F_0(1/2)\}\le \Theta_0$. If $\Theta_{0,i}<\infty$, then
\eqref{eq:chain-inclusions} and Theorem~\ref{thm:center}, applied to
$T_{(x_i,\sigma),1}u_\eps$ with $\Theta_0=\Theta_{0,i}$, yield
\[
\begin{aligned}
\Theta(u_\eps,(x_i,\sigma),1)&\le \Theta_{0,i},\\
F_{x_i}(1)&\le C(\Theta_{0,i})F_{x_i}(1/2)\\
&\le C(\Theta_{0,i})F_{x_{i+1}}(1).
\end{aligned}
\]
Consequently
\begin{equation}\label{eq:chain-recursion}
\Theta_{0,i+1}\le C(\Theta_{0,i})\Theta_{0,i},\qquad 0\le i<8.
\end{equation}
Taking $C\ge1$ in Theorem~\ref{thm:center} nondecreasing in
its argument, iteration of \eqref{eq:chain-recursion} gives
$\Theta_{0,8}\le C(\Theta_0)$. A final application with
$\Theta_0=\Theta_{0,8}$ at $(x,\sigma)$ gives \eqref{eq:allcenters}
with a constant $\Theta_1=\Theta_1(\Theta_0)$, uniformly in $x$ and $\sigma$.
\end{proof}

\section{Truncated doubling estimates and nondegeneracy}\label{sec:drop-nondegeneracy}

For the geometric estimates, the local control hypothesis is
\begin{equation}\label{eq:uniform-theta}
\sup_{\substack{\mathbf z\in Q(\mathbf0,2)\\0<r\le1}}
\Theta(u_\eps,\mathbf z,r)\le\Theta_1<\infty.
\end{equation}
By Corollary~\ref{thm:allcenters} and \eqref{eq:theta-local},
\eqref{eq:rough} implies this condition with $\Theta_1=\Theta_1(\Theta_0)$.
We first propagate noninteger doubling thresholds at interior centers.

\begin{theorem}\label{thm:ls}
Let $u_\eps$ solve \eqref{eq:homogenization-equation} in $Q(\mathbf0,6)$
under \textup{(A1)--(A4)} and satisfy \eqref{eq:uniform-theta}.
Fix an integer $l\ge1$ and $\delta_0\in(0,1/2]$.
There are constants
\[
\kappa_*=\kappa_*(\Theta_1,l,\delta_0)\ge2,
\qquad
0<\eta_*=\eta_*(\Theta_1,l,\delta_0)<1,
\]
such that the following holds. Let $\mathbf z_0\in Q(\mathbf0,1)$, and
\begin{equation}\label{eq:initial-radius}
0<r\le\frac1{\kappa_*},\qquad 0<\eps\le\eta_*r.
\end{equation}
For either fixed sign, if
\begin{equation}\label{eq:ls-initial-bounds}
D_1(u_\eps;\mathbf z_0,r)\le l\pm\delta_0,
\end{equation}
then
\begin{equation}\label{eq:ls-one-step}
D_1(u_\eps;\mathbf z_0,s)\le l\pm\delta_0,
\qquad r/4\le s\le r/2.
\end{equation}
Moreover,
\begin{equation}\label{eq:ls-iterated}
D_1(u_\eps;\mathbf z_0,s)\le l\pm\delta_0,
\qquad \frac{\eps}{2\eta_*}\le s\le r/2.
\end{equation}
\end{theorem}

\begin{proof}
Apply Lemma~\ref{lem:gap} with $\Theta_*=\max\{1,\Theta_1\}$, $l$ and $\delta_0$.
For $\mathbf z_0=(x_0,t_0)$, define
$v_\eps(y,t):=u_\eps(x_0+y,t_0+t)$.
For $r_j=2^{-j}r$ with $\eps\le\eta_*r_j$, we have
\[
r_j\le r\le\kappa_*^{-1},\qquad \eps/r_j\le\eta_*.
\]
For every $k\in\mathbb N_0$ with $r_j\le2^{-k}\le1$,
\eqref{eq:uniform-theta} gives
\[
\Theta(v_\eps,\mathbf0,2^{-k})
=\Theta(u_\eps,\mathbf z_0,2^{-k})\le\Theta_1.
\]
Thus \eqref{eq:single-center-theta} holds at every such $r_j$.
Applying Lemma~\ref{lem:gap} at $j=0$ gives \eqref{eq:ls-one-step}.
Inductively, the conclusion
at $s=r_j/2=r_{j+1}$ supplies the next input, and the lemma applies
on each $[r_j/4,r_j/2]$.
If $J$ is the largest index with $\eps\le\eta_*r_J$, then
$r_J/2<\eps/\eta_*\le r_J$, and
\[
\bigcup_{j=0}^J[r_j/4,r_j/2]=[r_J/4,r/2]
\supset[\eps/(2\eta_*),r/2].
\]
This proves \eqref{eq:ls-iterated}.
\end{proof}

\begin{proposition}
\label{prop:theta-controls-truncated-doubling}
Let $u_\eps$ solve \eqref{eq:homogenization-equation} under
\textup{(A1)} and satisfy \eqref{eq:uniform-theta}. There exists $L=L(\Theta_1)\ge2$,
independent of $\eps$, such that
\begin{equation}\label{eq:all-truncated-doubling}
D_1(u_\eps;\mathbf z,r)\le L,
\qquad \mathbf z\in Q(\mathbf0,1),\quad 0<r\le1.
\end{equation}
\end{proposition}

\begin{proof}
Fix $\mathbf z=(x_0,t_0)\in Q(\mathbf0,1)$ and $0<r\le1$.
Write
\[
F_{x_0,t}(s)=\int_{B(x_0,s)}u_\eps^2(y,t)\dd y.
\]
Both $(x_0,t_0-r^2)$ and $(x_0,t_0-r^2/4)$ lie in
$Q(\mathbf0,2)$. The inclusion
\[
Q^-((x_0,t_0-r^2),r)
\subset Q^-((x_0,t_0-r^2/4),2r)
\]
and the interior estimate give
\begin{equation}\label{eq:two-slice-comparison}
\begin{aligned}
F_{x_0,t_0-r^2}(r/2)
&\le\frac C{r^2}\int_{Q^-((x_0,t_0-r^2),r)}
 |u_\eps(y,s)|^2\dd y\dd s\\
&\le\frac C{r^2}\int_{Q^-((x_0,t_0-r^2/4),2r)}
 |u_\eps(y,s)|^2\dd y\dd s\\
&\le C\Theta_1F_{x_0,t_0-r^2/4}(r/2).
\end{aligned}
\end{equation}
On $B(x_0,r/2)$, we have
$\vartheta(y-x_0)=1$ and $G_{x_0,r^2/4}(y)\ge cr^{-n}$.
Thus \eqref{eq:two-slice-comparison} yields
\begin{equation}\label{eq:truncated-height-lower-bound}
\begin{aligned}
H_1(u_\eps;\mathbf z,r/2)
&\ge cr^{-n}F_{x_0,t_0-r^2/4}(r/2)\\
&\ge c\Theta_1^{-1}r^{-n}F_{x_0,t_0-r^2}(r/2).
\end{aligned}
\end{equation}
For the rest of the proof, write $F(s)=F_{x_0,t_0-r^2}(s)$.
Applying Lemma~\ref{lem:spatial} after translation, with
$r/2\le s\le1$ in its hypothesis, gives
\begin{equation}\label{eq:terminal-support-growth}
F(s)\le C_0\left(\frac{s}{r/2}\right)^{\beta_0}F(r/2),
\qquad r/2\le s\le2.
\end{equation}

Set
\[
A_j:=\bigl(B(x_0,2^jr)\setminus B(x_0,2^{j-1}r)\bigr)
       \cap B(x_0,2),\qquad s_j:=\min\{2^jr,2\}.
\]
Since the cutoff is supported in $B(x_0,2)$,
\[
\begin{aligned}
H_1(u_\eps;\mathbf z,r)
={}&\int_{B(x_0,r/2)}|u_\eps(y,t_0-r^2)|^2\dd\nu_{x_0,r^2}(y)\\
&+\sum_{\substack{j\ge0\\2^{j-1}r<2}}
\int_{A_j}|u_\eps(y,t_0-r^2)|^2\vartheta^2(y-x_0)
\dd\nu_{x_0,r^2}(y).
\end{aligned}
\]
On $A_j$, $G_{x_0,r^2}(y)\le Cr^{-n}e^{-4^j/16}$, while
\eqref{eq:terminal-support-growth} gives
\[
\int_{A_j}|u_\eps(y,t_0-r^2)|^2\dd y
\le F(s_j)\le C(\Theta_1)2^{(j+1)\beta_0}F(r/2).
\]
Therefore
\begin{equation}\label{eq:annular-height-bound}
\begin{aligned}
H_1(u_\eps;\mathbf z,r)
&\le Cr^{-n}\left[F(r/2)+
 \sum_{\substack{j\ge0\\2^{j-1}r<2}}e^{-4^j/16}F(s_j)\right]\\
&\le C(\Theta_1)r^{-n}F(r/2)
 \left[1+\sum_{j=0}^{\infty}e^{-4^j/16}2^{(j+1)\beta_0}\right]\\
&\le C(\Theta_1)r^{-n}F(r/2).
\end{aligned}
\end{equation}
Combining \eqref{eq:annular-height-bound} with
\eqref{eq:truncated-height-lower-bound} gives
\begin{equation}\label{eq:uniform-height-ratio}
\frac{H_1(u_\eps;\mathbf z,r)}{H_1(u_\eps;\mathbf z,r/2)}
\le C(\Theta_1).
\end{equation}
Taking $\log_4$ and increasing the resulting constant to at least two
proves the assertion.
\end{proof}

\begin{lemma}\label{lem:caloric-frequency-drop}
Let $v\not\equiv0$ be an ancient caloric function on
$\R^n\times(-\infty,0)$, continuous up to $t=0$, and suppose
\[
|v(x,t)|\le C(1+|x|+\sqrt{-t})^d\qquad(t\le0)
\]
for some $C>0$ and $d\ge0$. Let $l$ be a positive integer and
$\delta\in(0,1/2]$. Suppose
$N(v;\mathbf0,r)\le l-\delta$. Then, for $0<s\le1$,
\begin{equation}\label{eq:caloric-frequency-drop}
N(v;\mathbf0,sr)\le l-1+\frac{l^2}{\delta}s^2.
\end{equation}
In particular,
\begin{equation}\label{eq:caloric-doubling-drop}
D(v;\mathbf0,r)\le l-\delta
\quad\Longrightarrow\quad
D\left(v;\mathbf0,\frac{\delta r}{4l}\right)
\le l-1+\frac\delta4.
\end{equation}
\end{lemma}

\begin{proof}
By the Liouville theorem, $v$ is a caloric polynomial.
By scaling and multiplication by a constant, take $r=1$ and $H(v;\mathbf0,1)=1$.
Equation~\eqref{eq:gaussian-degree-identities} gives
\[
\sum_k \alpha_k=1,\qquad \sum_k k \alpha_k\le l-\delta.
\]
Thus
\[
\sum_{k<l}\alpha_k=1-\sum_{k\ge l}\alpha_k
\ge1-\frac1l\sum_{k\ge l}k \alpha_k\ge\frac\delta l.
\]
For $0<s\le1$, we have
\[
\begin{aligned}
N(v;\mathbf0,s)
&=\frac{\sum_{k<l}k \alpha_ks^{2k}+\sum_{k\ge l}k \alpha_ks^{2k}}
{\sum_k \alpha_ks^{2k}}\\
&\le l-1+\frac{\sum_{k\ge l}k \alpha_ks^{2k}}{\sum_{k<l}\alpha_ks^{2k}}\\
&\le l-1+s^2\frac{\sum_{k\ge l}k \alpha_k}{\sum_{k<l}\alpha_k}\\
&\le l-1+\frac{l^2}{\delta}s^2.
\end{aligned}
\]
By the parabolic frequency formula of Poon \cite{Poon}, we have
\begin{equation}\label{eq:full-doubling-average}
D(v;\mathbf0,r)=\frac1{\log2}\int_{r/2}^r
N(v;\mathbf0,q)\frac{\dd q}{q}.
\end{equation}

By \eqref{eq:full-doubling-average} and monotonicity of $N$,
\[
N(v;\mathbf0,r/2)\le D(v;\mathbf0,r)\le l-\delta.
\]
Apply \eqref{eq:caloric-frequency-drop} at scale $r/2$ with
$s=\delta/(2l)$ to obtain
\[
D\left(v;\mathbf0,\frac{\delta r}{4l}\right)
\le N\left(v;\mathbf0,\frac{\delta r}{4l}\right)
\le l-1+\frac\delta4.\qedhere
\]
\end{proof}

\begin{theorem}\label{thm:truncated-doubling-drop}
Fix $\Theta_1<\infty$, $L\ge2$ and $\delta\in(0,1/2]$.
There exist
\[
\kappa_*=\kappa_*(\Theta_1,L,\delta)\ge2,\qquad
\eta_*=\eta_*(\Theta_1,L,\delta)>0
\]
with the following property. Let $u_\eps$ solve
\eqref{eq:homogenization-equation} under \textup{(A1)--(A4)} in
$Q^-(\mathbf z_0,4)$, and assume
\begin{equation}\label{eq:drop-center-theta}
\Theta(u_\eps,\mathbf z_0,\rho)\le\Theta_1,
\qquad 0<\rho\le1.
\end{equation}
If
\[
0<r\le\frac1{\kappa_*},\qquad 0<\eps<\eta_*r,
\qquad 1\le l\le L,\quad l\in\mathbb N,
\]
then
\begin{equation}\label{eq:truncated-drop-implication}
D_1(u_\eps;\mathbf z_0,r)\le l-\delta
\quad\Longrightarrow\quad
D_1\left(u_\eps;\mathbf z_0,\frac{\delta r}{4l}\right)
\le l-1+\delta.
\end{equation}
\end{theorem}

\begin{proof}
We argue by contradiction. For the candidate constants $\kappa_*=j+2$ and $\eta_*=j^{-1}$,
choose counterexamples with
\[
0<r_j\le\frac1{j+2},\qquad
0<\eps_j<r_j/j,\qquad \kappa_j:=\frac1{r_j}\ge j+2.
\]
Write $\mathbf z_j=(x_j,t_j)$, and define
\begin{equation}\label{eq:drop-normalization}
c_j^2=\int_{B_1}|u_{\eps_j}(x_j+r_jy,t_j)|^2\dd y>0,
\qquad
w_j(y,t)=c_j^{-1}u_{\eps_j}(x_j+r_jy,t_j+r_j^2t).
\end{equation}
The translated coefficient matrices satisfy \textup{(A1)--(A4)}.
By \eqref{eq:drop-center-theta}, we have
\[
\begin{gathered}
\widetilde\eps_j:=\eps_j/r_j<1/j,\qquad
4\kappa_jr_j=4,\qquad
\int_{B_1}|w_j(y,0)|^2\dd y=1,\\
\Theta(w_j,\mathbf0,s)\le\Theta_1\quad(0<s\le \kappa_j).
\end{gathered}
\]
Lemma~\ref{lem:compactness} yields
\[
w_j\longrightarrow v\quad\text{locally uniformly},\qquad
\partial_tv-\Delta v=0,\qquad \int_{B_1}|v(y,0)|^2\dd y=1,
\]
with the growth bound \eqref{eq:limit-growth}. By
\eqref{eq:truncated-scaling}, the convergence of $H_{\kappa_j}(w_j;\mathbf0,q_j)$ for
$q_j\to q\in[\delta/(8L),1]$ to $H(v;\mathbf0,q)>0$
excludes $H_{\kappa_j}(w_j;\mathbf0,q_j)=0$. Passing to a further
subsequence, $l_j=l$, and the failed implication gives
\[
D(v;\mathbf0,1)\le l-\delta,\qquad
D\left(v;\mathbf0,\frac\delta{4l}\right)\ge l-1+\delta.
\]
Lemma~\ref{lem:caloric-frequency-drop}, with $r=1$, gives
\[
l-1+\delta
\le D\left(v;\mathbf0,\frac\delta{4l}\right)
\le l-1+\frac\delta4,
\]
a contradiction.
\end{proof}

\begin{theorem}\label{thm:low-doubling-nondegeneracy}
Fix $\Theta_1<\infty$ and $\mu>0$. There exist
\[
\kappa_*=\kappa_*(\Theta_1,\mu)\ge2,\qquad
\eta_*=\eta_*(\Theta_1,\mu)>0
\]
with the following property. Let $u_\eps$ solve
\eqref{eq:homogenization-equation} under \textup{(A1)--(A5)} in
$Q^-(\mathbf z_0,4)$, and suppose
\[
u_\eps(\mathbf z_0)=0,\qquad
\Theta(u_\eps,\mathbf z_0,\rho)\le\Theta_1\quad(0<\rho\le1).
\]
If
\[
0<r\le\frac1{\kappa_*},\qquad 0<\eps<\eta_*r,\qquad
D_1(u_\eps;\mathbf z_0,r)\le\frac32,
\]
then
\begin{equation}\label{eq:low-doubling-nondegeneracy}
|\nabla_xu_\eps(\mathbf z_0)|>0.
\end{equation}
In particular, if $\Theta(u_\eps,\mathbf z_0,\rho)\le\Theta_1$ for $0<\rho\le1$, $0<r\le1/\kappa_*$, and $0<\eps<\eta_*r$, every
$\mathbf z_0\in S(u_\eps)$ satisfies
\begin{equation}\label{eq:singular-center-lower-doubling}
D_1(u_\eps;\mathbf z_0,r)>\frac32.
\end{equation}
\end{theorem}

\begin{proof}
We argue by contradiction. Choose counterexamples with
$0<r_j\le1/(j+2)$ and $0<\eps_j<r_j/j$.
Set $\kappa_j=1/r_j\ge j+2$.
Use \eqref{eq:drop-normalization} at scale $r_j$; then
$\eps_j/r_j<1/j\to0$, and \eqref{eq:truncated-scaling} uses the
fixed rescaled cutoff radius $\kappa_j$.
Equation~\eqref{eq:height-convergence} gives positive limiting masses at $q=1,1/2$, and hence
$H_{\kappa_j}(w_j;\mathbf0,q)>0$ at these scales for all large $j$. Hence, after extraction,
\[
w_j(\mathbf0)=0,\qquad \nabla w_j(\mathbf0)=0,\qquad
\int_{B_1}|w_j(y,0)|^2\dd y=1,\qquad
D_{\kappa_j}(w_j;\mathbf0,1)\le\frac32.
\]
Translations preserve \eqref{eq:corrector-nondegeneracy}. By
Lemma~\ref{lem:compactness} and \eqref{eq:corrected-gradient-convergence},
\[
\begin{gathered}
w_j\longrightarrow v\quad\text{locally uniformly},\qquad
v(\mathbf0)=0,\qquad \int_{B_1}|v(y,0)|^2\dd y=1,\\
D(v;\mathbf0,1)\le\frac32,\\
\mu|\nabla v(\mathbf0)|
\le|\mathcal J_{A_j}(0,0)\nabla v(\mathbf0)|
=|\nabla w_j(\mathbf0)-\mathcal J_{A_j}(0,0)\nabla v(\mathbf0)|
\longrightarrow0.
\end{gathered}
\]
Thus $v(\mathbf0)=0$ and $\nabla v(\mathbf0)=0$. By
Lemma~\ref{lem:gaussian-degrees}, $\alpha_0=\alpha_1=0$, and
\[
\frac{H(v;\mathbf0,1)}{H(v;\mathbf0,1/2)}
=\frac{\sum_{k\ge2}\alpha_k}{\sum_{k\ge2}4^{-k}\alpha_k}\ge16,
\qquad 2\le D(v;\mathbf0,1)\le\frac32,
\]
a contradiction.
\end{proof}

\section{Caloric approximation and summable turning}\label{sec:turning}

We first establish two Gaussian estimates for caloric functions.
The subsequent approximation and turning estimates use the fixed
unit-cutoff quantities $H_1,D_1$ from \eqref{eq:truncated-height} and
\eqref{eq:truncated-doubling-index} at every scale.

For $m\in\mathbb N_0$, let $\Pi_m$ be the orthogonal projection in
$L^2(\mathbb R^n,\nu_{0,1})$, where
\[
 \dd\nu_{0,1}(y)=(4\pi)^{-n/2}e^{-|y|^2/4}\,dy,
\]
onto the degree-$m$ Hermite eigenspace of
$-\mathcal L=-\Delta+\frac y2\cdot\nabla$, with eigenvalue $m/2$.
This space consists of the slices $P(\cdot,-1)$ of parabolically
homogeneous caloric polynomials $P$ of degree $m$.

If $v$ is caloric on $\R^n\times(-16,0]$ and has polynomial growth
on $\R^n\times[-8,0]$, set $f_s(y)=v(sy,-s^2)$. For $0<s<\sqrt8$,
\[
 s\partial_s f_s=-2\mathcal Lf_s,
 \qquad s\partial_s\Pi_kf_s=k\Pi_kf_s,
 \qquad \Pi_kf_s=s^k\Pi_kf_1.
\]
Polynomial growth and interior estimates justify differentiation and
Gaussian integration by parts on compact subintervals. If
$\Pi_kf_1=a_kp_k(\cdot,-1)$, where $a_k\ge0$ and $p_k$ is a
homogeneous caloric polynomial of degree $k$ with
$\|p_k(\cdot,-1)\|_{L^2(\nu_{0,1})}=1$, Parseval's identity gives
\begin{equation}\label{eq:finite-time-hermite}
\begin{aligned}
 f_s&=\sum_{k\ge0}a_ks^kp_k(\cdot,-1)
       &&\text{in }L^2(\nu_{0,1}),\\
 H(v;\mathbf0,s)&=\sum_{k\ge0}a_k^2s^{2k},
 &E(v;\mathbf0,s)&=\sum_{k\ge0}ka_k^2s^{2k}.
\end{aligned}
\end{equation}
The energy identity follows from
$2\langle-\mathcal Lf_s,f_s\rangle=2\|\nabla f_s\|^2$.

\begin{lemma}\label{lem:dominant-degree}
Let $v$ be a nonzero caloric function on $\mathbb R^n\times(-16,0]$,
with polynomial growth on $\mathbb R^n\times[-8,0]$. Suppose
\[
 H(v;\mathbf0,1)=1,\qquad
 D(v;\mathbf0,2),D(v;\mathbf0,1/2)\in[m-\delta,m+\delta],
 \qquad m\in\mathbb N_0,\quad0<\delta\le1/2.
\]
Let $a_k,p_k$ be as in \eqref{eq:finite-time-hermite}. Then
\begin{gather}
 a_m^2\ge1-4\delta,\label{eq:dominant-mass}\\
 \|v(\cdot,-1)-a_mp_m(\cdot,-1)\|_{L^2(\nu_{0,1})}
 \le2\sqrt\delta,\label{eq:dominant-error}\\
 \sup_{Q^-(\mathbf0,1/2)}|\nabla_x^j(v-a_mp_m)|
 \le C(j)\sqrt\delta.
 \label{eq:dominant-local-error}
\end{gather}
Estimate~\eqref{eq:dominant-local-error} holds for every $j\in\mathbb N_0$.
If $v$ extends calorically with polynomial growth through time $1/4$,
\eqref{eq:dominant-local-error} holds on $Q(\mathbf0,1/2)$.
\end{lemma}

\begin{proof}
Write $H(s)=H(v;\mathbf0,s)$, $N(s)=N(v;\mathbf0,s)$ and
$D(s)=D(v;\mathbf0,s)$. By \eqref{eq:finite-time-hermite},
\[
 H(s)=\sum_ka_k^2s^{2k},\qquad
 N(s)=\frac{\sum_k k a_k^2s^{2k}}{H(s)},\qquad
 \sum_ka_k^2=1.
\]
By frequency monotonicity and \eqref{eq:full-doubling-average},
\[
 D(1)\ge m-\delta,\qquad N(1)\le m+\delta,
 \qquad N(1/2)\ge m-\delta.
\]
Set
\[
 A=\sum_k4^{m-k}a_k^2=4^{m-D(1)}\le4^\delta\le2.
\]
Then
\begin{align*}
 \sum_k(k-m)(1-4^{m-k})a_k^2
 &=N(1)-m-A\{N(1/2)-m\}\\
 &\le\delta(1+A)\le3\delta.
\end{align*}
For $k\ne m$, $(k-m)(1-4^{m-k})\ge3/4$. Hence
\[
 \|v(\cdot,-1)-a_mp_m(\cdot,-1)\|_{L^2(\nu_{0,1})}^2
 =\sum_{k\ne m}a_k^2=1-a_m^2\le4\delta.
\]
Put $g=v(\cdot,-1)-a_mp_m(\cdot,-1)$. For $|x|\le1/2$ and
$3/4\le h\le5/4$, the Gaussian weight satisfies
\[
 \int_{\R^n}\frac{|\nabla_x^jG_{x,h}(y)|^2}{G_{0,1}(y)}\,dy\le C(j).
\]
The semigroup representation at $h=t+1$ and Cauchy--Schwarz give
\[
 |\nabla_x^j(v-a_mp_m)(x,t)|
 =\left|\int\nabla_x^jG_{x,t+1}(y)g(y)\,dy\right|
 \le C(j)\|g\|_{L^2(\nu_{0,1})}.
\]
This proves all assertions.
\end{proof}

\begin{lemma}
\label{lem:fc-caloric-turning}
Let $v$ be a nonzero caloric function on $\mathbb R^n\times(-16,0]$,
with polynomial growth on $\mathbb R^n\times[-8,0]$.
Suppose
\[
 D(v;\mathbf0,2),\ D(v;\mathbf0,1/2)\in[m-\delta,m+\delta],
 \qquad m\in\mathbb N_0,\quad0<\delta<1.
\]
With $\|\cdot\|=\|\cdot\|_{L^2(\nu_{0,1})}$, we have
\begin{equation}\label{eq:caloric-projection-turning}
 \left\|\frac{\Pi_m[v(\cdot,-1)]}{\sqrt{H(v;\mathbf0,1)}}
       -\frac{\Pi_m[v(\cdot/2,-1/4)]}{\sqrt{H(v;\mathbf0,1/2)}}
 \right\|
 \le\frac{D(v;\mathbf0,2)-D(v;\mathbf0,1/2)}{2(1-\delta)}.
\end{equation}
\end{lemma}

\begin{proof}
Let $a_k,p_k$ be as in \eqref{eq:finite-time-hermite}. Writing
$H(s)=H(v;\mathbf0,s)$, $N(s)=N(v;\mathbf0,s)$ and
$D(s)=D(v;\mathbf0,s)$, we have
\begin{equation}\label{eq:turning-height-frequency}
 H(s)=\sum_{k\ge0}a_k^2s^{2k},\qquad
 N(s)=\frac{\sum_{k\ge0}k a_k^2s^{2k}}{H(s)}.
\end{equation}
Since $H(2)=\sum_k4^ka_k^2<\infty$, the required differentiated
series converge locally uniformly for $0<s<2$. In particular,
\[
 \sum_{k\ge0}(1+k)^2a_k^2\le C H(2)<\infty
\]
justifies termwise differentiation on $[1/2,1]$. Direct calculation
yields
\begin{equation}\label{eq:turning-scale-derivatives}
 sH'(s)=2N(s)H(s),\qquad
 sN'(s)=\frac{2}{H(s)}\sum_{k\ge0}(k-N(s))^2a_k^2s^{2k}\ge0.
\end{equation}
By \eqref{eq:full-doubling-average},
\[
 D(s)=\frac1{\log2}\int_{s/2}^sN(\rho)\frac{\dd\rho}{\rho}.
\]
The hypotheses therefore imply, for $1/2\le s\le1$,
\[
 m-\delta\le D(1/2)\le N(1/2)\le N(s)
 \le N(1)\le D(2)\le m+\delta.
\]
Since $k-m\in\mathbb Z$,
\[
 \sum_k(k-m)^2a_k^2s^{2k}
 \ge\left|\sum_k(k-m)a_k^2s^{2k}\right|=|N-m|H.
\]
Consequently,
\begin{equation}\label{eq:turning-integer-gap}
\begin{aligned}
 \frac{sN'}2
 &=\frac{\sum_k(k-m)^2a_k^2s^{2k}}{H}-(N-m)^2\\
 &\ge |N-m|-(N-m)^2
 \ge(1-\delta)|N-m|.
\end{aligned}
\end{equation}
Within this proof, put
\[
 P_m(s)=\frac{\Pi_m[v(s\,\cdot,-s^2)]}{\sqrt{H(s)}}
 =\frac{a_ms^m}{\sqrt{H(s)}}p_m(\cdot,-1).
\]
Then $\|P_m(s)\|\le1$ and
\[
 P_m'(s)=\frac{m-N(s)}{s}P_m(s),\qquad
 \|P_m'(s)\|\le\frac{|m-N(s)|}{s}
 \le\frac{N'(s)}{2(1-\delta)}.
\]
Integrating over $[1/2,1]$ gives
\begin{equation}\label{eq:turning-integral-bound}
\begin{aligned}
 \|P_m(1)-P_m(1/2)\|
 &\le\int_{1/2}^1\|P_m'(s)\|\,ds\\
 &\le\frac{N(1)-N(1/2)}{2(1-\delta)}
 \le\frac{D(2)-D(1/2)}{2(1-\delta)}.
\end{aligned}
\end{equation}
This proves \eqref{eq:caloric-projection-turning}.
\end{proof}

\begin{lemma}
\label{lem:fc-approximation}
Let $u_\eps$ solve \eqref{eq:homogenization-equation} under
\textup{(A1)--(A4)} in $Q^-(\mathbf0,4)$ and assume
\[
 \Theta(u_\eps,\mathbf0,s)\le\Theta_1
 \qquad(0<s\le1).
\]
Set
\[
 a_r^2=r^{-n}\int_{B_r}|u_\eps(x,0)|^2\,dx.
\]
There exist $\kappa_*\ge8$ and $\eta_*,c,C>0$, depending only on $\Theta_1$ and
structural constants, such that if
\[
 0<r\le1/\kappa_*,\qquad 0<\eps/r\le\eta_*,\qquad a_r>0,
\]
there is a full-space caloric function $v_r$ on
$\mathbb R^n\times(-16r^2,0]$ with polynomial growth in $x$
uniformly in $t$. For $1/4\le q\le2$, we have
\begin{equation}\label{eq:fc-approximation-bound}
 \left\|\vartheta u_\eps(\cdot,-q^2r^2)-v_r(\cdot,-q^2r^2)
 \right\|_{L^2(\nu_{0,q^2r^2})}
 \le Ca_r\left(\frac{\eps}{r}+e^{-c/r^2}\right).
\end{equation}
Moreover,
\begin{equation}\label{eq:fc-mass-bounds}
 c a_r^2\le H_1(u_\eps;\mathbf0,qr),\,
 H(v_r;\mathbf0,qr)\le Ca_r^2.
\end{equation}
Consequently, for $1/2\le q\le2$,
\begin{equation}\label{eq:fc-doubling-error}
 |D_1(u_\eps;\mathbf0,qr)-D(v_r;\mathbf0,qr)|
 \le C\left(\frac{\eps}{r}+e^{-c/r^2}\right).
\end{equation}
\end{lemma}

\begin{proof}
Put $\kappa=1/r$, $\eta=\eps/r$ and
$w(y,s)=a_r^{-1}u_\eps(ry,r^2s)$. By \eqref{eq:truncated-scaling}, $H_\kappa(w;\mathbf0,q)
=a_r^{-2}H_1(u_\eps;\mathbf0,qr)$. Moreover,
\[
 \int_{B_1}w^2(y,0)\,dy=1,
 \quad \Theta(w,\mathbf0,s)\le\Theta_1\ (0<s\le \kappa).
\]
The rescaled function satisfies
\begin{equation}\label{eq:fc-rescaled-equation}
 \partial_s w-\operatorname{div}_y\!\left[
 A\!\left(\frac y{\eta},
 \frac s{\eta^{\,2}}\right)\nabla_yw
 \right]=0
 \quad\text{in }Q^-(\mathbf0,4\kappa).
\end{equation}
Let $d=\max\{0,(\beta_0-n)/2\}$. Applying \eqref{eq:polynomial-spatial-growth} to $w$, using $\Theta(w,\mathbf0,s)\le\Theta_1$ for $0<s\le\kappa$, gives, for $1\le s\le\kappa$,
\begin{equation}
 \int_{B_s}w^2(y,0)\,dy\le C_0s^{\beta_0},
 \qquad
 \int_{Q^-(\mathbf0,4s)}w^2\le Cs^{\beta_0+2}.
\end{equation}
Interior estimates and \cite[Theorem~2.6]{GS20} give
\begin{equation}
 \|w\|_{L^\infty(Q^-(\mathbf0,3s))}
 +s\|\nabla w\|_{L^\infty(Q^-(\mathbf0,3s))}
 \le Cs^d\qquad(1\le s\le\kappa).
\end{equation}
The constant is independent of $\eta$.
Since $\kappa\ge8$, the preceding estimate applies on
$B_{2\kappa}\times[-16,0]$. Thus, for $-16\le s\le0$,
\begin{alignat}{2}
 &|w(y,s)|\le C(1+|y|)^d,
 &\qquad &|y|\le2\kappa,\label{eq:fc-envelope}\\
 &|w(y,s)|+\kappa|\nabla w(y,s)|\le C\kappa^d,
 &\qquad &\kappa\le|y|\le2\kappa.
\end{alignat}
Write $\vartheta_\kappa(y)=\vartheta(y/\kappa)$, extend $W=\vartheta_\kappa w$ by zero,
and set $f=W(\cdot,-16)$. Thus
\[
 |f(y)|\le C(1+|y|)^d\quad(y\in\mathbb R^n),
\]
with $C$ independent of $\kappa$. Let
$\Gamma_{\eta}(x,t;y,s)$ denote the whole-space fundamental
solution of \eqref{eq:fc-rescaled-equation}. Define
\[
 V_{\eta}(x,t)=\int_{\mathbb R^n}\Gamma_{\eta}(x,t;y,-16)f(y)\,dy,
 \qquad
 V(x,t)=\int G_{x,t+16}(y)f(y)\,dy.
\]
For $-4\le t\le0$, \cite[Theorem~1.1]{GS20} gives
\begin{equation}\label{eq:fc-kernel-error}
\begin{aligned}
 |V_{\eta}(x,t)-V(x,t)|
 &\le C\eta\int e^{-c|x-y|^2}(1+|y|)^d\,dy\\
 &\le C\eta(1+|x|)^d.
\end{aligned}
\end{equation}
Moreover, the heat-kernel representation and the polynomial growth
bound for $f$ give
\begin{equation}\label{eq:fc-V-envelope}
 |V(x,t)|\le C(1+|x|)^d,
 \qquad x\in\mathbb R^n,\quad -16<t\le0.
\end{equation}

A direct product computation gives
\[
 \partial_s W-\operatorname{div}_y\!\left[
 A\!\left(\frac y{\eta},
 \frac s{\eta^{\,2}}\right)\nabla_yW
 \right]=F+\operatorname{div}_yG.
\]
Here
\[
 F=-\nabla\vartheta_\kappa\cdot A^{\eta}\nabla w,
 \quad G=-wA^{\eta}\nabla\vartheta_\kappa,
\]
where
\[
 \operatorname{supp}(F,G)\subset E_\kappa=\{\kappa\le|y|\le2\kappa\},
 \quad |F|\le C\kappa^{d-2},\quad |G|\le C\kappa^{d-1}.
\]
Duhamel's formula reads
\[
 W(x,t)-V_{\eta}(x,t)=\int_{-16}^t\!\int
 \bigl[\Gamma_{\eta}(x,t;y,s)F(y,s)
 -\nabla_y\Gamma_{\eta}(x,t;y,s)\cdot G(y,s)\bigr]\,dy\,ds.
\]
The Gaussian estimates for $\Gamma_{\eta}$ and
$\nabla_y\Gamma_{\eta}$ \cite[(1.6), Theorem~2.7]{GS20}
imply, for $|x|\le \kappa/2$, $-4\le t\le0$,
\begin{equation}\begin{aligned}
 |W(x,t)-V_{\eta}(x,t)|
 &\le C\int_0^{16}\int_{E_\kappa}
 \left(\kappa^{d-2}\tau^{-n/2}+\kappa^{d-1}\tau^{-(n+1)/2}\right)
 e^{-c|x-y|^2/\tau}\,dy\,d\tau\\
 &\le C \kappa^B e^{-c\kappa^2}\le C e^{-c'\kappa^2}.
\end{aligned}\label{eq:fc-boundary-error}\end{equation}
Here $B$ depends only on $n,d$; the last inequality decreases $c'$.
Combining \eqref{eq:fc-kernel-error}--\eqref{eq:fc-boundary-error} on
$\{|x|\le\kappa/2\}$,
and using \eqref{eq:fc-envelope}, \eqref{eq:fc-V-envelope} on
$\{|x|>\kappa/2\}$, we obtain
\begin{equation}\begin{aligned}
 &\sup_{1/4\le q\le2}
 \|W(\cdot,-q^2)-V(\cdot,-q^2)\|_{L^2(\nu_{0,q^2})}
 +\|w(\cdot,0)-V(\cdot,0)\|_{L^\infty(B_1)}\\
 &\hspace{35mm}\le C\left(\eta+e^{-c\kappa^2}\right).
\end{aligned}\label{eq:fc-profile-error}\end{equation}
Set
\[
 v_r(x,t)=a_rV(x/r,t/r^2).
\]
Then \eqref{eq:fc-approximation-bound} follows by scaling.

Choose $\kappa_*$ large and $\eta_*$ small. Then
\[
 \|V(\cdot,0)\|_{L^2(B_1)}\ge\tfrac12.
\]
For $1/4\le q\le2$, the heat semigroup and Cauchy--Schwarz give
\begin{equation}\begin{aligned}
 |V(x,0)|^2
 &=\left|\int V(y,-q^2)G_{x,q^2}(y)\,dy\right|^2\\
 &\le H(V;\mathbf0,q)
 \int\frac{G_{x,q^2}(y)^2}{G_{0,q^2}(y)}\,dy
 =e^{|x|^2/(2q^2)}H(V;\mathbf0,q).
\end{aligned}\end{equation}
Integrating over $B_1$ and using
$\|V(\cdot,0)\|_{L^2(B_1)}\ge1/2$ gives a uniform positive lower
bound for $H(V;\mathbf0,q)$. The upper bound follows from
\eqref{eq:fc-V-envelope}. Thus
\[
 c\le H(V;\mathbf0,q)\le C,\qquad 1/4\le q\le2.
\]
By \eqref{eq:fc-profile-error},
\[
 \left|\|W(\cdot,-q^2)\|_{L^2(\nu_{0,q^2})}
 -H(V;\mathbf0,q)^{1/2}\right|
 \le C\left(\eta+e^{-c\kappa^2}\right).
\]
Decreasing $\eta_*$ and increasing $\kappa_*$ if necessary, we obtain
\[
 c\le\|W(\cdot,-q^2)\|_{L^2(\nu_{0,q^2})}^2
 =H_\kappa(w;\mathbf0,q)\le C.
\]
Scaling gives \eqref{eq:fc-mass-bounds}. The norm upper bounds and
\eqref{eq:fc-profile-error} also give
\[
 |H_\kappa(w;\mathbf0,q)-H(V;\mathbf0,q)|
 \le C\left(\eta+e^{-c\kappa^2}\right).
\]
For $1/2\le q\le2$, apply the mass comparison at $q$ and $q/2$.
The uniform lower bounds and
$|\log a-\log b|\le c^{-1}|a-b|$ for $a,b\ge c$ give
\[
 |D_\kappa(w;\mathbf0,q)-D(V;\mathbf0,q)|
 \le C\left(\eta+e^{-c\kappa^2}\right),
 \qquad 1/2\le q\le2.
\]
Scaling proves \eqref{eq:fc-doubling-error}.
\end{proof}

\begin{theorem}\label{prop:fc-turning}
Let $u_\eps$ solve
\eqref{eq:homogenization-equation} under \textup{(A1)--(A4)} in
$Q^-(\mathbf0,4)$. Assume
\[
 \Theta(u_\eps,\mathbf0,s)\le\Theta_1<\infty
 \qquad(0<s\le1).
\]
Fix $m\in\mathbb N_0$ and $0<\delta_0\le1/16$. There exist
$\kappa_*=\kappa_*(\delta_0,\Theta_1)\ge8$ and
$\eta_*=\eta_*(\delta_0,\Theta_1)>0$ such that the following holds.
Let $J\in\mathbb N_0$ and
\begin{equation}\label{eq:turning-scale-hypotheses}
 r_j=2^{-j}r_0\quad(0\le j\le J+1),\qquad
 0<r_0\le1/\kappa_*,\qquad
 0<\eps/r_J\le\eta_*.
\end{equation}
Suppose
\begin{equation}\label{eq:turning-pinching}
 D_1(u_\eps;\mathbf0,2r_j),\,
 D_1(u_\eps;\mathbf0,r_j/2)
 \in[m-\delta_0,m+\delta_0]\qquad(0\le j\le J).
\end{equation}
Set
\begin{equation}\label{eq:turning-projected-profile}
 \mathfrak p_m(r)=
 \frac{\Pi_m[\vartheta(r\,\cdot)u_\eps(r\,\cdot,-r^2)]}
 {\sqrt{H_1(u_\eps;\mathbf0,r)}},
 \qquad D_1(r)=D_1(u_\eps;\mathbf0,r).
\end{equation}
With $\|\cdot\|=\|\cdot\|_{L^2(\nu_{0,1})}$, we have
\begin{equation}\label{eq:turning-endpoints}
\begin{aligned}
 \sum_{j=0}^{J}\|\mathfrak p_m(r_j)-\mathfrak p_m(r_{j+1})\|
 &\le \frac47\{D_1(2r_0)+D_1(r_0)
       -D_1(r_J)-D_1(r_{J+1})\}\\
 &\quad+C\{\eps/r_J+e^{-c'/r_0^2}\}.
\end{aligned}
\end{equation}
Here $c',C>0$ depend only on $\Theta_1$ and structural constants.
In particular,
\begin{equation}\label{eq:turning-total}
 \sum_{j=0}^{J}\|\mathfrak p_m(r_j)-\mathfrak p_m(r_{j+1})\|
 \le\frac{16}{7}\delta_0
 +C\{\eps/r_J+e^{-c'/r_0^2}\}.
\end{equation}
\end{theorem}

\begin{proof}
Let $c_1>0$ and $C_1\ge1$ be constants in
\eqref{eq:fc-doubling-error}. Decrease the small threshold $\eta_*$
and increase the large threshold $\kappa_*$ from
Lemma~\ref{lem:fc-approximation} so that
\begin{equation}\label{eq:turning-threshold-choice}
 0<\eta_*\le\frac{\delta_0}{2C_1},\qquad
 \kappa_*\ge\sqrt{c_1^{-1}\log(2C_1/\delta_0)}.
\end{equation}
Then, for $0\le j\le J$,
\begin{equation}\label{eq:turning-uniform-error}
\begin{gathered}
 r_j\le r_0\le1/\kappa_*,\qquad
 \eps/r_j\le\eps/r_J\le\eta_*,\\
 C_1(\eps/r_j+e^{-c_1/r_j^2})
 \le C_1(\eta_*+e^{-c_1\kappa_*^2})\le\delta_0.
\end{gathered}
\end{equation}
Let $v_j$ be the caloric approximation constructed in
Lemma~\ref{lem:fc-approximation} at scale $r_j$, and put
\[
 \mathfrak q_{j,m}(q)=
 \frac{\Pi_m[v_j(qr_j\,\cdot,-q^2r_j^2)]}
 {\sqrt{H(v_j;\mathbf0,qr_j)}}.
\]
Equations~\eqref{eq:fc-approximation-bound} and \eqref{eq:fc-mass-bounds}, followed by
normalization and orthogonal projection, give
\begin{equation}\label{eq:turning-projection-comparison}
 \|\mathfrak p_m(qr_j)-\mathfrak q_{j,m}(q)\|
 \le C(\eps/r_j+e^{-c/r_j^2})\qquad(q=1,1/2).
\end{equation}
By \eqref{eq:fc-doubling-error} and
\eqref{eq:turning-uniform-error},
\begin{equation}\label{eq:turning-doubling-comparison}
 |D(v_j;\mathbf0,qr_j)-D_1(qr_j)|
 \le C_1(\eps/r_j+e^{-c_1/r_j^2})\le\delta_0
 \qquad(q=2,1/2).
\end{equation}
Thus both caloric doubling indices belong to
$[m-2\delta_0,m+2\delta_0]\subset[m-1/8,m+1/8]$.
Apply Lemma~\ref{lem:fc-caloric-turning} to
$T_{\mathbf0,r_j}v_j$ with $\delta=1/8$. Since
$1/[2(1-\delta)]=4/7$, \eqref{eq:turning-projection-comparison}, \eqref{eq:turning-doubling-comparison}, and the triangle inequality yield
\begin{equation}\label{eq:turning-one-step}
\begin{aligned}
 \|\mathfrak p_m(r_j)-\mathfrak p_m(r_j/2)\|
 &\le\|\mathfrak q_{j,m}(1)-\mathfrak q_{j,m}(1/2)\|
       +C(\eps/r_j+e^{-c/r_j^2})\\
 &\le\frac47\{D_1(2r_j)-D_1(r_j/2)\}
       +C(\eps/r_j+e^{-c/r_j^2}).
\end{aligned}
\end{equation}

Sum \eqref{eq:turning-one-step} and use
\begin{equation}\label{eq:turning-cancellation}
\begin{aligned}
 \sum_{j=0}^J\{D_1(2r_j)-D_1(r_j/2)\}
 &=D_1(2r_0)+D_1(r_0)\\
 &\quad-D_1(r_J)-D_1(r_{J+1}),
\end{aligned}
\end{equation}
together with
\begin{equation}\label{eq:turning-error-sums}
 \sum_{j=0}^J\eps/r_j\le2\eps/r_J,\qquad
 \sum_{j=0}^Je^{-c/r_j^2}
 =\sum_{j=0}^Je^{-c4^j/r_0^2}\le Ce^{-c'/r_0^2}.
\end{equation}
This proves \eqref{eq:turning-endpoints}. By
\eqref{eq:turning-pinching}, the endpoint combination is at most
$4\delta_0$. Hence \eqref{eq:turning-total} follows.
\end{proof}

\section{Geometry of homogeneous caloric polynomials}\label{sec:geometry}

A parabolic linear subspace is of the form $W\times\{0\}$ or
$W\times\R$, where $W$ is a linear subspace of $\R^n$; its parabolic
dimension $\dim_P$ is $\dim W$ or $\dim W+2$, respectively.

For a homogeneous caloric polynomial $P$, write
\[
 \|P\|_G^2=\int_{\R^n}P^2(y,-1)\,d\nu_{0,1}(y),\qquad
 \mathcal I_x(P)=\{a\in\R^n:a\cdot\nabla_xP\equiv0\}.
\]
The finite-dimensional space of degree-$m$ homogeneous caloric
polynomials is denoted by $\mathcal P_m$.

\begin{lemma}\label{lem:parabolic-spine}
Let $0\ne P\in\mathcal P_m$, $m\ge2$, and define
\[
 \mathcal I(P)=\{(a,b):P(x+a,t+b)=P(x,t)\text{ for all }(x,t)\}.
\]
Then the following statements hold.
\begin{enumerate}[label=\textup{(\roman*)},leftmargin=3em]
\item The invariant spaces and their spatial dimensions satisfy
\begin{equation}\label{eq:spine-classification}
\begin{alignedat}{2}
 \partial_tP\not\equiv0:\quad&
 \mathcal I(P)=\mathcal I_x(P)\times\{0\},
 &\qquad&\dim\mathcal I_x(P)\le n-1,\\
 \partial_tP\equiv0:\quad&
 \mathcal I(P)=\mathcal I_x(P)\times\R,
 &\qquad&\dim\mathcal I_x(P)\le n-2.
\end{alignedat}
\end{equation}
\item We have
\[
 \dim_P\mathcal I(P)\le n.
\]
Equality requires $P$ to be a stationary harmonic polynomial depending
on exactly two spatial variables.
\item The polynomial $P$ is homogeneous of degree $m$ about
$\mathbf z$ if and only if $\mathbf z\in\mathcal I(P)$.
\end{enumerate}
\end{lemma}

\begin{proof}
We first prove \textup{(i)}. If $(a,b)\in\mathcal I(P)$, then, for
fixed $(x,t)$, the polynomial $q\mapsto P(x+qa,t+qb)$ has the same
value at every integer. It is therefore constant, and
\[
 (a\cdot\nabla_x+b\partial_t)P=0.
\]
The two terms have parabolic degrees $m-1$ and $m-2$, respectively;
therefore
\[
 a\cdot\nabla_xP=0,\qquad b\partial_tP=0.
\]
Conversely, these identities imply translation invariance by
integration along the line in direction $(a,b)$. This proves the two
formulas for $\mathcal I(P)$ in \eqref{eq:spine-classification}.
If $\partial_tP\not\equiv0$ and $\mathcal I_x(P)=\R^n$, then $P$ is
independent of $x$, and the heat equation gives $\partial_tP\equiv0$, a
contradiction. Thus $\dim\mathcal I_x(P)\le n-1$. If
$\partial_tP\equiv0$ and $\dim\mathcal I_x(P)\ge n-1$, then either $P$
is constant or
\[
 P(x)=c(e\cdot x)^m
\]
for some unit vector $e$. Since $P$ is harmonic and $m\ge2$, this
forces $c=0$, again a contradiction. Hence
$\dim\mathcal I_x(P)\le n-2$.

For \textup{(ii)}, by \textup{(i)}, if
$\partial_tP\not\equiv0$, then
\[
 \dim_P\mathcal I(P)=\dim\mathcal I_x(P)\le n-1.
\]
If $\partial_tP\equiv0$, then
\[
 \dim_P\mathcal I(P)=\dim\mathcal I_x(P)+2\le n.
\]
Equality can occur only in the second case with
$\dim\mathcal I_x(P)=n-2$, which is precisely the stated stationary
harmonic case with two essential spatial variables. This proves
\textup{(ii)}.

For \textup{(iii)}, suppose that $P$ is homogeneous about
$\mathbf z=(a,b)$. The Euler identities about $\mathbf0$ and
$\mathbf z$ are
\[
 (x\cdot\nabla_x+2t\partial_t)P=mP,
 \qquad ((x-a)\cdot\nabla_x+2(t-b)\partial_t)P=mP.
\]
Their difference gives
\[
 a\cdot\nabla_xP+2b\partial_tP=0.
\]
Again the two homogeneous terms vanish, and \textup{(i)} gives
$\mathbf z\in\mathcal I(P)$. The converse is trivial.
\end{proof}

\begin{lemma}\label{lem:cylinder-projection}
Suppose $n\ge2$, $2\le m\le L$, and $W\subset\R^n$ has dimension
$n-2$. For an orthonormal basis $\{e_i\}_{i=1}^{n-2}$ of $W$, set
\[
 \nabla_WP:=\sum_{i=1}^{n-2}(\partial_{e_i}P)e_i,
 \qquad
 \|\nabla_WP\|_G^2
 =\sum_{i=1}^{n-2}\|\partial_{e_i}P\|_G^2.
\]
For every $P\in\mathcal P_m$ there is a decomposition
$P=\phi+\psi$, orthogonal for $\|\cdot\|_G$, such that
\begin{equation}\label{eq:cylinder-kernel}
 \partial_t\phi=0,\qquad \nabla_W\phi=0.
\end{equation}
For every $\eta>0$, this decomposition satisfies
\begin{equation}\label{eq:cylinder-distance}
 \|\partial_tP\|_G+\|\nabla_WP\|_G
 \le\eta\|P\|_G
 \quad\Longrightarrow\quad
 \|\psi\|_G\le C(n,L)\eta\|P\|_G.
\end{equation}
Moreover,
\begin{equation}\label{eq:cylinder-gradient}
 |\nabla\phi(x)|
 =\frac{m}{2^{m-1/2}\sqrt{m!}}\|\phi\|_G
       \operatorname{dist}(x,W)^{m-1}.
\end{equation}
\end{lemma}

\begin{proof}
Define the linear map
\[
 T_WP=(\partial_tP,\nabla_WP).
\]
Its kernel consists precisely of the stationary degree-$m$
homogeneous caloric polynomials that are invariant along $W$.
Let $\phi$ be the Gaussian orthogonal projection of $P$ onto
$\ker T_W$, and set
\[
 \psi=P-\phi\in(\ker T_W)^\perp.
\]
Then $P=\phi+\psi$ is an orthogonal decomposition and
$T_W\phi=0$, which proves \eqref{eq:cylinder-kernel}.

For fixed $m$, work in orthonormal coordinates with
$W=\R^{n-2}\times\{0\}$. The restriction
\[
 T_W\big|_{(\ker T_W)^\perp}
\]
is injective. Since the spaces are finite dimensional, its smallest
singular value is positive. Taking the minimum over the finitely many degrees
$2\le m\le L$ gives $c(n,L)>0$, independent of $m$ and $W$, such that
\[
 c(n,L)\|\psi\|_G
 \le \|\partial_t\psi\|_G+\|\nabla_W\psi\|_G
 =\|\partial_tP\|_G+\|\nabla_WP\|_G.
\]
The hypothesis in \eqref{eq:cylinder-distance} now gives the asserted
bound for $\|\psi\|_G$.

Finally, $W^\perp$ is two dimensional. By
\eqref{eq:cylinder-kernel}, $\phi$ is an $m$-homogeneous harmonic
polynomial on $W^\perp$. In orthonormal polar coordinates
$(r,\theta)$ on this plane,
\[
 \phi=r^m(A\cos m\theta+B\sin m\theta),\qquad
 |\nabla\phi|^2=m^2(A^2+B^2)r^{2m-2},
\]
\[
 \|\phi\|_G^2
 =\frac{A^2+B^2}{4}\int_0^\infty r^{2m+1}e^{-r^2/4}\,dr
 =2^{2m-1}m!(A^2+B^2).
\]
Since $r=\operatorname{dist}(x,W)$, substitution gives
\eqref{eq:cylinder-gradient}.
\end{proof}

\begin{lemma}\label{lem:geometric-dichotomy}
Fix $n\ge2$, $L\ge2$, $0<\gamma<1/4$, and $\tau>0$.
There exists $\eta=\eta(n,L,\gamma,\tau)>0$ with the following
property. For every $P\in\mathcal P_m$, $2\le m\le L$, $\|P\|_G=1$,
there is a parabolic plane $V$ such that
\begin{equation}\label{eq:approximate-directions}
 (a,b)\in\partial Q(\mathbf0,1),\quad
 \|a\cdot\nabla P\|_G+|b|\|\partial_tP\|_G\le\eta
 \quad\Longrightarrow\quad d_P((a,b),V)\le\gamma.
\end{equation}
In addition, one of the following holds:
\begin{enumerate}[label=\textup{(\roman*)},leftmargin=3em]
 \item $\dim_PV\le n-1$;
 \item $V=W\times\R$, $\dim W=n-2$, and the decomposition
 $P=\phi+\psi$ in Lemma~\ref{lem:cylinder-projection} satisfies
 $\|\psi\|_G\le\tau$.
\end{enumerate}
\end{lemma}

\begin{proof}
Let
\[
 \mathcal{G}_P:=(\langle\partial_iP,\partial_jP\rangle_G)_{i,j=1}^n.
\]
This matrix is symmetric and nonnegative, since
\[
 a^T\mathcal{G}_Pa=\|a\cdot\nabla P\|_G^2\ge0.
\]
Choose an orthonormal eigenbasis $e_1,\ldots,e_n$, with corresponding
eigenvalues $\lambda_1\ge\cdots\ge\lambda_n\ge0$, and write
$a_i=a\cdot e_i$. Then
\[
 \|a\cdot\nabla P\|_G^2=\sum_{i=1}^n\lambda_i a_i^2.
\]
Gaussian integration by parts and $\|P\|_G=1$ give
\[
 \operatorname{tr}\mathcal{G}_P=\|\nabla P\|_G^2=m/2,
 \qquad\lambda_1\ge\frac{m}{2n}\ge\frac1n.
\]

We claim that there exist $c_*,\sigma_*>0$, depending only on $n,L$, such that
\begin{equation}\label{eq:second-eigenvalue-gap}
 \|\partial_tP\|_G\le\sigma_*
 \quad\Longrightarrow\quad\lambda_2\ge c_*.
\end{equation}
Indeed, otherwise finite-dimensional compactness, after fixing the
degree along a subsequence, gives a limit $P_\infty$ with
\[
 \|P_\infty\|_G=1,\qquad \partial_tP_\infty=0,
 \qquad\operatorname{rank}\mathcal{G}_{P_\infty}\le1.
\]
Thus $\dim\mathcal I_x(P_\infty)\ge n-1$, contradicting
Lemma~\ref{lem:parabolic-spine}\textup{(i)}.

For every $(a,b)$ satisfying the hypotheses of
\eqref{eq:approximate-directions},
\[
 \sum_{i=1}^n\lambda_i a_i^2\le\eta^2,
 \qquad |b|\|\partial_tP\|_G\le\eta.
\]

For $\sigma>0$ to be determined, we consider the following cases.

\emph{Case 1: $\|\partial_tP\|_G\ge\sigma$.}\par
\noindent
Take
$V=e_1^\perp\times\{0\}$, so $\dim_PV=n-1$. Since
\[
 |a_1|\le\eta/\sqrt{\lambda_1}\le\sqrt n\eta,
 \qquad |b|\le\eta/\sigma,
\]
we obtain
\[
 d_P((a,b),V)=\max\{|a_1|,\sqrt{|b|}\}
 \le\max\{\sqrt n\eta,\sqrt{\eta/\sigma}\}\le\gamma.
\]
Here we require $\eta\le\min\{\gamma/\sqrt n,\sigma\gamma^2\}$.
This proves alternative \textup{(i)}.

\emph{Case 2: $\|\partial_tP\|_G<\sigma$, $n\ge3$, and
$\lambda_3\ge\sigma^2$.}\par
\noindent
Take $V=\operatorname{span}\{e_4,\ldots,e_n\}\times\R$.
Then $\dim_PV=(n-3)+2=n-1$, and
\[
 d_P((a,b),V)^2
 =\sum_{i=1}^3a_i^2
 \le\sigma^{-2}\sum_{i=1}^3\lambda_i a_i^2
 \le\eta^2/\sigma^2\le\gamma^2.
\]
Choosing $\eta\le\sigma\gamma$ proves alternative \textup{(i)}.

\emph{Case 3: $\|\partial_tP\|_G<\sigma$ and either $n=2$ or
$\lambda_3<\sigma^2$.}\par
\noindent
Choose $0<\sigma\le\sigma_*$, depending only on $n,L,\tau$, so that
\[
 C(n,L)(1+\sqrt{n-2})\sigma\le\tau,
\]
where $C(n,L)$ is the constant in \eqref{eq:cylinder-distance}.
Set $W=\operatorname{span}\{e_3,\ldots,e_n\}$ and
$V=W\times\R$, with $W=\{0\}$ when $n=2$.
By \eqref{eq:second-eigenvalue-gap},
\[
 d_P((a,b),V)^2=a_1^2+a_2^2
 \le c_*^{-1}\sum_{i=1}^2\lambda_i a_i^2
 \le\eta^2/c_*\le\gamma^2.
\]
Here we require $\eta\le\gamma\sqrt{c_*}$. Moreover,
\[
 \|\nabla_WP\|_G^2=\sum_{i=3}^n\lambda_i
 \le(n-2)\sigma^2,
 \qquad
 \|\partial_tP\|_G+\|\nabla_WP\|_G
 \le(1+\sqrt{n-2})\sigma\|P\|_G.
\]
Lemma~\ref{lem:cylinder-projection} gives the decomposition
$P=\phi+\psi$, with $\partial_t\phi=0$, $\nabla_W\phi=0$, and
\[
 \|\psi\|_G\le C(n,L)(1+\sqrt{n-2})\sigma\le\tau.
\]
This proves alternative \textup{(ii)}. With $\sigma$ fixed as in
Case 3, choose $\eta>0$ to satisfy all three cases. These choices
depend only on $n,L,\gamma,\tau$ and complete the proof.
\end{proof}

\section{Comparison of centers and localization of singular points}
\label{sec:multicenter-geometry}

We first establish forward growth and two-sided homogeneous approximation,
then compare centers and localize singular points using the polynomial
geometry of Section~\ref{sec:geometry}.

For a degree-$m$ Hermite polynomial $f$, let $\mathcal C_m f$ be
its homogeneous caloric extension, characterized by
$(\mathcal C_m f)(y,-1)=f(y)$. For $\mathbf z=(x_0,t_0)$, set
\[
 \Psi_{\mathbf z,s}
 :=\Pi_m\!\left[\vartheta(sy)u_\eps(x_0+sy,t_0-s^2)\right],
 \qquad
 \psi_{\mathbf z,s}
 :=\mathcal C_m\!\left(
 \frac{\Psi_{\mathbf z,s}}{\|\Psi_{\mathbf z,s}\|_{L^2(\nu_{0,1})}}
 \right).
\]
The denominator is positive under the pinching assumptions below.

\begin{lemma}\label{lem:forward-bound}
Let $u_\eps$ solve \eqref{eq:homogenization-equation} under
\textup{(A1)--(A4)} in $Q(\mathbf0,6)$ and satisfy
\eqref{eq:uniform-theta}. Set $d=\max\{0,(\beta_0(\Theta_1)-n)/2\}$.
There exist $\kappa_*\ge8$ and $\eta_*,C>0$, depending only on
$\Theta_1$ and structural constants, such that if
\[
 \mathbf z=(x_0,t_0)\in Q(\mathbf0,1),\qquad
 0<s\le1/\kappa_*,\qquad \eps/s\le\eta_*,
\]
then the normalized rescaling
\[
 w=\frac{T_{\mathbf z,s}u_\eps}{H_1(u_\eps;\mathbf z,s)^{1/2}}
\]
satisfies
\[
 |w(y,\tau)|\le C(1+|y|)^d
 \qquad(|y|\le2/s,\ 0\le\tau\le4).
\]
\end{lemma}

\begin{proof}
Put $T=4$ and choose the thresholds so that
Lemma~\ref{lem:fc-approximation} applies at scale $s$. Set
\[
 a_s^2=s^{-n}\int_{B(x_0,s)}u_\eps^2(x,t_0)\,dx,\qquad
 \kappa=1/s,\qquad \widehat w=a_s^{-1}T_{\mathbf z,s}u_\eps,
\]
and write the rescaled coefficient as
\[
 \widetilde A(y,\tau)
 =A\!\left(\frac{x_0+sy}{\eps},
             \frac{t_0+s^2\tau}{\eps^2}\right).
\]
The uniform hypothesis \eqref{eq:uniform-theta}, applied at the shifted
time centers, gives
\[
 \Theta(\widehat w,(0,\tau),q)\le\Theta_1
 \qquad(0\le\tau\le T,\ 0<q\le \kappa).
\]
Define
\[
 F(\tau)=\int_{B_1}\widehat w^2(y,\tau)\,dy,\qquad
 M_T=\max_{0\le\tau\le T}F(\tau).
\]
Then $F(0)=1$. Applying \eqref{eq:polynomial-spatial-growth} after translation to $(0,\tau)$, and then the interior $L^2$-to-$L^\infty$ estimate, yields
\[
 |\widehat w(y,\tau)|\le C(1+|y|)^dF(\tau)^{1/2}
 \qquad(|y|\le2\kappa,\ 0\le\tau\le T),
\]
and \cite[Theorem~2.6]{GS20} gives, on the cutoff annulus,
\begin{equation}\label{eq:two-sided-annular-bound}
 |\widehat w(y,\tau)|+\kappa|\nabla\widehat w(y,\tau)|
 \le C\kappa^dF(\tau)^{1/2}
 \quad(\kappa\le|y|\le2\kappa,\ 0\le\tau\le T).
\end{equation}
Thus it suffices to bound $M_T$. Write
$\vartheta_\kappa(y)=\vartheta(y/\kappa)$ and $W=\vartheta_\kappa\widehat w$.
The cutoff equation is
\[
 \partial_\tau W-\operatorname{div}(\widetilde A\nabla W)
 =\mathcal F+\operatorname{div}\mathcal G,
\]
where
\[
 \mathcal F=-\nabla\vartheta_\kappa\cdot\widetilde A\nabla\widehat w,
 \qquad
 \mathcal G=-\widehat w\widetilde A\nabla\vartheta_\kappa.
\]
Both sources are supported in $E_\kappa=\{\kappa\le|y|\le2\kappa\}$, and
\eqref{eq:two-sided-annular-bound} gives
\[
 |\mathcal F(y,\tau)|\le C\kappa^{d-2}M_T^{1/2},\qquad
 |\mathcal G(y,\tau)|\le C\kappa^{d-1}M_T^{1/2}.
\]
Let $\Gamma_{\mathbf z,s}$ be the whole-space fundamental solution for
$\widetilde A$, and put $f_0=W(\cdot,0)$.
Since $F(0)=1$, $|f_0(y)|\le C(1+|y|)^d$.
Duhamel's formula gives
\[
\begin{aligned}
 W(x,t)
 ={}&\int_{\R^n}\Gamma_{\mathbf z,s}(x,t;y,0)f_0(y)\,dy\\
 &+\int_0^t\int_{\R^n}
 \Gamma_{\mathbf z,s}(x,t;y,\tau)\mathcal F(y,\tau)\,dy\,d\tau\\
 &-\int_0^t\int_{\R^n}
 \nabla_y\Gamma_{\mathbf z,s}(x,t;y,\tau)
 \cdot\mathcal G(y,\tau)\,dy\,d\tau.
\end{aligned}
\]
For $x\in B_1$, the source annulus is at distance at least $\kappa-1$.
The Gaussian bounds for $\Gamma_{\mathbf z,s}$ and
$\nabla_y\Gamma_{\mathbf z,s}$ \cite[(1.6), Theorem~2.7]{GS20}
therefore give
\begin{align*}
 |W(x,t)|
 &\le\int_{\R^n}\Gamma_{\mathbf z,s}(x,t;y,0)|f_0(y)|\,dy
       +C\kappa^Be^{-c\kappa^2}M_T^{1/2}\\
 &\le C+Ce^{-c'\kappa^2}M_T^{1/2}
 \qquad(x\in B_1,\ 0<t\le T).
\end{align*}
Here the power $B$ is fixed by the dimension and the growth exponent.
Since $W=\widehat w$ on $B_1$, taking the $L^2(B_1)$ norm and the
supremum in time yields
\[
 M_T^{1/2}\le C+Ce^{-c'\kappa^2}M_T^{1/2}.
\]
Increasing $\kappa_*$ makes the last coefficient at most $1/2$,
so $M_T\le C$. Substitution into the pointwise growth estimate gives
\begin{equation}\label{eq:two-sided-growth}
 |\widehat w(y,\tau)|\le C(1+|y|)^d
 \quad(|y|\le2\kappa,\ 0\le\tau\le T).
\end{equation}
Finally, \eqref{eq:fc-mass-bounds}, applied at center $\mathbf z$ and scale $s$, gives
\[
 c a_s^2\le H_1(u_\eps;\mathbf z,s)\le C a_s^2,
 \qquad
 w=\frac{a_s}{H_1(u_\eps;\mathbf z,s)^{1/2}}\,\widehat w.
\]
The factor is uniformly bounded, so \eqref{eq:two-sided-growth}
proves the desired estimate.
\end{proof}

\begin{lemma}\label{lem:two-sided-pinching}
Let $u_\eps$ solve \eqref{eq:homogenization-equation} under
\textup{(A1)--(A4)} in $Q(\mathbf0,6)$ and satisfy
\eqref{eq:uniform-theta}. For every
$L\ge2$ and $\zeta>0$
there exist $\delta_*,\eta_*>0$ and $\kappa_*\ge2$ with the following
property. Suppose
\[
 \mathbf z=(x_0,t_0)\in Q(\mathbf0,1),\quad
 0<r\le1/\kappa_*,
 \quad\eps/r\le\eta_*,
\]
\[
 |D_1(u_\eps;\mathbf z,2^{-j}r)-m|\le\delta_*
 \qquad(0\le j\le3),
\]
where $2\le m\le L$. Set $s=r/4$ and
\[
 w=\frac{T_{\mathbf z,s}u_\eps}{H_1(u_\eps;\mathbf z,s)^{1/2}},
 \quad \psi=\psi_{\mathbf z,s},\quad
 J(y,t)=\mathcal J_A((x_0+sy)/\eps,(t_0+s^2t)/\eps^2).
\]
Then
\[
 \|w-\psi\|_{L^\infty(Q(\mathbf0,1))}
 +\|\nabla_yw-J\nabla_y\psi\|_{L^\infty(Q(\mathbf0,1))}
 \le\zeta.
\]
\end{lemma}

\begin{proof}
We argue by contradiction. Choose the thresholds so that Lemma~\ref{lem:forward-bound} applies at
$s=r/4$, noting that $\eps/s=4\eps/r$. Put $T=4$.
If the assertion failed, choose counterexamples with
$r_j\to0$, $\eps_j/r_j\to0$ and pinching errors $\delta_j\to0$,
fixing the degree $m$ along a subsequence. Put $s_j=r_j/4$ and let
$w_j$ be normalized as in the statement, so that
\[
 H_{1/s_j}(w_j;\mathbf0,1)=1.
\]
The mass comparison \eqref{eq:fc-mass-bounds} at scale $s_j$ gives
\[
 c\le\int_{B_1}|w_j(y,0)|^2\,dy
 =\frac{s_j^{-n}\displaystyle\int_{B(x_j,s_j)}
              |u_{\eps_j}(x,t_j)|^2\,dx}
        {H_1(u_{\eps_j};\mathbf z_j,s_j)}\le C,
 \qquad \mathbf z_j=(x_j,t_j).
\]
After normalizing this terminal $L^2$ mass, the argument giving
\eqref{eq:terminal-control}--\eqref{eq:polynomial-cylinder-bound} in
Lemma~\ref{lem:compactness} applies up to $R=1/s_j$, since here the
$\Theta$ bound holds at every scale in this range. Scaling back the bounded
normalizing factor yields
\begin{equation}\label{eq:two-sided-backward-growth}
 \int_{B_R}|w_j(y,0)|^2\,dy\le CR^{\beta_0},\qquad
 \sup_{Q^-(\mathbf0,3R)}|w_j|\le CR^d
 \quad(1\le R\le1/s_j).
\end{equation}
Here $d=\max\{0,(\beta_0-n)/2\}$. Lemma~\ref{lem:forward-bound}
provides the complementary estimate
\begin{equation}\label{eq:two-sided-forward-growth}
 |w_j(y,t)|\le C(1+|y|)^d
 \qquad(|y|\le2/s_j,\ 0\le t\le T).
\end{equation}
Interior H\"older estimates and homogenization compactness yield
\[
 w_j\to v\quad\text{locally uniformly on }\R^n\times(-\infty,T),
 \qquad \partial_tv-\Delta v=0.
\]
Equations~\eqref{eq:two-sided-backward-growth} and \eqref{eq:two-sided-forward-growth} control the Gaussian tails. Passing the
normalization and the pinching to the limit gives
\[
 H(v;\mathbf0,1)=1,\qquad
 D(v;\mathbf0,q)=m\quad(q=4,2,1,1/2).
\]
Lemma~\ref{lem:caloric-gaussian-doubling} in Appendix~\ref{sec:gaussian-appendix} shows that $v=P$ is an
$m$-homogeneous caloric polynomial for $t\le0$. Uniqueness for the
forward heat equation in the polynomial-growth class extends this
identity to $t<T$. In particular, $\|P\|_G=1$.

Local uniform convergence and the Gaussian tail bounds give
\[
 f_j:=\vartheta(s_j\,\cdot)w_j(\cdot,-1)
 \longrightarrow P(\cdot,-1)
 \quad\text{in }L^2(\nu_{0,1}).
\]
Since $P(\cdot,-1)$ belongs to the degree-$m$ Hermite space,
continuity of the orthogonal projection gives
\[
 \Pi_m f_j\longrightarrow P(\cdot,-1),\qquad
 \|\Pi_m f_j\|_{L^2(\nu_{0,1})}\longrightarrow\|P\|_G=1.
\]
In particular, these norms are nonzero for large $j$.
The positive amplitude factor $H_1(u_{\eps_j};\mathbf z_j,s_j)^{-1/2}$
cancels in the normalized projection. Hence the definition of
$\psi_j$ and continuity of the homogeneous caloric extension give
\[
 \psi_j=\mathcal C_m\!\left(
       \frac{\Pi_m f_j}{\|\Pi_m f_j\|_{L^2(\nu_{0,1})}}\right)
 \longrightarrow\mathcal C_m[P(\cdot,-1)]=P.
\]
By finite dimensionality, this convergence holds in every polynomial
norm, including the spatial derivative norms on fixed compact sets.

Since $\overline{Q(\mathbf0,3/2)}\Subset\R^n\times(-\infty,T)$,
local uniform convergence directly implies
\[
 \|w_j-P\|_{L^2(Q(\mathbf0,3/2))}\longrightarrow0.
\]
Cover $Q(\mathbf0,1)$ by finitely many backward cylinders of radius
$1/8$ whose doubled cylinders lie in $Q(\mathbf0,3/2)$.
Applying \cite[Theorem~6.1]{GS20} on this fixed cover gives, with
$\eta_j=\eps_j/s_j$,
\[
 \|\nabla w_j-J_j\nabla P\|_{L^\infty(Q(\mathbf0,1))}
 \le C\|w_j-P\|_{L^2(Q(\mathbf0,3/2))}
 +C_P\eta_j\log(2+\eta_j^{-1})\longrightarrow0.
\]
Since $J_j$ is uniformly bounded, the convergence of $w_j$ and $\psi_j$
to $P$ implies
\[
 \|w_j-\psi_j\|_{L^\infty(Q(\mathbf0,1))}
 +\|\nabla w_j-J_j\nabla\psi_j\|_{L^\infty(Q(\mathbf0,1))}
 \longrightarrow0.
\]
This contradicts the choice of counterexamples.
\end{proof}

\begin{lemma}\label{lem:two-center-qualitative}
Let $u_\eps$ solve \eqref{eq:homogenization-equation} under
\textup{(A1)--(A4)} in $Q(\mathbf0,6)$ and assume \eqref{eq:uniform-theta}. Fix
$L\ge2$ and $\zeta>0$.
There exist $\delta_*,\eta_*>0$ and $\kappa_*\ge2$, depending only on
$\Theta_1,L,\zeta$ and structural constants, with the following property.
Suppose
\[
 0<r\le1/\kappa_*,
 \quad \eps/r\le\eta_*,\quad
 \mathbf z_0,\mathbf z_1\in Q(\mathbf0,1),
 \quad d_P(\mathbf z_0,\mathbf z_1)\le r/4.
\]
If $2\le m\le L$ and
\[
 |D_1(u_\eps;\mathbf z_i,2^{-j}r)-m|\le\delta_*
 \qquad(i=0,1;\ j=0,1,2,3),
\]
then, writing $\psi_i=\psi_{\mathbf z_i,r/4}$,
\begin{align}
 &\|\psi_0-\psi_1\|_G\le\zeta,\\
 &\left\|\frac{x_1-x_0}{r}\cdot\nabla\psi_0\right\|_G
 +\frac{|t_1-t_0|}{r^2}\|\partial_t\psi_0\|_G\le\zeta.
\end{align}
\end{lemma}

\begin{proof}
We argue by contradiction. For fixed $L,\Theta_1,\zeta$, choose
counterexamples with $r_j\to0$, $\eps_j/r_j\to0$ and pinching errors
$\delta_j\to0$. Passing to a subsequence fixes their degree
$m\in[2,L]$. At least one asserted estimate fails, or one defining
projection vanishes. Denote the original pair
$\{\mathbf z_{0,j},\mathbf z_{1,j}\}$ by
$\{\mathbf z_{+,j},\mathbf z_{-,j}\}$, where
$\mathbf z_{\pm,j}=(x_{\pm,j},t_{\pm,j})$ and
$t_{+,j}\ge t_{-,j}$.
Multiplication by a positive constant preserves the equation,
$\Theta$, the doubling conditions and $\psi$. We therefore normalize
and rescale so that
\[
 r_j^{-n}\int_{B(x_{+,j},r_j)}u_{\eps_j}^2(x,t_{+,j})\,dx=1,
 \qquad w_j=T_{\mathbf z_{+,j},r_j}u_{\eps_j},\qquad K_j=r_j^{-1}.
\]
After extraction, the relative positions satisfy
\[
 (a_j,b_j):=\left(\frac{x_{-,j}-x_{+,j}}{r_j},
                  \frac{t_{-,j}-t_{+,j}}{r_j^2}\right)\to(a,b),
 \qquad |a_j|\le\tfrac14,\quad-\tfrac1{16}\le b_j\le0.
\]
Writing $A_j$ for the coefficient field of the $j$th counterexample,
we have
\[
 \partial_sw_j-\operatorname{div}_y\!\left[
 A_j\!\left(\frac{x_{+,j}+r_jy}{\eps_j},
                  \frac{t_{+,j}+r_j^2s}{\eps_j^2}\right)
 \nabla_yw_j\right]=0
 \quad\text{in }Q^-(\mathbf0,4K_j).
\]
Indeed, this region maps to
$Q^-(\mathbf z_{+,j},4)\subset Q(\mathbf0,6)$.
The translated coefficients satisfy \textup{(A1)--(A4)}, with
homogenization parameter $\eps_j/r_j$. By
\eqref{eq:uniform-theta} and \eqref{eq:theta-scale-invariance},
\[
 \Theta(w_j,\mathbf0,q)\le\Theta_1\quad(0<q\le K_j),
 \qquad \int_{B_1}w_j^2(y,0)\,dy=1.
\]

Lemma~\ref{lem:compactness} gives, after extraction,
\[
\begin{gathered}
 w_j\to v\quad\text{locally uniformly on }\R^n\times(-\infty,0],\\
 \partial_tv-\Delta v=0,\qquad \int_{B_1}v^2(y,0)\,dy=1.
\end{gathered}
\]
Equation~\eqref{eq:height-convergence} in Lemma~\ref{lem:compactness} gives convergence of the Gaussian masses at
$\mathbf0$, with positive limiting masses. Using
\eqref{eq:truncated-scaling}, the pinching at $\mathbf z_{+,j}$
passes to the limit:
\[
 D(v;\mathbf0,2^{-k})=m\qquad(k=0,1,2,3).
\]
Monotonicity then yields
\[
 D(v;\mathbf0,r)\equiv m\qquad(1/8\le r\le1).
\]
Lemma~\ref{lem:caloric-gaussian-doubling} in Appendix~\ref{sec:gaussian-appendix} implies that $v$ is a
caloric polynomial homogeneous of degree $m$. Set $P=v$.

For the moving centers, both cutoff supports lie in $B(0,3K_j)$.
The argument for \eqref{eq:polynomial-cylinder-bound}, applied using
the available $\Theta$ bound at every scale up to $K_j$, gives
$|w_j(y,b_j-q^2)|\le C(1+|y|)^d$ on these supports,
where $d=\max\{0,(\beta_0-n)/2\}$ and $1/16\le q\le1$.
Consequently, their Gaussian tails vanish uniformly. Local uniform
convergence and \eqref{eq:truncated-scaling} yield
\[
 H_1(u_{\eps_j};\mathbf z_{-,j},r_jq)
 =H_{K_j}(w_j;(a_j,b_j),q)
 \longrightarrow H(P;(a,b),q)>0.
\]
The last inequality holds because $P$ is a nonzero caloric polynomial.
Passing the pinching at $\mathbf z_{-,j}$ to the limit gives
$D(P;(a,b),2^{-k})=m$ for $k=0,1,2,3$; applying
Lemma~\ref{lem:caloric-gaussian-doubling} to $P(x+a,t+b)$ makes $P$ homogeneous of degree $m$ about $(a,b)$.
Lemma~\ref{lem:parabolic-spine}\textup{(iii), (i)} therefore gives
\[
 (a,b)\in\mathcal I(P),\qquad
 a\cdot\nabla P\equiv0,\qquad b\,\partial_tP\equiv0,
 \qquad P(x+a,t+b)=P(x,t).
\]

Local uniform convergence and the uniform Gaussian tail bound on the moving cutoff supports give Gaussian $L^2$ convergence. Homogeneity and translation
invariance of $P$, followed by the continuous projection $\Pi_m$, then give
\[
 \Psi_{\mathbf z_{\pm,j},r_j/4}
 \longrightarrow4^{-m}P(\cdot,-1)
 \quad\text{in }L^2(\nu_{0,1}).
\]
Since the limit is nonzero, normalization and homogeneous caloric
extension yield
\[
 \psi_{\pm,j}:=\psi_{\mathbf z_{\pm,j},r_j/4}
 \longrightarrow P/\|P\|_G.
\]
By finite dimensionality, the spatial and time derivatives also
converge in Gaussian norm. Hence
\[
\begin{aligned}
 &\|\psi_{+,j}-\psi_{-,j}\|_G\\
 &\quad+\max_{i\in\{+,-\}}\left\{
 \|a_j\cdot\nabla\psi_{i,j}\|_G
 +|b_j|\|\partial_t\psi_{i,j}\|_G\right\}\longrightarrow0.
\end{aligned}
\]
This proves both asserted estimates for the original centers, regardless
of their time ordering, and contradicts the choice of counterexamples.
\end{proof}

\begin{proposition}
\label{prop:multicenter-plane}
Let $u_\eps$ solve \eqref{eq:homogenization-equation} under
\textup{(A1)--(A5)} in $Q(\mathbf0,6)$ and satisfy
\eqref{eq:uniform-theta}.
Fix $L\ge2$, $0<\gamma<1/10$ and $0<\omega\le1/2$. There exist
$\delta,\eta_*>0$ and $\kappa_*\ge2$ such that the following holds.
Let
\[
 0<r_0\le1/\kappa_*,
 \quad \mathbf z_i\in Q(\mathbf z_*,r_0/8)\subset Q(\mathbf0,1),
 \quad \eta_*^{-1}\eps\le r_i\le r_0/64.
\]
Suppose $Q(\mathbf z_i,r_i/20)$ are mutually disjoint and
\[
 |D_1(u_\eps;\mathbf z_i,s)-m|\le\delta
 \quad(r_i\le s\le r_0),\qquad 2\le m\le L.
\]
Then there is a parabolic linear subspace $V$ with
\[
 d_P(\mathbf z_i-\mathbf z_k,V)
 \le\gamma d_P(\mathbf z_i,\mathbf z_k)
 \qquad(i,k),
\]
and one of the following alternatives holds:
\begin{enumerate}[label=\textup{(\roman*)},leftmargin=3em]
\item $\dim_PV\le n-1$;
\item $V=W\times\mathbb R$, $\dim W=n-2$, and there exist
$P,\phi\in\mathcal P_m$ such that
\begin{equation}\label{eq:mc-cylindrical-model}
 \|P\|_G=1,\qquad \partial_t\phi=0,\qquad \nabla_W\phi=0,
 \qquad \|P-\phi\|_G\le\omega,
\end{equation}
and
\begin{equation}\label{eq:mc-common-profile}
 \|\psi_{\mathbf z_i,2^{-k}r_0/4}-P\|_G\le\omega
 \qquad(k\in\mathbb N_0,\ 2^{-k}r_0/4\ge r_i).
\end{equation}
\end{enumerate}
\end{proposition}

\begin{proof}
Put $s_0=r_0/4$ and define
$\mathfrak p_i(s)$ by \eqref{eq:turning-projected-profile}, with
$u_\eps$ replaced by $T_{\mathbf z_i,1}u_\eps$.
The translated periodic coefficients satisfy \textup{(A1)--(A4)}
with the same constants. Lemmas~\ref{lem:fc-approximation} and
\ref{lem:dominant-degree} give
\[
 \|\mathfrak p_i(s_0)\|\ge3/4.
\]
Theorem~\ref{prop:fc-turning} gives, at dyadic $s=2^{-k}s_0\ge r_i$,
\[
 \|\mathfrak p_i(s)-\mathfrak p_i(s_0)\|
 \le C\{\delta+\eta_*+e^{-c/r_0^2}\}\le1/4.
\]
Thus $\|\mathfrak p_i(s)\|\ge1/2$, and normalization yields
\begin{equation}\label{eq:mc-uniform-turning}
 \|\psi_{\mathbf z_i,r_0/4}
       -\psi_{\mathbf z_i,2^{-k}r_0/4}\|_G
 \le C\{\delta+\eta_*+e^{-c/r_0^2}\}
\end{equation}
whenever $2^{-k}r_0/4\ge r_i$.
Choose a center $\mathbf z_0$ from the family and set
\[
 P:=\psi_{\mathbf z_0,r_0/4}\in\mathcal P_m,
 \qquad\|P\|_G=1.
\]
Since $d_P(\mathbf z_i,\mathbf z_0)\le r_0/4$,
Lemma~\ref{lem:two-center-qualitative} at scale $r_0$, with
accuracy $\zeta>0$, gives
\begin{equation}\label{eq:mc-common-scale-comparison}
 \|\psi_{\mathbf z_i,r_0/4}-P\|_G\le\zeta
 \qquad\text{for all }i.
\end{equation}

Fix $i\ne k$ and write $d=d_P(\mathbf z_i,\mathbf z_k)$.
Disjointness and $\mathbf z_i,\mathbf z_k\in Q(\mathbf z_*,r_0/8)$ give
\[
 \max\{r_i,r_k\}\le20d,\qquad d\le r_0/4.
\]
Choose a dyadic scale $r=2^{-j}r_0$, $j\ge0$, such that
\[
 \max\{4d,8r_i,8r_k\}\le r<2\max\{4d,8r_i,8r_k\}.
\]
This is possible since the maximum is at most $r_0$. Then
\[
 r\le r_0,\qquad r/8\ge\max\{r_i,r_k\},\qquad
 d/r\le1/4,\qquad 4\le r/d<320.
\]
Lemma~\ref{lem:two-center-qualitative} applies at scale $r$ with
accuracy $\zeta$. Set $P_i=\psi_{\mathbf z_i,r/4}$ and
$a=(x_i-x_k)/d$, $b=(t_i-t_k)/d^2$. Changing the normalization
from $r$ to $d$ gives
\[
 \|a\cdot\nabla P_i\|_G+|b|\|\partial_tP_i\|_G
 \le320^2\zeta,
 \qquad\max\{|a|,|b|^{1/2}\}=1.
\]
Equations~\eqref{eq:mc-common-scale-comparison} and \eqref{eq:mc-uniform-turning} yield
\[
 \|P_i-P\|_G
 \le\zeta+C\{\delta+\eta_*+e^{-c/r_0^2}\}.
\]
Finite-dimensional derivative bounds therefore imply
\begin{equation}\label{eq:mc-invariance}
\begin{aligned}
 \|a\cdot\nabla P\|_G+|b|\|\partial_tP\|_G
 &\le C_L\{320^2\zeta+\delta+\eta_*+e^{-c/r_0^2}\}\\
 &\le\tau.
\end{aligned}
\end{equation}
For any prescribed $\tau>0$, the last inequality follows by choosing
$\zeta$, then $\delta,\eta_*$ sufficiently small and $\kappa_*$
sufficiently large.

For $n=1$, we argue by contradiction. If there were two distinct
centers, then
\[
 \|\partial_xP\|_G=\sqrt{m/2}\ge1,
 \qquad \|\partial_tP\|_G\ge c(L)>0.
\]
By \eqref{eq:mc-invariance},
\[
 1=\max\{|a|,\sqrt{|b|}\}
 \le\max\{\tau,\sqrt{\tau/c(L)}\}<1
\]
for $\tau<\min\{1,c(L)\}$. Hence there is at most one center, and
$V=\{\mathbf0\}$ suffices.

Let $n\ge2$. Take $\sigma=\omega$ and choose $\tau>0$ below the
threshold in Lemma~\ref{lem:geometric-dichotomy} for $\gamma,\sigma$.
Together with \eqref{eq:mc-invariance}, Lemma~\ref{lem:geometric-dichotomy} gives a parabolic
linear subspace $V$ such that
\[
 d_P(\mathbf z_i-\mathbf z_k,V)
 =d\cdot d_P((a,b),V)\le\gamma d.
\]
Either $\dim_P V\le n-1$,
or $V=W\times\mathbb R$, $\dim W=n-2$, and
\eqref{eq:mc-cylindrical-model} holds. Finally, choose
$\zeta,\delta,\eta_*$ sufficiently small and $\kappa_*$ sufficiently
large so that \eqref{eq:mc-invariance} holds and
\[
 \zeta+C\{\delta+\eta_*+e^{-c/r_0^2}\}\le\omega.
\]
Equations~\eqref{eq:mc-common-scale-comparison} and
\eqref{eq:mc-uniform-turning} then give \eqref{eq:mc-common-profile}.
\end{proof}

\begin{proposition}
\label{prop:cylindrical-singular-points}
There exists $\omega_0=\omega_0(L,\gamma)>0$ with the following
property. Under the assumptions of Proposition~\ref{prop:multicenter-plane},
take $0<\omega\le\omega_0$ and choose its thresholds also to satisfy
Lemma~\ref{lem:two-sided-pinching} with $\zeta=\mu\omega$.
If alternative \textup{(ii)} holds, then the subspace $W$ in that
alternative satisfies, for every $i$,
\[
 \operatorname{dist}(x-x_i,W)
 \le\gamma d_P(\mathbf z,\mathbf z_i)
\]
whenever
\[
 \mathbf z=(x,t)\in S(u_\eps),\qquad
 r_i/4\le d_P(\mathbf z,\mathbf z_i)\le r_0/64.
\]
\end{proposition}

\begin{proof}
Let $P,\phi,W$ be as in alternative \textup{(ii)} of
Proposition~\ref{prop:multicenter-plane}. Fix $i$ and
$\mathbf z=(x,t)\in S(u_\eps)$ with
$r_i/4\le d=d_P(\mathbf z,\mathbf z_i)\le r_0/64$.
Choose $r=2^{-j}r_0$ with
\[
 8\max\{d,r_i\}\le r<16\max\{d,r_i\},\qquad s=r/4.
\]
Then $r/8\ge r_i$, $r\le r_0/4$ and
\[
 (y,t')=((x-x_i)/s,(t-t_i)/s^2)\in Q(\mathbf0,1),
 \qquad 1/16\le d/s\le1/2.
\]
Since $r\ge8r_i$, we have $\eps/r\le\eta_*/8$.
Lemma~\ref{lem:two-sided-pinching}, with $\zeta=\mu\omega$, and
\textup{(A5)} give at this singular point
\[
 |\nabla\psi_{\mathbf z_i,s}(y,t')|\le\omega.
\]
By \eqref{eq:mc-cylindrical-model}, \eqref{eq:mc-common-profile}
and the finite-dimensional gradient bound,
\[
 |\nabla\phi(y)|\le(2C_L+1)\omega.
\]
Set $c_m=m/(2^{m-1/2}\sqrt{m!})$.
Choose $0<\omega_0\le1/2$ so that
\begin{equation}\label{eq:mc-tolerance-choice}
 (2C_L+1)\omega_0
 \le\frac12\min_{2\le m\le L}c_m(\gamma/16)^{m-1}.
\end{equation}
By \eqref{eq:cylinder-gradient}, the stationary two-dimensional harmonic polynomial satisfies
\[
 |\nabla\phi(y)|=c_m\|\phi\|_G\operatorname{dist}(y,W)^{m-1}.
\]
Since $\|\phi\|_G\ge1-\omega\ge1/2$, the choice
\eqref{eq:mc-tolerance-choice} implies
$\operatorname{dist}(y,W)\le\gamma/16$. Hence
\[
 \operatorname{dist}(x-x_i,W)
 =s\operatorname{dist}(y,W)\le(\gamma/16)s\le\gamma d.\qedhere
\]
\end{proof}

\section{Covering and the singular-set estimate}\label{sec:measure}

For fixed $\Theta_1$ and coefficient data,
let $L\ge2$ be an integer and $\gamma=1/20$.
Fix $\omega=\omega_0(L,\gamma)/2$ from
Proposition~\ref{prop:cylindrical-singular-points}. Choose
$0<\delta\le1/16$ and the remaining thresholds jointly as in
Propositions~\ref{prop:multicenter-plane}--\ref{prop:cylindrical-singular-points}.
Choose $\kappa\ge8$ and
$0<\eps_0\le1/2$ such that $\kappa\ge\kappa_*$ and
$2\eps_0\le\eta_*$ for the thresholds in
Theorems~\ref{thm:ls}, \ref{thm:truncated-doubling-drop},
\ref{thm:low-doubling-nondegeneracy}, and
Propositions~\ref{prop:multicenter-plane}--\ref{prop:cylindrical-singular-points}.
For Theorem~\ref{thm:ls}, use all integers $1\le l\le L$ and both
signs with $\delta_0=\delta$, and also require $4L\eps_0\le\delta\eta_*$
for each of these thresholds. Set
\begin{equation}\label{eq:cover-fixed-parameters}
 \eta_0=\frac{\eps_0}{200},\qquad \kappa_0=4\kappa.
\end{equation}
These constants are fixed throughout this section. Dependence on the
fixed $\Theta_1$ and coefficient data is suppressed, including in $c(L)$.

\subsection{Stopping-scale cover}\label{sec:stopping}

We follow the covering construction of Lin and Shen
\cite[Lemma~7.1]{LS}.

\begin{lemma}\label{thm:localized-stopping-cover}
Let $u_\eps$ solve \eqref{eq:homogenization-equation} in
$Q(\mathbf0,6)$ under \textup{(A1)--(A5)} and satisfy
\eqref{eq:uniform-theta}. Let $l\in\{2,\ldots,L\}$, and suppose
\[
 0<R\le\frac1{\kappa_0},\quad 0<\eps\le\sigma<\eta_0R,
\]
with $Q(\mathbf z,2R)\subset Q(\mathbf0,1)$, and assume
\begin{equation}\label{eq:level-condition}
 D_1(u_\eps;\mathbf y,4R)\le l+\delta
 \qquad\bigl(\mathbf y\in S(u_\eps)\cap Q(\mathbf z,2R)\bigr).
\end{equation}
Put $E=S(u_\eps)\cap Q(\mathbf z,3R/2)$. There exist constants
$C>0$, $c=c(L)>0$ and finitely many balls $Q(\mathbf v_i,r_i)$, with
\begin{equation}\label{eq:localized-cover-radii}
 \mathbf v_i\in E,\qquad 0<r_i\le R/100,
\end{equation}
such that:
\begin{enumerate}[label=\arabic*.,leftmargin=2.5em]
\item
\begin{equation}\label{eq:localized-cover-inclusion}
 E\subset\bigcup_i Q(\mathbf v_i,r_i/4).
\end{equation}
\item
\begin{equation}\label{eq:localized-cover-content}
 \sum_i r_i^n\le CR^n.
\end{equation}
\item For every $i$, either
\[
 \frac{\sigma}{40\eps_0}\le r_i\le\frac{6\sigma}{\eps_0},
\]
or $r_i>6\sigma/\eps_0$ and
\begin{equation}\label{eq:localized-cover-drop}
 D_1(u_\eps;\mathbf y,cr_i)
 \le l-1+\delta
 \qquad\bigl(\mathbf y\in E\cap Q(\mathbf v_i,r_i)\bigr).
\end{equation}
\end{enumerate}
The constants are independent of $\eps$ and $\sigma$; $C$ depends
only on $\Theta_1,L$ and the coefficient data.
\end{lemma}

\begin{proof}
\emph{Step 1. The stopping scale.}
Set $r_0=R/100$ and $s_0=\eps_0^{-1}\sigma<r_0/2$.
Theorem~\ref{thm:ls}, applied to \eqref{eq:level-condition} at scale
$4R$, gives
\begin{equation}\label{eq:cover-upper-propagation}
 D_1(u_\eps;\mathbf y,s)\le l+\delta
 \qquad(\mathbf y\in E,\ s_0\le s\le2R).
\end{equation}
For $\mathbf y\in E$, define
\begin{equation}\label{eq:cover-stopping-radius}
 r_*(\mathbf y)=\sup\bigl\{s_0<s\le r_0:
 D_1(u_\eps;\mathbf y,s)\le l-\delta\bigr\}.
\end{equation}
If the set is empty, we set $r_*(\mathbf y)=s_0$.
Continuity in $s$ and Theorem~\ref{thm:ls} imply
\begin{align}
 r_*(\mathbf y)<r_0\quad&\Longrightarrow\quad
 |D_1(u_\eps;\mathbf y,s)-l|\le\delta
 \quad(r_*(\mathbf y)\le s\le r_0),
 \label{eq:cover-stopping-pinching}\\
 r_*(\mathbf y)>s_0\quad&\Longrightarrow\quad
 D_1(u_\eps;\mathbf y,s)\le l-\delta
 \quad(s_0\le s\le r_*(\mathbf y)/2).
 \label{eq:cover-stopping-low}
\end{align}
The scale conditions follow from $4R\le\kappa^{-1}$ and
$\eps\le\sigma=\eps_0s_0$.

\emph{Step 2. Good centers.}
Let
\begin{equation}\label{eq:geometric-good-centers}
 G=\bigl\{\mathbf x\in E:r_*(\mathbf y)\ge r_*(\mathbf x)/3
 \text{ for all }\mathbf y\in E\cap Q(\mathbf x,r_*(\mathbf x))\bigr\}.
\end{equation}
The metric $5r$ covering lemma supplies points $\mathbf x_i\in G$,
with $\rho_i=r_*(\mathbf x_i)$, such that
\begin{equation}\label{eq:geometric-good-cover}
 G\subset\bigcup_iQ(\mathbf x_i,\rho_i/4),\qquad
 Q(\mathbf x_i,\rho_i/20)\cap Q(\mathbf x_k,\rho_k/20)=\varnothing
 \quad(i\ne k).
\end{equation}
There are finitely many selected balls, since $\rho_i\ge s_0$.
The balls with $\rho_i>r_0/64$ satisfy
\[
 \sum_{\rho_i>r_0/64}\rho_i^n
 \le(64/r_0)^2\sum_i \rho_i^{n+2}\le CR^n.
\]
Partition the other centers into a bounded number of groups
$X_q\subset Q(\mathbf z_q,r_0/8)$, with $\mathbf z_q\in E$.
For each group, \eqref{eq:cover-stopping-pinching} and
Proposition~\ref{prop:multicenter-plane} give a parabolic linear
subspace $V_q$ such that
\begin{equation}\label{eq:packing-fixed-cone}
 d_P(\mathbf x_i-\mathbf x_k,V_q)
 \le\gamma d_P(\mathbf x_i,\mathbf x_k)
 \qquad(\mathbf x_i,\mathbf x_k\in X_q).
\end{equation}
Either $d_q:=\dim_PV_q\le n-1$, or
$V_q=W_q\times\R$ with $\dim W_q=n-2$ and $d_q=n$.
Let $\pi_q$ be the coordinate projection onto $V_q$, and put
$d_{V_q}=d_P|_{V_q\times V_q}$ and $a_\gamma=(1-\gamma^2)^{1/2}$.
Then
\begin{equation}\label{eq:projection-bilipschitz-centers}
 a_\gamma d_P(\mathbf x_i,\mathbf x_k)
 \le d_{V_q}(\pi_q\mathbf x_i,\pi_q\mathbf x_k)
 \le d_P(\mathbf x_i,\mathbf x_k).
\end{equation}
If the spatial distance realizes $d_P$,
\eqref{eq:packing-fixed-cone} bounds its projection from below by
the factor $a_\gamma$. Otherwise, $V_q$ must be vertical, and
$\pi_q$ retains the time coordinate. Writing
$Q_{V_q}(\mathbf v,s)=V_q\cap Q(\mathbf v,s)$, we see that the balls $Q_{V_q}(\pi_q\mathbf x_i,a_\gamma \rho_i/80)$ are
disjoint. Lebesgue measure on $V_q$ scales as $s^{d_q}$, so
\[
 \sum_{\mathbf x_i\in X_q}\rho_i^{d_q}\le Cr_0^{d_q},\qquad
 \sum_{\mathbf x_i\in X_q}\rho_i^n\le Cr_0^n.
\]
If $d_q=0$, \eqref{eq:projection-bilipschitz-centers} gives
$\#X_q\le1$ and the latter estimate still holds. Summing yields
\begin{equation}\label{eq:good-parabolic-content}
 \sum_i \rho_i^n\le CR^n.
\end{equation}

\emph{Step 3. Covering the remaining points.}
For $\mathbf y\in E\setminus G$, start with $\mathbf z_0=\mathbf y$
and successively choose
\[
 \mathbf z_{k+1}\in E\cap Q(\mathbf z_k,r_*(\mathbf z_k)),
 \qquad r_*(\mathbf z_{k+1})<r_*(\mathbf z_k)/3.
\]
Since $r_*\ge s_0$, the process ends at some $\mathbf x\in G$, with
\[
 d_P(\mathbf y,\mathbf x)<\frac32r_*(\mathbf y),\qquad
 r_*(\mathbf x)<\frac13r_*(\mathbf y).
\]
Choose $i$ with $\mathbf x\in Q(\mathbf x_i,\rho_i/4)$.
Then $r_*(\mathbf x)\ge \rho_i/3$, whence $\rho_i<r_*(\mathbf y)$ and,
writing $X=\{\mathbf x_i\}$,
\begin{equation}\label{eq:distance-good-centers}
 d_P(\mathbf y,X)<2r_*(\mathbf y)
 \qquad(\mathbf y\in E\setminus G).
\end{equation}
For $\mathbf y\in E':=E\setminus\bigcup_iQ(\mathbf x_i,\rho_i/4)$,
set $t(\mathbf y)=d_P(\mathbf y,X)/10$. We have
\[
 s_0/40\le t(\mathbf y)<r_*(\mathbf y)/5\le r_0/5.
\]
Choose a finite disjoint family $Q(\mathbf y_j,t_j/20)$, where
$\mathbf y_j\in E'$ and $t_j=t(\mathbf y_j)$, whose fivefold
dilates cover $E'$. Consequently,
\begin{equation}\label{eq:cover-unpacked-balls}
 E\subset\bigcup_iQ(\mathbf x_i,\rho_i/4)
       \cup\bigcup_jQ(\mathbf y_j,t_j/4).
\end{equation}

\emph{Step 4. The decrease of the doubling level.}
For $\mathbf w\in E\cap Q(\mathbf y_j,t_j)$,
$d_P(\mathbf w,X)\ge9t_j$. If $\mathbf w\notin G$,
\eqref{eq:distance-good-centers} gives $r_*(\mathbf w)>9t_j/2$.
If $\mathbf w\in G$, choose $i$ with
$\mathbf w\in Q(\mathbf x_i,\rho_i/4)$; then
\[
 10t_j\le d_P(\mathbf y_j,\mathbf x_i)<t_j+\rho_i/4,
 \qquad r_*(\mathbf w)\ge \rho_i/3>12t_j.
\]
Together with \eqref{eq:geometric-good-centers}, this proves
\begin{equation}\label{eq:cover-collar-lower}
 \begin{aligned}
 r_*(\mathbf w)&\ge \rho_i/3
 &&(\mathbf w\in E\cap Q(\mathbf x_i,\rho_i)),\\
 r_*(\mathbf w)&\ge4t_j
 &&(\mathbf w\in E\cap Q(\mathbf y_j,t_j)).
 \end{aligned}
\end{equation}
Let $(\mathbf v,h)$ be any pair $(\mathbf x_i,\rho_i)$ or
$(\mathbf y_j,t_j)$ with $h>6s_0$.
Equations~\eqref{eq:cover-stopping-low} and
\eqref{eq:cover-collar-lower} give
\[
 D_1(u_\eps;\mathbf w,h/6)\le l-\delta
 \qquad(\mathbf w\in E\cap Q(\mathbf v,h)).
\]
Since $h/6\le\kappa^{-1}$ and $\eps\le\sigma<\eps_0h/6$,
Theorem~\ref{thm:truncated-doubling-drop} at scale $h/6$ gives
\[
 D_1\!\left(u_\eps;\mathbf w,\frac{\delta h}{24l}\right)
 \le l-1+\delta.
\]
Take $c=\delta/(48L)$, and let $\eta_*$ be the threshold in
Theorem~\ref{thm:ls} for level $l-1$ and $\delta_0=\delta$.
The choices of $\eps_0$ and $h$ give
\[
 \frac{\eps}{2\eta_*}
 <\frac{\eps_0h}{12\eta_*}
 \le ch\le\frac{\delta h}{48l}.
\]
Applying \eqref{eq:ls-iterated} with the plus sign, initial radius
$\delta h/(24l)$ and target radius $ch$ proves
\eqref{eq:localized-cover-drop}.

\emph{Step 5. The sum of the remaining radii.}
Put $c_*=2^{-12}$. By disjointness,
\[
 \sum_{t_j>c_*r_0}t_j^n
 \le(c_*r_0)^{-2}\sum_jt_j^{n+2}\le CR^n.
\]
For each remaining $j$, choose a nearest center
$\mathbf x_{i(j)}\in X$. Since $\mathbf y_j\in E'$,
\[
 \rho_{i(j)}/4\le d_P(\mathbf y_j,\mathbf x_{i(j)})=10t_j,
 \qquad \rho_{i(j)}\le40c_*r_0<r_0/64.
\]
Assign $j$ to the group $X_q$ containing $\mathbf x_{i(j)}$.
For all indices assigned to $q$,
\begin{equation}\label{eq:whitney-distance-radius}
 d_P(\mathbf y_j,X_q)=10t_j,\qquad
 \mathbf y_j\in Q(\mathbf z_q,r_0),\qquad
 10t_j\le r_0/64.
\end{equation}

Suppose first that $d_q\le n-1$. By
\eqref{eq:projection-bilipschitz-centers}, $X_q$ has a
$2^{-k}r_0$-net with at most $C2^{kd_q}$ points. Let
$J_{q,k}$ consist of the assigned indices with
$2^{-k-1}r_0<t_j\le2^{-k}r_0$.
Each $Q(\mathbf y_j,t_j/20)$, $j\in J_{q,k}$, lies in a ball of
radius $12\cdot2^{-k}r_0$ centered at a point of the net.
These disjoint balls have radii greater than $2^{-k}r_0/40$.
Ambient volume comparison therefore gives
\[
 \#J_{q,k}\le C2^{kd_q},\qquad
 \sum_{j\text{ assigned to }q}t_j^n
 \le Cr_0^n\sum_{k=0}^{\infty}2^{-k(n-d_q)}\le Cr_0^n.
\]

Suppose now that $d_q=n$, so $V_q=W_q\times\R$ and
$\dim W_q=n-2$. Proposition~\ref{prop:cylindrical-singular-points}
and \eqref{eq:whitney-distance-radius} give
\begin{equation}\label{eq:whitney-nearest-cone}
 d_P(\mathbf y_j-\mathbf x_{i(j)},V_q)\le10\gamma t_j.
\end{equation}
Write $y^\perp=\pi_{W_q^\perp}y$ and
$\theta=\gamma/a_\gamma$. For $h=10t_j$ and any
$\mathbf x_i\in X_q$, put
$p=d_{V_q}(\pi_q\mathbf y_j,\pi_q\mathbf x_i)$. Equations
\eqref{eq:packing-fixed-cone}--\eqref{eq:projection-bilipschitz-centers}
and \eqref{eq:whitney-nearest-cone} imply
\[
 \begin{aligned}
 |y_j^\perp-x_i^\perp|
 &\le\gamma h+\theta\,
 d_{V_q}(\pi_q\mathbf x_{i(j)},\pi_q\mathbf x_i)
 \le\gamma h+\theta(h+p),\\
 h&\le d_P(\mathbf y_j,\mathbf x_i)
 \le p+|y_j^\perp-x_i^\perp|
 \le(1+\theta)p+(\gamma+\theta)h.
 \end{aligned}
\]
Since $(1-\gamma-\theta)/(1+\theta)\ge1/2$, we obtain
\begin{equation}\label{eq:projected-whitney-distance}
 5t_j\le d_{V_q}(\pi_q\mathbf y_j,\pi_qX_q)\le10t_j.
\end{equation}
If $\mathbf v\in Q_{V_q}(\pi_q\mathbf y_j,t_j/20)$, then
\[
 \frac{99}{20}t_j
 \le d_{V_q}(\mathbf v,\pi_qX_q)
 \le\frac{201}{20}t_j.
\]
For any collection of projected balls meeting at one point, put
$T=\max_j t_j$. The preceding inequality gives $T/3\le t_j\le T$.
Their nearest centers in $X_q$ have projected
distances at most $21T$ from one another. By
\eqref{eq:projection-bilipschitz-centers}, the corresponding
disjoint balls $Q(\mathbf y_j,t_j/20)$ lie in one ball of radius
$43T$ and have radii at least $T/60$. Hence
\[
 \sum_{j\text{ assigned to }q}
 \mathbf1_{Q_{V_q}(\pi_q\mathbf y_j,t_j/20)}\le C.
\]
Integrating on $V_q$, whose parabolic dimension is $n$, gives
$\sum_{j\text{ assigned to }q}t_j^n\le Cr_0^n$.
Summing over the groups and adding the balls with $t_j>c_*r_0$,
we conclude that
\begin{equation}\label{eq:bad-parabolic-content}
 \sum_jt_j^n\le CR^n.
\end{equation}
Relabel the balls $Q(\mathbf x_i,\rho_i)$ and
$Q(\mathbf y_j,t_j)$ as $Q(\mathbf v_i,r_i)$. Their radii satisfy
$s_0/40\le r_i\le r_0$. Equations~\eqref{eq:cover-unpacked-balls},
\eqref{eq:good-parabolic-content}, \eqref{eq:bad-parabolic-content}
and Step~4 give the three conclusions.
\end{proof}

\subsection{Iteration to a prescribed scale}\label{subsec:cover-iteration}

As in \cite[Lemma~7.2]{LS}, we retain balls of radius comparable to
a prescribed scale $\sigma\ge\eps$ and apply the stopping cover to
the other balls with the
doubling level decreased by one. All indices in this iteration use
the fixed unit cutoff in \eqref{eq:truncated-doubling-index}.

\begin{lemma}\label{lem:level-induction}
Let $u_\eps$ solve \eqref{eq:homogenization-equation} in
$Q(\mathbf0,6)$ under \textup{(A1)--(A5)} and satisfy
\eqref{eq:uniform-theta}. Let $\delta,\eps_0,\eta_0,\kappa_0$ be as in
Section~\ref{sec:measure}, let $c$ be as in
Lemma~\ref{thm:localized-stopping-cover}, and put $a=c/(160\eps_0)$.
There exists $C_{\rm cov}\ge1$, depending only on $\Theta_1,L$ and
the coefficient data and independent of $\eps,\sigma$, with the
following property. Suppose
\[
 0<R\le\frac1{\kappa_0},\quad 0<\eps\le\sigma<\eta_0R,
\]
with $Q(\mathbf z,2R)\subset Q(\mathbf0,1)$. For an integer
$1\le l\le L$, assume \eqref{eq:level-condition}.
Then there are finitely many balls $Q(\mathbf z_i,\rho_i)$ such that
\begin{gather}
 \mathbf z_i\in S(u_\eps)\cap Q(\mathbf z,2R),\qquad
 a\sigma\le\rho_i\le\min\{\sigma/\eta_0,R/8\},
 \label{eq:terminal-cover-radii}\\
 S(u_\eps)\cap Q(\mathbf z,R)
 \subset\bigcup_iQ(\mathbf z_i,\rho_i),\qquad
 \sum_i\rho_i^n\le C_{\rm cov}^{\,l-1}R^n.
 \label{eq:terminal-cover-content}
\end{gather}
\end{lemma}

\begin{proof}
First let $l\ge2$ and apply
Lemma~\ref{thm:localized-stopping-cover}. For each ball
$Q(\mathbf v_i,r_i/4)$,
cover
\[
 S(u_\eps)\cap Q(\mathbf z,R)\cap Q(\mathbf v_i,r_i/4)
\]
by at most $Cc^{-(n+2)}$ balls $Q(\mathbf y_j,s_j)$,
centered in this set, with $s_j=cr_i/4$. Then
\begin{gather}
 a\sigma\le s_j\le R/8,\qquad
 S(u_\eps)\cap Q(\mathbf z,R)
 \subset\bigcup_jQ(\mathbf y_j,s_j),
 \label{eq:refined-stopping-cover}\\
 \sum_js_j^n\le Cc^{-2}\sum_i r_i^n
 \le C_{\rm cov}R^n.
 \label{eq:refined-stopping-content}
\end{gather}
For a ball arising from $Q(\mathbf v_i,r_i/4)$, we have
\[
 Q(\mathbf y_j,2s_j)
 \subset Q(\mathbf v_i,r_i)\cap Q(\mathbf z,3R/2)
 \subset Q(\mathbf z,2R).
\]
If $s_j>\sigma/\eta_0$, then $r_i>6\eps_0^{-1}\sigma$. Since
$4s_j=cr_i$, \eqref{eq:localized-cover-drop} gives
\begin{equation}\label{eq:refined-stopping-drop}
 D_1(u_\eps;\mathbf y,4s_j)\le l-1+\delta
 \qquad\bigl(\mathbf y\in S(u_\eps)\cap Q(\mathbf y_j,2s_j)\bigr).
\end{equation}
The constants $a,\eta_0,C_{\rm cov}$ can be used for every
$2\le l\le L$.

We now argue by induction on $l$. If $l=1$, then
\eqref{eq:level-condition} and $\delta\le1/16$ give
$D_1(u_\eps;\mathbf y,4R)\le3/2$ at each singular point in
$Q(\mathbf z,2R)$. Since $4R\le\kappa^{-1}$ and
$\eps\le\sigma<\eta_0R<4\eps_0R$, Theorem~\ref{thm:low-doubling-nondegeneracy}
implies $S(u_\eps)\cap Q(\mathbf z,2R)=\varnothing$.

Suppose that the assertion holds at level $l-1$.
Retain the balls in \eqref{eq:refined-stopping-cover} with
$s_j\le \sigma/\eta_0$. For each $s_j>\sigma/\eta_0$, apply the induction
hypothesis with center $\mathbf y_j$, radius $s_j$, and the same
$\sigma$.
The required doubling bound is \eqref{eq:refined-stopping-drop};
also $s_j\le R/8\le \kappa_0^{-1}$ and
$Q(\mathbf y_j,2s_j)\subset Q(\mathbf z,2R)\subset Q(\mathbf0,1)$.
The retained balls and the covers obtained by induction satisfy
\eqref{eq:terminal-cover-radii}, and their radii obey
\[
 \begin{aligned}
 \sum_i\rho_i^n
 &\le\sum_{s_j\le \sigma/\eta_0}s_j^n
   +C_{\rm cov}^{\,l-2}\sum_{s_j>\sigma/\eta_0}s_j^n\\
 &\le C_{\rm cov}^{\,l-2}\sum_js_j^n
 \le C_{\rm cov}^{\,l-1}R^n.
 \end{aligned}
\]
Each decomposition reduces the integer level by one, so after at
most $l-1$ decompositions all remaining balls have radii between
$a\sigma$ and $\sigma/\eta_0$.
\end{proof}

\subsection{Microscopic estimates and proofs of the measure bounds}
\label{subsec:microscopic-proof}

\begin{proposition}
\label{prop:microscopic-measure}
Let $u_\eps$ solve \eqref{eq:homogenization-equation} in
$Q(\mathbf0,6)$ under \textup{(A1)--(A3)} and satisfy
\eqref{eq:uniform-theta}. For every $\kappa\ge1$, there are
$C=C(\Theta_1,\kappa)$ and $c=c(\Theta_1,\kappa)>0$ such that,
if
\[
 \mathbf z_0\in Q(\mathbf0,1),\qquad
 0<\rho\le\min\{1/32,\kappa\eps\},\qquad
 E=S(u_\eps)\cap Q(\mathbf z_0,\rho),
\]
then
\begin{align}
 \mathcal H_P^{n+2}\bigl(Q(E,s)\bigr)
 &\le C\rho^ns^2,\qquad 0<s\le c\rho,
 \label{eq:microscopic-tube}\\
 \mathcal H_P^n(E)&\le C\rho^n.
 \label{eq:microscopic-hausdorff}
\end{align}
The constants are independent of $\eps$.
\end{proposition}

\begin{proof}
Write $\mathbf z_0=(x_0,t_0)$ and set $h=8\rho\le1/4$ and
$v=T_{\mathbf z_0,h}u_\eps$. On $Q(\mathbf0,5)$, we have
\[
 \partial_\tau v-\operatorname{div}(A_h\nabla v)=0,\qquad
 A_h(y,\tau)=A\!\left(\frac{x_0+hy}{\eps},
                         \frac{t_0+h^2\tau}{\eps^2}\right).
\]
The ellipticity bounds are unchanged, and
\[
 |A_h(y,\tau)-A_h(y',\tau')|
 \le16\kappa L_A\,d_P\bigl((y,\tau),(y',\tau')\bigr).
\]
Together with \textup{(A1)}, this verifies the coefficient condition
\cite[(1.1)]{HKM}, with a constant depending only on the coefficient
data and $\kappa$.
For $\sigma\in[-2,2]$, write $t_\sigma=t_0+h^2\sigma$ and
$\mathbf z_\sigma=(x_0,t_\sigma)\in Q(\mathbf0,2)$.
The bound \eqref{eq:uniform-theta} at scale $2h$, followed by
\eqref{eq:spatial-single-step} at $h/2$ and $h$, gives
\begin{equation}\label{eq:microscopic-HKM-control}
 \begin{aligned}
 &\int_{B(x_0,5h)\times[t_0-25h^2,t_\sigma]}|u_\eps|^2\dd x\dd t\\
 &\qquad\le\int_{Q^-(\mathbf z_\sigma,8h)}|u_\eps|^2\dd x\dd t
 \le4h^2\Theta_1\int_{B(x_0,2h)}|u_\eps(x,t_\sigma)|^2\dd x\\
 &\qquad\le C\Theta_1^3h^2
       \int_{B(x_0,h/2)}|u_\eps(x,t_\sigma)|^2\dd x.
 \end{aligned}
\end{equation}
After rescaling, the quotient in \cite[(1.4)]{HKM} satisfies
\[
 \Theta_{\rm HKM}(v)
 :=\sup_{\sigma\in[-2,2]}
 \frac{\displaystyle\int_{B_5\times[-25,\sigma]}|v|^2
                         \dd\mathcal H_P^{n+2}}
      {\displaystyle\int_{B_{1/2}}|v(y,\sigma)|^2\dd y}
 \le C\Theta_1^3.
\]
Theorem~1.4 of \cite{HKM} yields
\[
 \mathcal H_P^{n+2}
 \bigl(Q(S(v),\tau)\cap Q(\mathbf0,1/4)\bigr)
 \le C(\Theta_1,\kappa)\tau^2,
 \qquad 0<\tau\le\tau_0(\Theta_1,\kappa).
\]
For $E_v=S(v)\cap Q(\mathbf0,1/8)$ and
$0<\tau\le\min\{\tau_0,1/16\}$,
$Q(E_v,\tau)\subset Q(\mathbf0,1/4)$.
Since the dilation of ratio $h$ maps $E_v$ onto $E$,
\[
 \mathcal H_P^{n+2}(Q(E,s))
 =h^{n+2}\mathcal H_P^{n+2}(Q(E_v,s/h))
 \le Ch^ns^2\le C\rho^ns^2,
\]
which proves \eqref{eq:microscopic-tube} for
$0<s\le8\min\{\tau_0,1/16\}\rho$.

Choose a maximal $s$-separated set
$\{\mathbf z_i\}_{i=1}^{N_s}\subset E$.
The balls $Q(\mathbf z_i,s)$ cover $E$, while the balls
$Q(\mathbf z_i,s/2)$ are disjoint and contained in $Q(E,s)$.
Thus
\[
 c_nN_ss^{n+2}\le\mathcal H_P^{n+2}(Q(E,s))
 \le C\rho^ns^2,\qquad
 N_s(2s)^n\le C\rho^n.
\]
Letting $s\downarrow0$ proves \eqref{eq:microscopic-hausdorff}.
\end{proof}

\begin{proof}[Proof of Theorem~\ref{thm:main-minkowski}]
By Corollary~\ref{thm:allcenters}, \eqref{eq:uniform-theta} holds
with $\Theta_1=\Theta_1(\Theta_0)$. Let $L_0=L(\Theta_1)$ be the
bound in Proposition~\ref{prop:theta-controls-truncated-doubling},
and fix the integer $L=\max\{2,\lceil L_0\rceil\}$.
Choose $\delta,\eps_0,a,\eta_0,\kappa,\kappa_0$ as in
Section~\ref{sec:measure} and Lemma~\ref{lem:level-induction}.
Put $R=\kappa_0^{-1}\le1/32$. For $\mathbf z\in Q(\mathbf0,1/2)$,
we have $Q(\mathbf z,2R)\subset Q(\mathbf0,1)$, and
\eqref{eq:all-truncated-doubling} gives
\[
 D_1(u_\eps;\mathbf y,4R)\le L_0\le L
 \qquad\bigl(\mathbf y\in S(u_\eps)\cap Q(\mathbf z,2R)\bigr).
\]
It suffices to estimate the tube about
$E_{\mathbf z}=S(u_\eps)\cap Q(\mathbf z,R)$ for
$\mathbf z\in Q(\mathbf0,1/2)$, since finitely many such balls cover
$Q(\mathbf0,1/2)$. The number of balls depends only on $\Theta_0$
and the fixed coefficient data. Write $c_{\rm mic}>0$ for the constant
in \eqref{eq:microscopic-tube} with $\kappa=\eta_0^{-1}$.

If $\eps\ge\eta_0R$, Proposition~\ref{prop:microscopic-measure}
with $\rho=R$ gives
\[
\mathcal H_P^{n+2}(Q(E_{\mathbf z},s))\le CR^ns^2
\qquad(0<s\le c_{\rm mic}R).
\]
For $c_{\rm mic}R<s\le1$, the same bound with a larger constant
depending on $R,c_{\rm mic}$ follows from
$Q(E_{\mathbf z},s)\subset Q(\mathbf z,R+s)$.

Suppose now that $\eps<\eta_0R$, and set
$b=\min\{1,c_{\rm mic}a\}$.
For $0<s<b\eps$, apply Lemma~\ref{lem:level-induction} with
$\sigma=\eps$. Its balls satisfy $\rho_i\ge a\eps$, so
$s\le c_{\rm mic}\rho_i$, and \eqref{eq:microscopic-tube} yields
\[
\begin{aligned}
\mathcal H_P^{n+2}(Q(E_{\mathbf z},s))
&\le\sum_i\mathcal H_P^{n+2}
 \bigl(Q(S(u_\eps)\cap Q(\mathbf z_i,\rho_i),s)\bigr)\\
&\le Cs^2\sum_i\rho_i^n\le CR^ns^2.
\end{aligned}
\]
For $b\eps\le s<\eta_0R$, apply the same lemma with
$\sigma=\max\{\eps,s\}<\eta_0R$. Then
\[
a s\le\rho_i\le\frac{s}{b\eta_0},\qquad
\sum_i\rho_i^n\le C_{\rm cov}^{L-1}R^n.
\]
Since $Q(E_{\mathbf z},s)\subset\bigcup_iQ(\mathbf z_i,\rho_i+s)$,
parabolic volume scaling gives
\[
\mathcal H_P^{n+2}(Q(E_{\mathbf z},s))
\le C_n\sum_i(\rho_i+s)^{n+2}
\le Cs^2\sum_i\rho_i^n\le CR^ns^2.
\]
Finally, for $\eta_0R\le s\le1$, the inclusion
$Q(E_{\mathbf z},s)\subset Q(\mathbf z,R+s)$ gives the bound
$C(\Theta_0)s^2$. Summing over the fixed finite cover proves
\eqref{eq:main-minkowski}.
\end{proof}

\begin{proof}[Proof of Theorem~\ref{thm:main-measure}]
Let $E_\eps=S(u_\eps)\cap Q(\mathbf0,1/2)$ and choose a maximal
$r$-separated set $\{\mathbf z_i\}_{i=1}^{N_r}\subset E_\eps$.
The balls $Q(\mathbf z_i,r)$ cover $E_\eps$, while
$Q(\mathbf z_i,r/2)$ are disjoint and contained in $Q(E_\eps,r)$.
By Theorem~\ref{thm:main-minkowski},
\[
 c_nN_rr^{n+2}\le\mathcal H_P^{n+2}(Q(E_\eps,r))
 \le C(\Theta_0)r^2,\qquad
 N_r(2r)^n\le C(\Theta_0).
\]
Letting $r\downarrow0$ gives the conclusion.
\end{proof}

\paragraph*{Acknowledgments.}
The author thanks Wenshuai Jiang and Jun Geng for suggesting this problem
and for their support.

\appendix
\section{Gaussian doubling for caloric functions}\label{sec:gaussian-appendix}

We record Gaussian doubling monotonicity and rigidity, the Gaussian
degree decomposition, and the large-scale degree formula for caloric
polynomials.

The frequency calculation in the proof of Lemma~\ref{lem:caloric-gaussian-doubling}
is due to Poon~\cite{Poon}; we reproduce it here for the reader's convenience.

\begin{lemma}\label{lem:caloric-gaussian-doubling}
Let $u\not\equiv0$ be a smooth caloric function on $\R^n\times(-\infty,0]$.
Assume that for every compact interval $I\subset(-\infty,0]$ and every $a>0$,
\[
|u(x,t)|\le C_{I,a}e^{a|x|^2}\qquad(x\in\R^n,\ t\in I).
\]
Then $D(u;r)$ is nondecreasing for $r>0$. If
$D(u;r)=D(u;R)$ for some $0<r<R$, then $u$ is a homogeneous
caloric polynomial, and this common value is its parabolic degree in $\mathbb N_0$.
\end{lemma}
\begin{proof}
Write $H(r)=H(u;\mathbf0,r)$, $E(r)=E(u;\mathbf0,r)$,
$N(r)=N(u;\mathbf0,r)$ and $D(r)=D(u;\mathbf0,r)$.
Set $\dd\gamma=\dd\nu_{0,1}$ and $v_r(y)=u(ry,-r^2)$. Then
\[
\begin{gathered}
r\partial_rv_r=y\cdot\nabla v_r-2\Delta v_r,\\
H(r)=\int_{\R^n}v_r^2\dd\gamma,\qquad
E(r)=2\int_{\R^n}|\nabla v_r|^2\dd\gamma.
\end{gathered}
\]
The growth assumption and interior estimates justify differentiation
and Gaussian integration by parts. Since
$\nabla G_{0,1}=-\frac y2G_{0,1}$, we obtain
\[
\begin{aligned}
H'(r)
&=2\int_{\R^n}v_r\partial_rv_r\dd\gamma
=\frac4r\int_{\R^n}|\nabla v_r|^2\dd\gamma
=\frac{2E(r)}r,\\
E'(r)
&=4\int_{\R^n}\nabla v_r\cdot\nabla\partial_rv_r\dd\gamma\\
&=2\int_{\R^n}(y\cdot\nabla v_r-2\Delta v_r)\partial_rv_r\dd\gamma
=2r\int_{\R^n}|\partial_rv_r|^2\dd\gamma.
\end{aligned}
\]
On any interval where $H>0$, this gives
\[
N'(r)=\frac{E'(r)}{H(r)}-\frac{E(r)H'(r)}{H(r)^2}
=\frac{2r}{H(r)}\int_{\R^n}
\left|\partial_rv_r-\frac{N(r)}r v_r\right|^2\dd\gamma\ge0.
\]
Fix $r_0$ with $H(r_0)>0$. Monotonicity of $N$ implies
\[
H(r)\ge H(r_0)\left(\frac r{r_0}\right)^{2N(r_0)}
\qquad(0<r\le r_0,\ H(r)>0).
\]
Together with $H'\ge0$ and continuity, this proves $H(r)>0$ for every $r>0$.
Consequently,
\[
D(r)=\frac1{\log2}\int_{r/2}^rN(s)\frac{\dd s}{s},\qquad
D'(r)=\frac{N(r)-N(r/2)}{r\log2}\ge0.
\]

If $D(r)=D(R)$ for $r<R$, then $N(s)=N(s/2)$ for $r\le s\le R$.
Monotonicity makes $N\equiv m$ on $[r/2,R]$. The identity for $N'$
therefore gives, at any $r_0\in(r/2,R)$,
\[
y\cdot\nabla v_{r_0}-2\Delta v_{r_0}=mv_{r_0}.
\]
The spectrum of $-2\Delta+y\cdot\nabla$ in $L^2(\gamma)$ is
$\mathbb N_0$, and its eigenspaces are spanned by Hermite polynomials.
Thus $m\in\mathbb N_0$, and there is a caloric polynomial $P$,
homogeneous of parabolic degree $m$, such that
$P(r_0y,-r_0^2)=v_{r_0}(y)$.
If $u-P\not\equiv0$, the preceding argument gives
$H(u-P;r_0)>0$, contrary to $(u-P)(r_0\,\cdot,-r_0^2)=0$.
Thus $u=P$, and $D(u;s)=m$ for every $s>0$.
\end{proof}

\begin{lemma}\label{lem:gaussian-degrees}
Let $v\not\equiv0$ be a caloric polynomial of parabolic degree $K$.
Then
\begin{equation}\label{eq:caloric-degree-decomposition}
v=\sum_{k=0}^{K}P_k,\qquad
(\partial_t-\Delta)P_k=0,\qquad
P_k(rx,r^2t)=r^kP_k(x,t).
\end{equation}
Set
$\alpha_k:=\int_{\R^n}|P_k(y,-1)|^2\dd\nu_{0,1}(y)$. For $r>0$,
\begin{equation}\label{eq:gaussian-degree-identities}
H(v;\mathbf0,r)=\sum_{k=0}^{K}\alpha_kr^{2k},\qquad
N(v;\mathbf0,r)=
\frac{\sum_{k=0}^{K}k \alpha_kr^{2k}}{\sum_{k=0}^{K}\alpha_kr^{2k}},
\qquad \alpha_k\ge0.
\end{equation}
Moreover,
\begin{equation}\label{eq:lowest-caloric-degrees}
P_0=v(\mathbf0),\qquad P_1(x,t)=\nabla_xv(\mathbf0)\cdot x.
\end{equation}
\end{lemma}

\begin{proof}
Grouping the monomials of $v$ by parabolic degree yields
\eqref{eq:caloric-degree-decomposition}, since $\partial_t-\Delta$
lowers that degree by two.

Set $f_k(y)=P_k(y,-1)$ and $\mathcal L=\Delta-\frac12y\cdot\nabla$.
Parabolic homogeneity and the heat equation give
\[
y\cdot\nabla P_k+2t\partial_tP_k=kP_k,
\qquad -2\mathcal Lf_k=kf_k.
\]
Since $\nabla G_{0,1}=-\frac y2G_{0,1}$, integration by parts yields
\[
\frac k2\int_{\R^n}f_if_k\dd\nu_{0,1}
=\int_{\R^n}\nabla f_i\cdot\nabla f_k\dd\nu_{0,1}
=\frac i2\int_{\R^n}f_if_k\dd\nu_{0,1}.
\]
Thus, for $i\ne k$,
\[
\int_{\R^n}f_if_k\dd\nu_{0,1}
=\int_{\R^n}\nabla f_i\cdot\nabla f_k\dd\nu_{0,1}=0,
\]
while
\[
2\int_{\R^n}|\nabla f_k|^2\dd\nu_{0,1}=k\alpha_k.
\]
By parabolic scaling,
\[
v(ry,-r^2)=\sum_{k=0}^K r^kf_k(y),\qquad
r\nabla_xv(ry,-r^2)=\sum_{k=0}^K r^k\nabla f_k(y).
\]
Changing variables $x=ry$ and using these orthogonality relations,
we obtain
\[
\begin{aligned}
H(v;\mathbf0,r)
&=\int_{\R^n}\left|\sum_{k=0}^K r^kf_k\right|^2\dd\nu_{0,1}
=\sum_{k=0}^K \alpha_kr^{2k},\\
E(v;\mathbf0,r)
&=2\int_{\R^n}\left|\sum_{k=0}^K r^k\nabla f_k\right|^2\dd\nu_{0,1}
=\sum_{k=0}^K k\alpha_kr^{2k}.
\end{aligned}
\]
Dividing the second identity by the first proves
\eqref{eq:gaussian-degree-identities}.
The terms of degrees zero and one give
\eqref{eq:lowest-caloric-degrees}.
\end{proof}

\Needspace{12\baselineskip}
\begin{lemma}\label{lem:polynomials}
Let $P\not\equiv0$ be a caloric polynomial of parabolic degree $d$. Then
\begin{equation}\label{eq:degree-at-infinity}
\lim_{r\to\infty}D(P;r)=d.
\end{equation}
\end{lemma}
\begin{proof}
By \eqref{eq:gaussian-degree-identities} with $v=P$ and $K=d$,
there are $\alpha_k\ge0$ with $\alpha_d>0$ such that
\[
H(P;r)=\sum_{k=0}^d\alpha_kr^{2k}=\alpha_dr^{2d}(1+o(1))
\qquad(r\to\infty).
\]
Taking the ratio at $r$ and $r/2$ proves \eqref{eq:degree-at-infinity}.
\end{proof}

\section{A counterexample without (A5)}\label{sec:counterexample}

We adapt the observation of Lin--Shen \cite[Remark~2.9]{LS} to a
singular set by using the reflection-symmetric conductivity of
Briane--Milton--Nesi \cite[Section~5]{BMN}. The additional point is to
place many periodic copies of a critical point on the same level set.

\subsection{Periodic conducting rings}

Let $Y=(-1,1)^3$ and define two solid tori by
\begin{equation}\label{eq:counterexample-tori}
\begin{aligned}
T_1&=\left\{y\in\R^3:
 y_1^2+\left(\sqrt{y_2^2+y_3^2}-\frac34\right)^2<\frac1{400}\right\},\\
T_2&=\left\{y\in\R^3:
 y_2^2+\left(\sqrt{y_1^2+(y_3-1)^2}-\frac34\right)^2<\frac1{400}\right\},\\
\mathcal T&=\bigcup_{k\in\mathbb Z^3}
 \bigl[(T_1+2k)\cup(T_2+2k)\bigr],\qquad T_-=T_2-2e_3.
\end{aligned}
\end{equation}
The major radius is $3/4$ and the tube radius is $1/20$.
Distinct tori have disjoint closures. They form parallel periodic chains
of linked rings, with alternating axes parallel to $e_1$ and $e_2$.
This is a concrete choice of the geometry in \cite[Section~5.1]{BMN};
the numerical radii are chosen here for definiteness.
The set $\mathcal T$ is invariant under each coordinate reflection
$R_jy=y-2y_je_j$.

\begin{figure}[htbp]
\centering
\includegraphics[width=\textwidth]{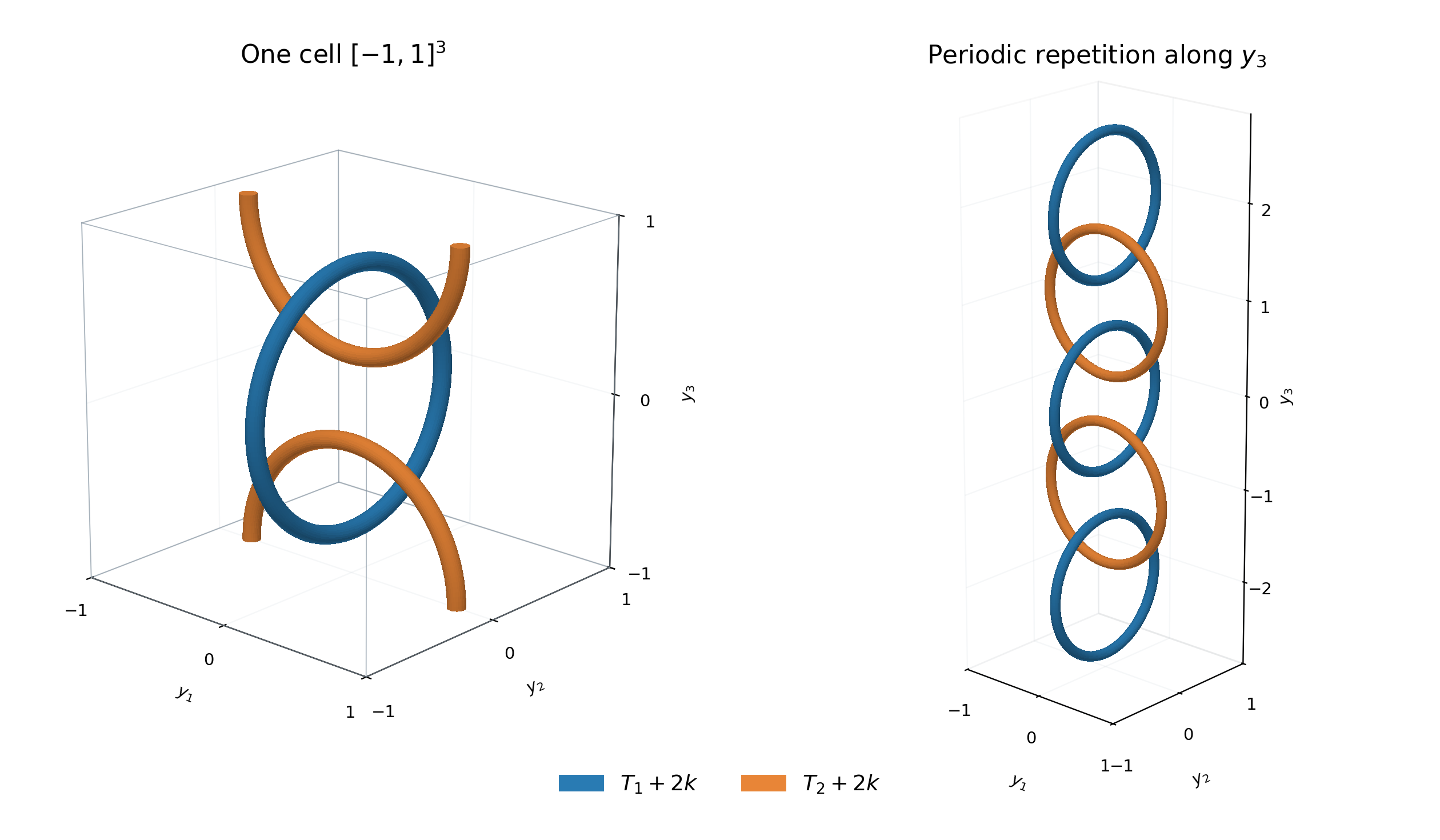}
\caption{The conducting rings in \eqref{eq:counterexample-tori}.
The left panel shows $\mathcal T\cap[-1,1]^3$: the complete ring
$T_1$ and the portions of $T_2$ and $T_-$ inside the cell.
Blue and orange indicate translates of $T_1$ and $T_2$, respectively.
The right panel shows one vertical chain; its translates by
$2k_1e_1+2k_2e_2$ give the full periodic structure.
The rings are linked but do not touch.}
\label{fig:conducting-rings}
\end{figure}

For $K>1$, set
\begin{equation}\label{eq:counterexample-coefficient}
 a_K=1+(K-1)\mathbf1_{\mathcal T}.
\end{equation}
Let $V_K=y_3+\chi_K$, where $\chi_K$ is the unique zero-mean
$2\mathbb Z^3$-periodic weak solution of
\[
 \operatorname{div}\bigl(a_K(e_3+\nabla\chi_K)\bigr)=0.
\]
Thus $\operatorname{div}(a_K\nabla V_K)=0$ and
\begin{equation}\label{eq:counterexample-symmetries}
\begin{gathered}
 V_K(y+2k)=V_K(y)+2k_3,\\
 V_K(R_1y)=V_K(y),\qquad V_K(R_2y)=V_K(y),\qquad
 V_K(R_3y)=-V_K(y).
\end{gathered}
\end{equation}
Indeed, $V_K\circ R_1$, $V_K\circ R_2$, and $-V_K\circ R_3$
solve the same equation as $V_K$ and have the same affine part and
zero-mean periodic correction. Uniqueness gives the reflection identities.

\subsection{A critical point on the symmetry axis}

We recall the perfect-conductor argument in \cite[Section~5.1, Proposition~1]{BMN}.
There is a smooth function $W$ such that $W-y_3$ is periodic and
$W$ is constant on each component of $\mathcal T$: assign the values
$2k_3$ on $T_1+2k$ and $1+2k_3$ on $T_2+2k$, and interpolate in
the gaps. The positive separation of the rings permits smooth cutoffs
with disjoint supports. Subtracting the mean of $W-y_3$ makes $W$
admissible in the corrector variational problem without changing its gradient.
Minimality gives
\begin{equation}\label{eq:counterexample-energy}
 \int_Y a_K|\nabla V_K|^2
 \le\int_Y a_K|\nabla W|^2
 =\int_{Y\setminus\mathcal T}|\nabla W|^2\le C.
\end{equation}
The periodic Poincar\'e inequality bounds $\chi_K$ in $H^1(Y)$.
Along a subsequence, $V_K\rightharpoonup V_\infty$ locally in $H^1$.
By \eqref{eq:counterexample-energy}, $\nabla V_\infty=0$ in each ring;
in the complement of the closed rings, $V_K$ and $V_\infty$ are harmonic,
and convergence holds locally in $C^1$. The identities
\eqref{eq:counterexample-symmetries} pass to the limit.

Write $c_0,c_+,c_-$ for the constant values of $V_\infty$ on
$T_1,T_2,T_-$. Since $R_3T_1=T_1$, $R_3T_2=T_-$, and
$T_2=T_-+2e_3$, respectively,
\[
 c_0=-c_0,\qquad c_-=-c_+,\qquad c_+=c_-+2.
\]
Consequently,
\begin{equation}\label{eq:counterexample-ring-values}
 V_\infty|_{T_1}=0,\qquad
 V_\infty|_{T_2}=1,\qquad V_\infty|_{T_-}=-1.
\end{equation}
Since $V_\infty\in H^1_{\mathrm{loc}}$, its interior and exterior
Sobolev traces agree on each ring boundary. Local boundary regularity
for harmonic functions with these constant Dirichlet data therefore
extends the values in \eqref{eq:counterexample-ring-values} continuously
to the ring boundaries.
For $f(s)=V_\infty(0,0,s)$ and
\[
 a_2=\frac15,\qquad b_2=\frac3{10},\qquad
 a_1=\frac7{10},\qquad b_1=\frac45,
\]
the intersections of the closed rings with the axis in $Y$
(see Figure~\ref{fig:axis-intersections}) are
$[-b_1,-a_1]\cup[a_1,b_1]$ for $T_1$, $[a_2,b_2]$ for $T_2$,
and $[-b_2,-a_2]$ for $T_-$.

\begin{figure}[htbp]
\centering
\includegraphics[width=0.82\textwidth]{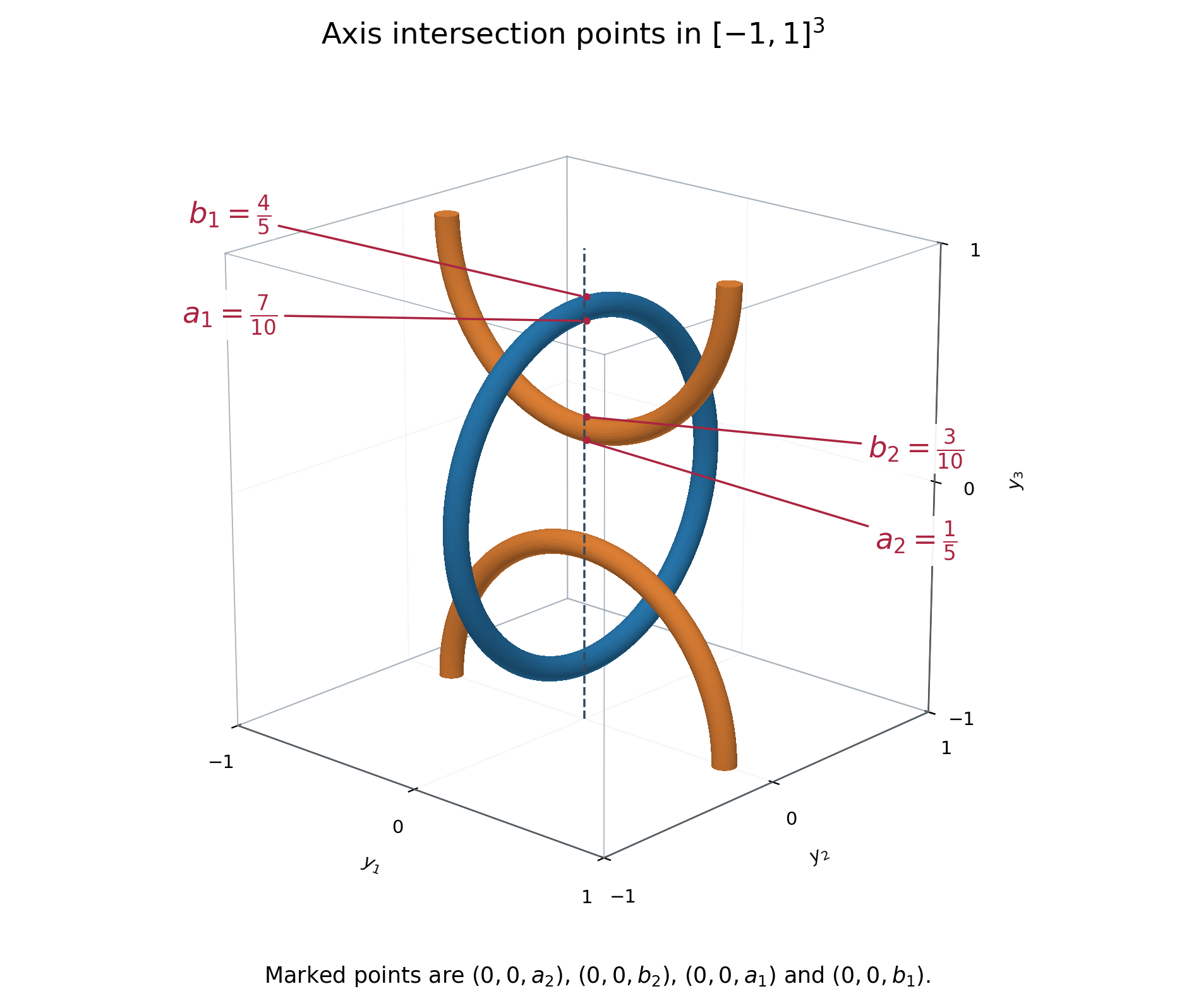}
\caption{The conducting region in the cell $[-1,1]^3$, with the four
points $(0,0,a_2)$, $(0,0,b_2)$, $(0,0,a_1)$, and $(0,0,b_1)$
marked on the $y_3$-axis.}
\label{fig:axis-intersections}
\end{figure}

Hence
\[
 f(-a_2)=-1,\qquad f(a_2)=f(b_2)=1,\qquad f(a_1)=0.
\]
The two intervening open intervals lie outside the rings, so the mean
value theorem gives points
\begin{equation}\label{eq:counterexample-opposite-signs}
 s_+\in(-a_2,a_2),\quad s_-\in(b_2,a_1),\qquad
 \partial_3V_\infty(0,0,s_+)>0>\partial_3V_\infty(0,0,s_-).
\end{equation}
Local $C^1$ convergence preserves these strict inequalities for a
sufficiently large finite $K$, which we now fix.

To obtain a smooth coefficient, choose a nonnegative radial mollifier
$\varrho\in C_c^\infty(B_1)$ with integral one, put
$\varrho_h(y)=h^{-3}\varrho(y/h)$, and set
\begin{equation}\label{eq:counterexample-smooth-coefficient}
 a_{K,h}=\varrho_h*a_K
 =1+(K-1)(\varrho_h*\mathbf1_{\mathcal T}).
\end{equation}
The coefficient is smooth, $2\mathbb Z^3$-periodic, reflection invariant,
and satisfies $1\le a_{K,h}\le K$; see also \cite[Section~2]{Cap15}.
Let $V_{K,h}=y_3+\chi_{K,h}$ be its harmonic coordinate.
Testing the difference of the two corrector equations yields
\[
 \|\nabla(\chi_{K,h}-\chi_K)\|_{L^2(Y)}
 \le\|(a_{K,h}-a_K)\nabla V_K\|_{L^2(Y)}\longrightarrow0
 \qquad(h\downarrow0).
\]
Near the two points in \eqref{eq:counterexample-opposite-signs},
$a_{K,h}=a_K=1$ for small $h$, so harmonic interior estimates give
local $C^1$ convergence there. Fix $h>0$ small enough that the two
strict signs persist, and write $a=a_{K,h}$ and $V=V_{K,h}$.
Both $K$ and $h$ remain fixed as $\eps\downarrow0$.
The reflection identities give
$\partial_1V(0,0,s)=\partial_2V(0,0,s)=0$.
Continuity of $\partial_3V$ therefore gives
\begin{equation}\label{eq:counterexample-critical-coordinate}
 y_0=(0,0,s_0),\qquad s_+<s_0<s_-,\qquad
 \nabla V(y_0)=\mathcal J_{aI}(y_0)e_3=0.
\end{equation}
Thus \textup{(A5)} fails.

\subsection{Failure of a uniform singular-set Minkowski bound}

Define the stationary solutions
\begin{equation}\label{eq:counterexample-solutions}
 u_\eps(x,t)=\eps\bigl[V(x/\eps)-V(y_0)\bigr].
\end{equation}
They satisfy
\[
 \partial_tu_\eps-\operatorname{div}_x
 \bigl(a(x/\eps)\nabla_xu_\eps\bigr)=0,
 \qquad
 \nabla_xu_\eps(x,t)=\nabla V(x/\eps).
\]
Since $V-y_3$ is bounded and periodic,
$u_\eps\to x_3$ locally uniformly. Thus the coarse-scale ratio in
\eqref{eq:rough} is bounded independently of $\eps$.
By \eqref{eq:counterexample-symmetries} and
\eqref{eq:counterexample-critical-coordinate},
\begin{equation}\label{eq:counterexample-singular-lattice}
 \bigl\{\eps(y_0+2k_1e_1+2k_2e_2):k_1,k_2\in\mathbb Z\bigr\}
 \times\R\subset S(u_\eps).
\end{equation}
For small $\eps$, at least $c\eps^{-2}$ such spatial points lie in
$B_{1/4}$, with separation $2\eps$. Taking their disjoint spatial
$r$-balls and the fixed time interval $(-1/8,1/8)$ gives
\begin{equation}\label{eq:counterexample-tube-lower}
 \bigl|Q(S(u_\eps)\cap Q(\mathbf0,1/2),r)
             \cap Q(\mathbf0,1/2)\bigr|
 \ge c\eps^{-2}r^3,
 \qquad 0<r\le c_0\eps,
\end{equation}
where $c_0>0$ is fixed and small. For $r=c_0\eps$, the ratio of
this volume to $r^2$ tends to infinity. Since parabolic
$\mathcal H_P^5$ is a dimensional constant times space-time Lebesgue
measure, the uniform bound in Theorem~\ref{thm:main-minkowski} fails.

The coefficient $aI$ satisfies \textup{(A1)--(A3)}, with all data fixed,
and violates \textup{(A5)} by
\eqref{eq:counterexample-critical-coordinate}. To impose
\textup{(A4)}, let $B=\widehat A^{1/2}$ for $A=aI$ and change variables
$x=B\widetilde x$. The new coefficient is
\[
 \widetilde A(y)=B^{-1}A(By)B^{-T},\qquad
 \widehat{\widetilde A}=I,
\]
with period lattice $B^{-1}(2\mathbb Z^3)$.
This fixed invertible change preserves gradient vanishing and the
failure of the tube bound, up to fixed constants. The transformed
solutions converge to the nonzero linear function $(B\widetilde x)_3$,
so their coarse-scale ratios remain bounded as well. Corrector
nondegeneracy still fails, since the transformed harmonic-coordinate
matrix is $B^T\mathcal J_A(By)B^{-T}$ and has the same rank.

For comparison, the construction in \cite[Remark~2.9]{LS} starts with
$v_\xi(y)=\xi\cdot(y+\chi(y))$ and $\nabla v_\xi(y_0)=0$.
Subtracting $v_\xi(y_0)$ gives, for every period vector $\ell$,
\[
 \nabla v_\xi(y_0+\ell)=0,\qquad
 v_\xi(y_0+\ell)-v_\xi(y_0)=\xi\cdot\ell.
\]
Only the vectors in $\mathcal L\cap\xi^\perp$ therefore give
singular points by this replication argument. In the present
construction, $\xi=e_3$ and $\mathcal L=2\mathbb Z^3$, so this
intersection contains the rank-two lattice used in
\eqref{eq:counterexample-singular-lattice}.
The argument proves failure of a uniform Minkowski estimate; it does
not prove failure of the Hausdorff estimate. Each time line in
\eqref{eq:counterexample-singular-lattice} has parabolic dimension two
and zero $\mathcal H_P^3$ measure.


\begin{thebibliography}{99}

\bibitem{ABG}
V.~Arya, A.~Banerjee, and N.~Garofalo,
\emph{\href{https://doi.org/10.1007/s00205-025-02139-3}{Sharp order of vanishing for parabolic equations, nodal set estimates and Landis type results}},
Arch. Ration. Mech. Anal. \textbf{249} (2025), no.~6, Paper No.~67.

\bibitem{AL87}
M.~Avellaneda and F.-H.~Lin,
\emph{\href{https://doi.org/10.1002/cpa.3160400607}{Compactness methods in the theory of homogenization}},
Comm. Pure Appl. Math. \textbf{40} (1987), no.~6, 803--847.

\bibitem{BLP}
A.~Bensoussan, J.-L.~Lions, and G.~Papanicolaou,
\emph{\href{https://www.sciencedirect.com/bookseries/studies-in-mathematics-and-its-applications/vol/5/suppl/C}{Asymptotic Analysis for Periodic Structures}},
Studies in Mathematics and its Applications, vol.~5,
North-Holland, Amsterdam, 1978.

\bibitem{BMN}
M.~Briane, G.~W.~Milton, and V.~Nesi,
\emph{\href{https://doi.org/10.1007/s00205-004-0315-8}{Change of sign of the corrector's determinant for homogenization in three-dimensional conductivity}},
Arch. Ration. Mech. Anal. \textbf{173} (2004), no.~1, 133--150.

\bibitem{Cap15}
Y.~Capdeboscq,
\emph{\href{https://doi.org/10.5802/jep.21}{On a counter-example to quantitative Jacobian bounds}},
J. \'Ec. polytech. Math. \textbf{2} (2015), 171--178.

\bibitem{CNV}
J.~Cheeger, A.~Naber, and D.~Valtorta,
\emph{\href{https://doi.org/10.1002/cpa.21518}{Critical sets of elliptic equations}},
Comm. Pure Appl. Math. \textbf{68} (2015), no.~2, 173--209.

\bibitem{DF}
H.~Donnelly and C.~Fefferman,
\emph{\href{https://doi.org/10.1007/BF01393691}{Nodal sets of eigenfunctions on Riemannian manifolds}},
Invent. Math. \textbf{93} (1988), no.~1, 161--183.

\bibitem{EFV}
L.~Escauriaza, F.~J.~Fern\'andez, and S.~Vessella,
\emph{\href{https://doi.org/10.1080/00036810500277082}{Doubling properties of caloric functions}},
Appl. Anal. \textbf{85} (2006), no.~1--3, 205--223.

\bibitem{GL86}
N.~Garofalo and F.-H.~Lin,
\emph{\href{https://doi.org/10.1512/iumj.1986.35.35015}{Monotonicity properties of variational integrals, $A_p$ weights and unique continuation}},
Indiana Univ. Math. J. \textbf{35} (1986), no.~2, 245--268.

\bibitem{GL87}
N.~Garofalo and F.-H.~Lin,
\emph{\href{https://doi.org/10.1002/cpa.3160400305}{Unique continuation for elliptic operators: a geometric-variational approach}},
Comm. Pure Appl. Math. \textbf{40} (1987), no.~3, 347--366.

\bibitem{GS}
J.~Geng and Z.~Shen,
\emph{\href{https://arxiv.org/abs/1308.5726}{Uniform regularity estimates in parabolic homogenization}},
Indiana Univ. Math. J. \textbf{64} (2015), no.~3, 697--733.

\bibitem{GS17}
J.~Geng and Z.~Shen,
\emph{\href{https://doi.org/10.1016/j.jfa.2016.10.005}{Convergence rates in parabolic homogenization with time-dependent periodic coefficients}},
J. Funct. Anal. \textbf{272} (2017), no.~5, 2092--2113.

\bibitem{GS20}
J.~Geng and Z.~Shen,
\emph{\href{https://doi.org/10.2140/apde.2020.13.147}{Asymptotic expansions of fundamental solutions in parabolic homogenization}},
Anal. PDE \textbf{13} (2020), no.~1, 147--170.

\bibitem{HKM}
M.~Hallgren, R.~Koirala, and Z.~Ma,
\emph{Structure theory of parabolic nodal and singular sets},
\href{https://arxiv.org/abs/2511.11570}{arXiv:2511.11570}, 2025.

\bibitem{Han94}
Q.~Han,
\emph{\href{https://doi.org/10.1512/iumj.1994.43.43043}{Singular sets of solutions to elliptic equations}},
Indiana Univ. Math. J. \textbf{43} (1994), no.~3, 983--1002.

\bibitem{HHL}
Q.~Han, R.~Hardt, and F.-H.~Lin,
\emph{Geometric measure of singular sets of elliptic equations},
Comm. Pure Appl. Math. \textbf{51} (1998), no.~11--12, 1425--1443.

\bibitem{HL94}
Q.~Han and F.-H.~Lin,
\emph{\href{https://doi.org/10.1002/cpa.3160470904}{Nodal sets of solutions of parabolic equations. II}},
Comm. Pure Appl. Math. \textbf{47} (1994), no.~9, 1219--1238.

\bibitem{HHHON}
R.~Hardt, M.~Hoffmann-Ostenhof, T.~Hoffmann-Ostenhof, and N.~Nadirashvili,
\emph{\href{https://doi.org/10.4310/jdg/1214425070}{Critical sets of solutions to elliptic equations}},
J. Differential Geom. \textbf{51} (1999), no.~2, 359--373.

\bibitem{HS}
R.~Hardt and L.~Simon,
\emph{\href{https://doi.org/10.4310/jdg/1214443599}{Nodal sets for solutions of elliptic equations}},
J. Differential Geom. \textbf{30} (1989), no.~2, 505--522.

\bibitem{HJ}
Y.~Huang and W.~Jiang,
\emph{Volume estimates for singular sets and critical sets of elliptic
equations with H\"older coefficients},
\href{https://arxiv.org/abs/2309.08089}{arXiv:2309.08089}, 2023.

\bibitem{HJ24}
Y.~Huang and W.~Jiang,
\emph{The nodal sets of solutions to parabolic equations},
\href{https://arxiv.org/abs/2406.05877}{arXiv:2406.05877}, 2024.

\bibitem{JKO}
V.~V.~Jikov, S.~M.~Kozlov, and O.~A.~Oleinik,
\emph{\href{https://doi.org/10.1007/978-3-642-84659-5}{Homogenization of Differential Operators and Integral Functionals}},
Springer-Verlag, Berlin, 1994.

\bibitem{KZZ}
C.~E.~Kenig, J.~Zhu, and J.~Zhuge,
\emph{\href{https://doi.org/10.1080/03605302.2021.1989699}{Doubling inequalities and nodal sets in periodic elliptic homogenization}},
Comm. Partial Differential Equations \textbf{47} (2022), no.~3, 549--584.

\bibitem{Kuk95}
I.~Kukavica,
\emph{\href{https://doi.org/10.1155/S1073792895000389}{Hausdorff measure of level sets for solutions of parabolic equations}},
Internat. Math. Res. Notices (1995), no.~13, 671--682.

\bibitem{Lin91}
F.-H.~Lin,
\emph{\href{https://doi.org/10.1002/cpa.3160440303}{Nodal sets of solutions of elliptic and parabolic equations}},
Comm. Pure Appl. Math. \textbf{44} (1991), no.~3, 287--308.

\bibitem{LS19}
F.~Lin and Z.~Shen,
\emph{\href{https://doi.org/10.1007/s10114-019-8228-5}{Nodal sets and doubling conditions in elliptic homogenization}},
Acta Math. Sin. (Engl. Ser.) \textbf{35} (2019), no.~6, 815--831.

\bibitem{LS}
F.~Lin and Z.~Shen,
\emph{\href{https://arxiv.org/abs/2203.13393}{Critical sets of solutions of elliptic equations in periodic homogenization}},
Comm. Pure Appl. Math. \textbf{77} (2024), 3143--3183
(published online in 2023).

\bibitem{LogUpper}
A.~Logunov,
\emph{\href{https://doi.org/10.4007/annals.2018.187.1.4}{Nodal sets of Laplace eigenfunctions: polynomial upper estimates of the Hausdorff measure}},
Ann. of Math. (2) \textbf{187} (2018), no.~1, 221--239.

\bibitem{LogLower}
A.~Logunov,
\emph{\href{https://doi.org/10.4007/annals.2018.187.1.5}{Nodal sets of Laplace eigenfunctions: proof of Nadirashvili's conjecture and of the lower bound in Yau's conjecture}},
Ann. of Math. (2) \textbf{187} (2018), no.~1, 241--262.

\bibitem{NV}
A.~Naber and D.~Valtorta,
\emph{\href{https://doi.org/10.1002/cpa.21708}{Volume estimates on the critical sets of solutions to elliptic PDEs}},
Comm. Pure Appl. Math. \textbf{70} (2017), no.~10, 1835--1897.

\bibitem{Poon}
C.-C.~Poon,
\emph{\href{https://doi.org/10.1080/03605309608821195}{Unique continuation for parabolic equations}},
Comm. Partial Differential Equations \textbf{21} (1996), no.~3--4,
521--539.

\bibitem{Shen18}
Z.~Shen,
\emph{\href{https://doi.org/10.1007/978-3-319-91214-1}{Periodic Homogenization of Elliptic Systems}},
Operator Theory: Advances and Applications, vol.~269,
Birkh\"auser/Springer, Cham, 2018.

\bibitem{Zhang21}
Y.~Zhang,
\emph{\href{https://doi.org/10.1137/20M1345323}{Approximate two-sphere
one-cylinder inequality in parabolic periodic homogenization}},
SIAM J. Math. Anal. \textbf{53} (2021), no.~5, 5835--5852.

\bibitem{Zhang25}
Y.~Zhang,
\emph{\href{https://doi.org/10.1007/s10231-025-01544-5}{Approximate
two-sphere one-cylinder inequality in parabolic periodic homogenization
with suitable lower-order terms}},
Ann. Mat. Pura Appl. (4) \textbf{204} (2025), 1689--1713.

\end{thebibliography}
\end{document}